\documentclass{rmmcart}
\usepackage{fullpage}
\usepackage{graphicx}
\usepackage[utf8]{inputenc}
\usepackage[english]{babel}
\usepackage{tikz}
\usepackage{amsmath,amssymb,amsthm,dsfont,enumerate,url,appendix,tabularx,amsfonts}
\usepackage{todonotes, comment}
\usepackage{stmaryrd}
\usepackage{caption}
\usepackage{subcaption}
\theoremstyle{plain}
\numberwithin{equation}{section}
\numberwithin{figure}{section}
\numberwithin{table}{section}

\newtheorem{theorem}{Theorem}[section]
\newtheorem{lemma}[theorem]{Lemma}
\newtheorem{corollary}[theorem]{Corollary}
\newtheorem{proposition}[theorem]{Proposition}

\newtheorem{assumption}[theorem]{Assumption}
\newtheorem{problem}[theorem]{Problem}
\newtheorem{notation}[theorem]{Notation}

\theoremstyle{definition}
\newtheorem{definition}[theorem]{Definition}
\newtheorem{example}[theorem]{Example}

\theoremstyle{remark}
\newtheorem{remark}[theorem]{Remark}

\newcommand{\ii}{ \mathrm{i} }
\newcommand{\laplace}{\Delta}
\newcommand{\vect}[1]{ #1 }
\newcommand{\norm}[1]{\left\|#1\right\|}
\newcommand{\triplenorm}[1]{{\left\vert\kern-0.25ex\left\vert\kern-0.25ex\left\vert #1 
    \right\vert\kern-0.25ex\right\vert\kern-0.25ex\right\vert}}
\newcommand{\abs}[1]{\left|#1\right|}

\newcommand{\R}{\mathbb{R}}
\newcommand{\C}{\mathbb{C}}
\newcommand{\N}{\mathbb{N}}

\newcommand{\matH}{\prodOp{\mathbf{H}}}

\newcommand{\prodSpace}[1]{\underline{#1}}
\newcommand{\prodOp}[1]{\underline{#1}}
\newcommand{\colvec}[1]{\begin{pmatrix}#1\end{pmatrix}}

\newcommand{\ones}{\mathds{1}}

\newcommand{\fdU}[1]{\vect{U}_h^{#1}}

\newcommand{\fdu}[1]{u_{h}^{#1}}
\newcommand{\fdl}[1]{\lambda_h^{#1}}

\newcommand{\fdUstar}[1]{\tilde{\vect{U}}_*^{#1}}

\newcommand{\fdustar}[1]{\tilde{u}_{*}^{#1}}

\newcommand{\sdU}[1]{\vect{U}^{#1}}
\newcommand{\sdu}[1]{u^{#1}}
\newcommand{\sdl}[1]{\lambda^{#1}}

\newcommand{\sdUstar}[1]{\vect{U}_*^{#1}}

\newcommand{\sdustar}[1]{u_{*}^{#1}}

\newcommand{\sdX}[1]{\vect{X}^{#1}}
\newcommand{\fdX}[1]{\vect{X}^{#1}_h}

\newcommand{\sdx}[1]{x^{#1}}
\newcommand{\fdx}[1]{x^{#1}_h}

\newcommand{\innerprod}[3]{\left(#2,#3\right)_{#1}}
\newcommand{\ltwoprodgen}[3]{\innerprod{L^2\left(#1\right)}{#2}{#3}}
\newcommand{\ltwoprodgenp}[3]{\innerprod{\underline{L^2}\left(#1\right)}{#2}{#3}}
\newcommand{\ltwoprodint}[2]{\ltwoprodgen{\Omega}{#1}{#2}}
\newcommand{\ltwoprodintp}[2]{\ltwoprodgenp{\Omega}{#1}{#2}}
\newcommand{\ltwoprodfull}[2]{\ltwoprodgen{\R^d \setminus \Gamma}{#1}{#2}}
\newcommand{\ltwoprodfullp}[2]{\ltwoprodgenp{\R^d \setminus \Gamma}{#1}{#2}}

\newcommand{\ltwonorm}[1]{\norm{#1}_{L^2(\R^d)}}
\newcommand{\ltwonormfull}[1]{\norm{#1}_{L^2(\R^d \setminus \Gamma)}}
\newcommand{\ltwonormint}[1]{\norm{#1}_{L^2(\Omega)}}

\newcommand{\ltwonormgen}[2]{\norm{#2}_{L^2(#1)}}

\newcommand{\honenorm}[1]{\norm{#1}_{H^1(\R^d)}}
\newcommand{\honenormint}[1]{\norm{#1}_{H^1(\Omega)}}
\newcommand{\honenormintp}[1]{\norm{#1}_{\prodSpace{H^1}(\Omega)}}
\newcommand{\hpnormint}[2]{\norm{#2}_{H^{#1}(\Omega)}}
\newcommand{\honenormfull}[1]{\norm{#1}_{H^1(\R^d \setminus \Gamma)}}
\newcommand{\hpnorm}[2]{\norm{#2}_{H^{#1}(\R^d)}}
\newcommand{\hpbdrynorm}[2]{\norm{#2}_{H^{#1}(\Gamma)}}
\newcommand{\hpbdrynormp}[2]{\norm{#2}_{\prodSpace{H^{#1}}(\Gamma)}}
\newcommand{\hppwbdrynorm}[2]{\norm{#2}_{H^{#1}_{pw}\left(\Gamma\right)}}
\newcommand{\bdryprod}[2]{\left<#1,#2\right>_{\Gamma}}
\newcommand{\bdryprodp}[2]{\left<#1,#2\right>_{\underline{\Gamma}}}

\newcommand{\dualprod}[3]{\left<#2,#3\right>_{{#1}' \times {#1}}}

\newcommand{\traceJump}[1]{\left\llbracket\gamma #1 \right\rrbracket}
\newcommand{\normalJump}[1]{\left\llbracket \partial_n #1 \right\rrbracket }

\newcommand{\blfA}[4]{\mathcal{A}_{#1,#2}\left(#3,#4\right)}
\newcommand{\blfAInt}[2]{\blfA{\Omega}{\mathcal{V}}{#1}{#2}}

\newcommand{\blfAFull}[2]{\blfA{\R^d \setminus \Gamma}{\mathcal{V}_0}{#1}{#2}}
\newcommand{\blfHStab}[2]{\widetilde{\mathbf{H}}\left(#1,#2\right)}
\newcommand{\blfHStabP}[2]{\underline{\widetilde{\mathbf{H}}}\left(#1,#2\right)}

\newcommand{\blfFp}[4]{ \innerprod{\prodSpace{\pairltwo}}{\colvec{#1 \\ #2} }{ \colvec{#3 \\#4}}}
\newcommand{\blfB}[4]{B\left( \colvec{#1 \\ #2} , \colvec{#3 \\#4} \right) }

\newcommand{\hpgen}[2]{H^{#1}\left(#2\right)}

\newcommand{\ltwogen}[1]{L^{2}\left(#1\right)}
\newcommand{\hpbdry}[1]{\hpgen{#1}{\Gamma}}
\newcommand{\hppwbdry}[1]{H^{#1}_{pw}\left(\Gamma\right)}
\newcommand{\prodHpBdry}[1]{\prodSpace{H^{#1}}\left(\Gamma\right)}
\newcommand{\prodHpInt}[1]{\prodSpace{H^{#1}}\left(\Omega\right)}
\newcommand{\prodHpFull}[1]{\prodSpace{H^{#1}}\left(\R^d \setminus \Gamma \right)}
\newcommand{\hpint}[1]{\hpgen{#1}{\Omega}}
\newcommand{\hpfull}[1]{\hpgen{#1}{\R^d \setminus \Gamma}}
\newcommand{\pairltwo}{\mathcal{X}^0}
\newcommand{\pairhone}{\mathcal{X}^1}

\def\vecrhs{\mathbf{d}}

\newcommand{\eex}{\hbox{}\hfill\rule{0.8ex}{0.8ex}}
\newcommand{\eremk}{\eex}

\iffalse
\usepackage{tikz}
\usetikzlibrary{shapes}
\usepackage{pgfplots}
\newcommand{\includeTikzOrEPS}[1]{\tikzexternalenable 
  \tikzsetnextfilename{#1}
  {\include{figures/#1}} \tikzexternaldisable}

\usetikzlibrary{external}
\else
\newcommand{\includeTikzOrEPS}[1]{\includegraphics{figures/#1}}
\fi

\title[FEM-BEM coupling for the Schrödinger equation]{Runge-Kutta convolution quadrature and FEM-BEM coupling for the time-dependent linear Schrödinger equation}
\author{Jens Markus Melenk}
\address{Technische Universit\"at Wien, Institut f\"ur Analysis und Scientific Computing, Wiedner Hauptstraße 8-10, A-1040 Vienna}
\email{melenk@tuwien.ac.at}
\author{Alexander Rieder}
\address{Technische Universit\"at Wien, Institut f\"ur Analysis und Scientific Computing, Wiedner Hauptstraße 8-10, A-1040 Vienna}
\email{alexander.rieder@tuwien.ac.at}

\date{\today}

\begin{document}
\maketitle
\begin{abstract}
  We propose a numerical scheme to solve the time-dependent linear Schrödinger equation. 
The discretization is carried out
  by combining a Runge-Kutta time stepping scheme with a finite element discretization in space.
  Since the Schrödinger equation is posed on the whole space $\R^d$ we combine the interior finite element 
  discretization with a convolution quadrature based boundary element discretization.
  In this paper we analyze the resulting fully discrete scheme in terms of stability and convergence rate.
  Numerical experiments confirm the theoretical findings.
\end{abstract}

\section{Introduction}
The Schrödinger equation is one of the main governing equations of quantum mechanics and as such has manifold applications in physics and engineering.
In its most common form, it is posed on the whole space of $\R^d$, making it difficult to discretize using 
standard finite element (FEM) or finite difference methods. Most numerical techniques rely on 
identifying a bounded computational domain on which a numerical methods such as the FEM is employed, and 
the unbounded exterior of the computational domains is accounted for by means of some (approximate) transparent boundary 
condition. A good recent survey is \cite{aabes_review}. A popular technique, which permits one to stay within
the FEM framework, is the  PML (perfectly matched layer) method, in which the computational domain is surrounded by 
a (thin) region that absorbs outgoing waves. Other techniques include the use of infinite elements or methods 
that approximate the exact or discrete boundary conditions. 
In the present paper, we also employ a FEM for the finite computational domain but account
for the unbounded complement by means of a boundary element method (BEM). 
Advantages of using a BEM based approach for the transparent boundary conditions include 
the great geometric flexibility, which allows one to choose non-convex computational domains, 
good stability (and energy conservation) properties, and the option to (cheaply) recover the exterior
solution by post processing. 

The method of the present article relies on a FEM-BEM coupling procedure. Two classical FEM-BEM coupling procedures 
are the \emph{symmetric coupling} introduced in 
\cite{costabel_symmetric_coupling} and \cite{han_a_new_fembem_coupling} and the 
\emph{Johnson-N\'ed\'elec coupling} \cite{jn_coupling}. In the present paper we will focus on 
the symmetric approach. 
%

Our treatment of the exterior domain also introduces non-local operators in time, specifically,
operator of convolution type in time.
Convolution quadrature (CQ) as a method to discretize convolution integrals or, more specifically, fractional derivatives
was introduced by Lubich in 1988 in the two papers \cite{lubich_cq_and_op_calc1,lubich_cq_and_op_calc2}. There, the CQ is based on multistep methods. 
Higher order convolution quadrature methods based on Runge-Kutta time stepping schemes were later introduced by Lubich and
Ostermann in \cite{lubich_ostermann_rk_cq}. Since then, the method has attracted significant interest as a technique to apply BEMs to time-dependent problems, not only for parabolic equations, for which
it was first conceived, but also for hyperbolic problems. 

A first numerical study of convolution quadrature for the Schrödinger equation, based on a coupling of finite elements and boundary elements was done by Schädle in \cite{schaedle_schroedinger2d},
where he observes optimal convergence rate in time when the 2D Schrödinger equation is discretized using convolution quadrature based on the trapezoidal rule in time
and a collocation BEM and both discretizations are combined using a one equation coupling (Johnson-Nédélec coupling).

The numerical analysis of hyperbolic convolution quadrature has mostly focused on the wave equation. Usually, the analysis is carried out in the Laplace domain as in 
\cite{rk_convq_wave,banjai_sauter_rapid_solution_wave_eqn,banjai_lubich_err_analysis_rkcq}. 
These works do not focus on a setting of FEM-BEM coupling; 
a milestone for studying CQ-based FEM-BEM couplings is the work by Laliena and Sayas \cite{laliena_sayas_cq_scatt}, 
which, however, focuses on the Laplace domain. 
The first full analysis of a FEM-BEM coupling arising 
from convolution quadrature for the wave 
equation is given in \cite{bls_cq_coupling}.

The analysis in the present paper is carried out directly in the time domain, making use of the theory of Runge-Kutta approximation of semigroups
as developed by Brenner, Thomée and Crouzeix \cite{brenner_thomee_rat_approx_semigroup,crouzeix_rk_approx_evolution}. This
allows for stability results that are of interest in their own right, and gives sharper growth conditions (in time) in the appearing constants of the
convergence results in comparison to the standard techniques that use the Laplace/Z-transform to carry out the analysis in the Laplace domain.
A similar observation has recently been made in \cite{banjai_sayas_fd_kirchoff_cq} when applying multistep method based convolution quadrature to the wave equation.

The present work differs from \cite{bls_cq_coupling} in the techniques employed. In particular, 
our tools allow us to analyze a large class of Runge-Kutta methods, 
whereas \cite{bls_cq_coupling} is specific to combining a leapfrog method in the interior with a multistep method 
for convolution quadrature.
In this connection, it is worth pointing out that we use the same Runge-Kutta method both in the interior and the exterior. This forces
us to use implicit schemes (since A-stability  is needed for convolution quadrature) also for the interior problem.

Another advantage of the point of view taken by us, in particular the avoidance of the Laplace domain, is that 
the analysis naturally covers methods that are not strongly $A$-stable, most notably the Gauss methods, 
which have better properties with respect to energy conservation and artificial dissipation.

By using the well established theory of semigroup approximation, we also benefit by avoiding the ``reduction of order'' phenomenon, which is present in 
all the Laplace domain based analyses of convolution quadrature. Instead we recover the full convergence order of the Runge-Kutta scheme employed (instead
of only the stage order or less), although this property also strongly depends on our restricting to the case of a homogeneous equation.


This paper is organized as follows:
In Section~\ref{sect:model_problem} we introduce the Schrödinger equation and the assumptions we need to make on the problem in order to apply our discretization scheme. We derive the semi-discrete problem after applying the 
Runge-Kutta method in time and reformulate it as a problem on a bounded domain with transparent boundary conditions. 
Section~\ref{sect:discretization} is concerned with the spatial discretization using a Galerkin scheme in some abstract subspaces. We then show existence and uniqueness
of the discrete problems and derive a different characterization of the scheme, which is better suited for analysis.
In Section~\ref{sect:abstract_analysis} we develop an abstract theory for problems of the form of Problem~\ref{problem:fully_discrete}, giving stability and approximation results.
Section~\ref{sect:convergence} is concerned with applying this theory to the Schrödinger equation to derive uniform stability and
a best approximation property of the fully discrete scheme with respect to the sequence of approximations which are 
semi-discrete in time.
In Section~\ref{sect:semi_discrete} we go back to the continuous in space/discrete in time setting to derive some stability, regularity, and approximation results. 
We do this by exploiting that the Runge-Kutta approximation can be viewed as a rational approximation of a semigroup.
Combining these results with the best approximation property and well-known results of finite element approximation in Section~\ref{sec:full_error_est}, we finally arrive at an
explicit convergence rate estimate for our approximation sequence.
Section~\ref{sec:numerics} is concerned with confirming the theoretical results of the previous sections in numerical experiments.
Appendix~\ref{appendix_bem_systems} deals with generalizing some results on boundary element methods from the scalar Helmholtz equation to systems of ``Helmholtz-like'' 
problems.

\subsection{Model problem and notation}
\label{sect:model_problem}
For a potential $\mathcal{V}: \R^d \to \R$, we define the Hamilton operator $\mathbf{H}: H^2(\R^d) \to L^2(\R^d)$ by 
\begin{align*}
  \mathbf{H} u:=-\laplace u + \mathcal{V}(\cdot) u.
\end{align*}
The Schrödinger equation reads: Given $\sdu{0} \in H^2 \left(\R^d\right)$, find $u \in C^1((0,\infty),H^2(\R^d)) \cap C^0([0,\infty),H^2(\R^d))$ such that
\begin{align}
  \label{eq:schroed_continuous}
  \ii u_t(t) &= \mathbf{H} u(t), \quad \quad \forall t > 0, \\
  u(0) &= \sdu{0}.
\end{align}

Let $\Omega \subseteq \R^d$ be the bounded Lipschitz domain of interest for the solution. We denote the exterior 
of $\Omega$ by
$\Omega^+:=\R^d \setminus \overline{\Omega}$ and its boundary by $\Gamma:= \partial \Omega$.
The internal trace operators will be denoted by $\gamma^-$ and $\partial_n^-$, while the traces on the exterior domain
will have the index $+$, where $\partial_n$ is the normal derivative with the normal pointing out of $\Omega$.
The jumps of a function $u$ over the boundary will be denoted as
\begin{align*}
  \traceJump{u}:=\gamma^- u - \gamma^+ u ,\quad
  \normalJump{u}:=\partial^-_n u - \partial^+_n u.
\end{align*}

In order to be able to apply our scheme, we need to make some assumptions on the problem.
\begin{assumption}
\label{assumption:problem}
  \begin{enumerate}[(i)]
  \item The potential $x \mapsto \mathcal{V}(x)$ is real valued and bounded.
  \item The potential is constant on  $\Omega^+$, i.e., $\mathcal{V}(x) \equiv \mathcal{V}_0$ $\forall x \in \Omega^+$.
  \item The initial condition vanishes outside of $\Omega$, i.e., $\operatorname{supp}\left(\sdu{0}\right) \subseteq \Omega$.
  \end{enumerate}
\end{assumption}
\begin{notation}
For a space $X$ we will denote the product space $X^m$ by $\prodSpace{X}$. For an operator $G: X \to X$ we will write
$\prodOp{G}: \prodSpace{X} \to \prodSpace{X}$ for the operator $\operatorname*{diag}(G,\dots,G)$. 
\end{notation}
We will also use the 
notation $\mathcal{B}(X,Y)$ to denote the set of all bounded linear operators from $X$ to $Y$.
 We write
$a \lesssim b$ if there exists a constant $C>0$ for which $a \leq C b$ holds; the constant $C$ may depend on $\Omega$, the Runge-Kutta method used,
the potential $\mathcal{V}$, but not on the principal quantities of interest, such as the time step size $k$, the exact solution $u$,
the approximations $u^n$, or the terminal time $T$. We will also write $a \sim b$ to mean $a \lesssim b \lesssim a$.

For any open set $\mathcal{O}$, we write $L^2(\mathcal{O})$ and $H^k(\mathcal{O})$ for the usual Lebesgue and Sobolev spaces. We will write $C_0^\infty(\mathcal{O})$ for
the set of smooth functions with compact support in $\mathcal{O}$. Given that BEM will feature prominently, 
we will also use the fractional order Sobolev spaces on the boundary $\Gamma$ of $\Omega$: 
$H^s(\Gamma)$ for $s>0$ and its dual $H^{-s}(\Gamma):= \left(H^{s}(\Gamma) \right)'$.
Occasionally we will need the adjoint operator of an operator $T$, this will be denoted by $T'$.

For a Banach space $V$, we write $V'$ for its dual space and $\dualprod{V}{\cdot}{\cdot}$ for the duality pairing. 
 The inner product of a Hilbert space $H$ is denoted $\innerprod{H}{\cdot}{\cdot}$. On the boundary $\Gamma$, we 
write $\bdryprod{\cdot}{\cdot}$ for the extension of the standard $L^2(\Gamma)$ inner product 
to $H^{-1/2}(\Gamma) \times H^{1/2}(\Gamma)$.
In order to simplify the notation, we will sometimes encounter matrix products with elements of an abstract Banach space. For $A \in \R^{n \times m}$ and 
$v \in \prodSpace{X}$, we write $A v \in \prodSpace{X}$, for $(A v)_i:=\sum_{j=0}^{m}{A_{ij} v_j}$, $i=1,\dots,n$.

In this paper we consider discretizations based on Runge-Kutta methods; we refer to \cite{hairer_wanner_2} 
for details on Runge-Kutta methods. 

\begin{definition}
  A Runge-Kutta method with $m$ stages is given by a matrix $A \in \R^{m \times m}$ and vectors $b \in \R^m$ and $c \in \R^m$.
  Given a step size $k>0$, and applied to the problem \eqref{eq:schroed_continuous}, the (time) discretization is given by
  \begin{subequations}
      \label{rk_approx_whole_space_1}
    \begin{align}
      \left( I + \ii k A \matH\right) \sdU{n} &= \sdu{n} \ones,\\
      \sdu{n+1}&=(1-\vect{b}^T A^{-1} \ones)\sdu{n} + \vect{b}^T A^{-1} \sdU{n},
    \end{align}
  \end{subequations}
  where $\sdU{n}$ is an $m$-dimensional vector, called \emph{stage vector}, and $\sdu{n}$ represents the
  approximation of $u(n k)$. Here $\ones$ denotes the constant-ones-vector $\ones=(1,\dots,1)^T \in \R^m$.
\end{definition}

We need to make some further assumptions on the Runge-Kutta method used, namely:
\begin{assumption}
\label{assumption:rk_method}
\begin{enumerate}
\item The Runge-Kutta method is \emph{A-stable},
  i.e., for all $z \in \C$ with $\Re(z) \leq 0$ the matrix $I - zA$ is regular, and
  the stability function
  \begin{align}
    \label{def_rk_stability_fkt}
    R(z):=1 + z \vect{b}^T(\vect{I} - z A)^{-1}\mathds{1}
  \end{align}
  satisfies $\abs{R(z)} \leq 1 $.
  \item The matrix $A$ is invertible.
\end{enumerate}
\end{assumption}

\begin{remark}
  Examples of A-stable Runge-Kutta methods with invertible matrix $A$ include the well-known families
  of Radau IIA and Gauss methods (see \cite{hairer_wanner_2} for their definitions). 
  Thus, methods of arbitrary order and some symplectic methods (the Gauss methods) are included. 
  It is common in the literature on convolution quadrature to make further assumptions on the 
  stability function $R$ such as $\abs{R(\ii t)}<1$ for $t \in \R \setminus \{0\}$, which excludes the Gauss methods; our analysis 
  naturally includes these methods without further difficulty.
\eremk
\end{remark}

We will often use an alternative representation of $R(z)$ (the simple proof of the equivalence can, for example, be found in \cite{rk_convq_wave}):
\begin{align}
  \label{def_rk_stability_fkt2}
  R(z)&=(1-\vect{b}^T A^{-1}\mathds{1}) + \vect{b}^T A^{-1} \left(I - zA \right)^{-1} \mathds{1}.
\end{align}

For the remainder of the paper we will use the definition $R(\infty):=1-\vect{b}^T A^{-1} \ones$, 
multiply equation \eqref{rk_approx_whole_space_1} by $-\ii A^{-1}$
and set $\vecrhs:=-\ii A^{-1} \ones$ to simplify the notation. This gives us the equivalent system:
\begin{subequations}
  \label{rk_approx_whole_space_2}
  \begin{align}
    \left( -\ii A^{-1} + k \matH\right) \sdU{n} &= \sdu{n} \vecrhs, \\
    \sdu{n+1}&=R(\infty)\sdu{n} + \vect{b}^T A^{-1} \sdU{n}.
  \end{align}
\end{subequations}

The properties of the system \eqref{rk_approx_whole_space_2} strongly depend on the spectrum of $A$. This is the content of the following lemma.
\begin{lemma}
\label{rk_spectrum_a}
If the matrix $A$ of an A-stable Runge-Kutta method is invertible, then its spectrum satisfies
\begin{align*}
  \sigma(A) \subseteq \{ \lambda \in \C: \; \Re(\lambda) > 0\}.
\end{align*}
\begin{proof}
  By assumption we have $0 \notin \sigma(A)$.
  For $\lambda \neq 0$ with $\Re(\lambda) \leq 0$ we calculate:
  \begin{align*}
    A-\lambda \mathbf{I}&= -\lambda \left(\mathbf{I} + \frac{1}{\lambda} A  \right).
  \end{align*}
  It holds
  \begin{align*}
    \Re\left(\frac{1}{\lambda}\right)=\Re\left(\frac{\overline{\lambda}}{\abs{\lambda}^2}\right) \leq 0.
  \end{align*}
  Since the method is A-stable, the matrix  $\left( \mathbf{I} + \frac{1}{\lambda} A  \right)$ is invertible (cf. (\ref{def_rk_stability_fkt2})) thus 
and $\lambda \notin \sigma(A)$.
\end{proof}
\end{lemma}

The tool we use to derive transparent boundary conditions will be the \emph{Z-transform} or \emph{generating function}. We
formulate this transformation in a general lemma:
\begin{lemma}
\label{z_transform_lemma}
  Let $X$ be a Hilbert space. Let $T$ be a closed, not necessarily bounded, operator on $X$.
  Let two sequences $\left( y_n \right)_{n \in \N}  \subseteq X$ and $\left( \vect{Y}^n \right)_{n\in \N} \subseteq \prodSpace{X}$ be given that satisfy
  \begin{align}
    y_0 &= 0,\\
    \left( -\ii  A^{-1} + k \prodOp{T} \right) \vect{Y}^n &= y_n \vecrhs, \label{z_trans_lemma_stages_eq}\\
    y_{n+1}&=R(\infty) y_n + \vect{b}^T A^{-1} \vect{Y}^n. \label{z_trans_lemma_step_eq}
  \end{align}

  We define the $Z$-transform of the sequence $(Y^n)_{n \in \N}$ as the formal power series
  \begin{align*}
    \hat{Y}:=\sum_{n=0}^{\infty}{Y^n z^n}.
  \end{align*}

  If we assume that the Z-transform of $\left( \vect{Y}^n \right)_{n \in \N}$ exists for 
  sufficiently small $z$ as a power series in $\prodSpace{X}$, for example, if we have $\norm{\vect{Y}^n} \leq C e^{\omega n}$ 
  for some  constants $C$ and $\omega$,
  then 
  the Z-transform of $(Y^n)_{n \in \N}$ solves 
  \begin{align}
    \label{z_transform_lemma_result}
    - \frac{\ii \delta(z)}{k} \hat{Y} + \prodOp{T} \,\hat{Y} &= 0,
  \end{align}

  where the matrix-valued function $z\mapsto \delta(z)$ is defined as
  \begin{align}
    \label{def:delta}
    \delta(z):=\left(A + \frac{z}{1-z} \ones \vect{b}^T \right)^{-1}.
  \end{align}

  \begin{proof}
    First we note a characterization of $\delta(z)$ that is a simple consequence of the Sherman-Morrison formula:
    For $\abs{z} < 1$ we have
    \begin{align*}
      \delta(z) = A^{-1} - \frac{z A^{-1} \ones \vect{b}^T A^{-1}}{1-z R(\infty)}.
    \end{align*}

    We consider the Z-transform of $\left(y_n\right)_{n\in \N}$. Starting from (\ref{z_trans_lemma_step_eq}) we multiply with $z^n$. Summing over all $n \in \N$ then gives
    \begin{align}
      \label{eq:_zy_in_terms_of_stages}
      z^{-1} \left(\hat{y} - y_0 \right) &=R(\infty) \hat{y} + \vect{b}^T A^{-1} \hat{\vect{Y}},
    \end{align}
    or, since we assumed that $y_0 = 0$:  
    \begin{align*}
      \hat{y} &= \left(z^{-1} - R(\infty)\right)^{-1} \vect{b}^T A^{-1} \hat{\vect{Y}}.
    \end{align*}
    
    The Z-transform of $\left( \vect{Y}^n \right)_{n\in \N}$ is more involved, since it involves an unbounded operator. 
    We start from (\ref{z_trans_lemma_stages_eq}) and multiply again with $z^n$ and sum up to a fixed $N \in \N$ to get
    \begin{align*}
      \prodOp{T} \sum_{n=0}^{N}{ \vect{Y}^n z^n} &=  \sum_{j=0}^{N}{ \left( \ii k^{-1} A^{-1} \vect{Y}^n + y_n k^{-1} \vecrhs \right) \,z^n}.
    \end{align*}

    If we assume that the Z-transforms exists, we have that both $a_N:=\sum_{j=0}^{N}{ \vect{Y}^n z^n}$ 
    and $b_N:= \prodOp{T} a_N= \sum_{n=0}^{N}{ \left( \ii k^{-1} A^{-1} \vect{Y}^n + y_n k^{-1}\vecrhs \right) \,z^n}$ converge for $N \to \infty$. 
    Since $T$ is closed we have $\prodOp{T} \lim_{N \to \infty} a_N = \lim_{N \to \infty} b_N$, or
    \begin{align*}
      \prodOp{T}\left( \sum_{n=0}^{\infty}{ \vect{Y}^n z^n} \right) &=  \sum_{n=0}^{\infty}{ \left( \ii k^{-1} A^{-1} \vect{Y}^n + y_n k^{-1} \vecrhs \right) \,z^n}.
    \end{align*}
    This is an equation for the Z-transforms:
    \begin{align*}
      \prodOp{T} \hat{\vect{Y}} &= \ii k^{-1} A^{-1} \hat{\vect{Y}} +  k^{-1} \vecrhs \hat{y}.
    \end{align*}
    Inserting the expressions for $\hat{y}$ and $\vecrhs$ gives
    \begin{align*}
      \prodOp{T} \hat{\vect{Y}} &= \ii k^{-1} \left( A^{-1} - A^{-1} \frac{1}{z^{-1} - R(\infty)} \ones \vect{b}^T A^{-1} \right) \hat{\vect{Y}},
    \end{align*}    
    and a simple calculation then concludes the proof.
  \end{proof}
\end{lemma}

The matrix-valued function $z\mapsto \delta(z)$ defined in \eqref{def:delta} plays an important role in the method. 
The following proposition, taken from \cite{rk_convq_wave}, estimates its spectrum:
\begin{proposition}[{\cite[Lemma~{2.6}]{rk_convq_wave}}]
\label{prop:spectrum_of_delta}
  For an RK-method with invertible $A$ and for $\abs{z} < 1$, the spectrum of $\delta(z)$ satisfies
  \begin{align*}
    \sigma(\delta(z)) \subseteq \sigma(A^{-1}) \cup \{ w \in \C: R(w) z = 1 \}.
  \end{align*}
  Hence, if the Runge-Kutta method is A-stable, then $\sigma(\delta(z))$ lies in the open right half-plane $\C_+:=\{z \in \C: \Re(z) > 0\}$.
\end{proposition}

We apply Lemma~\ref{z_transform_lemma} to our Runge-Kutta approximations, restricted to the exterior domain $\Omega^+$, with $X=L^2( \Omega^+ )$ and $T:=\mathbf{H}$.
Because the sequence of approximations is norm preserving (see Lemma~\ref{sd_stages_apriori}), we get
for $\abs{z}<1$ that the Z-transform exists as an $L^2(\Omega^+)$ power series and therefore solves the differential equation:
\begin{align}
\label{eq:referee-wanted-this}
  -\prodOp{\laplace} \hat{\vect{U}} - \left( \frac{\ii \delta(z)}{k} - \mathcal{V}_0\right) \hat{\vect{U}} &= 0 \quad \quad \text{in } \Omega^+.
\end{align}

The partial differential equation above is structurally similar to a Helmholtz problem with complex wave number (the difference that, for $m>1$, it is matrix-valued,
is addressed in Appendix~\ref{appendix_bem_systems}). 
This allows us to use boundary element methods for the discretization.
We recall some important definitions below (see the 
books \cite{book_sauter_schwab,book_steinbach,book_hsiao_wendland,book_mclean} 
for details on the BEM and integral equations).

\begin{definition}
  For $\Re(s) > 0$, the fundamental solution of the operator $\laplace - s^2$ is given by
  \begin{align}
    \Phi(x,y;s):= 
    \begin{cases}
      \frac{\ii}{4} H^{(1)}_{0} \left(\ii s \abs{x-y}\right) & \text{for $d=2$} \\
      \frac{e^{-s\abs{x-y}}}{4 \pi \abs{x-y}} & \text{for $d=3$}, 
      \end{cases}
  \end{align}
  where $H^{(1)}_0$ is the Hankel function of the first kind and order zero. Next, we define the 
  Newton, single and double layer potentials:
  For $f \in C^{\infty}_0(\R^d \setminus \Gamma)$, $\lambda \in \hpbdry{-1/2}$, and $\phi \in \hpbdry{1/2}$ we set
  \begin{subequations}
  \begin{align}
      \label{def_bem_potentials}
      \left(N(s) f\right)(x)&:=\int_{ \R^d \setminus \Gamma }{\Phi(x,y;s) f(y) \; dy }, &\forall &x \in \R^d \setminus \Gamma, \\
      \left(S(s) \lambda\right)(x)&:=\int_{\Gamma}{\Phi(x,y;s) \lambda(y) \; d\Gamma(y)}, &\forall  &x \in \R^d \setminus \Gamma, \\
      \left(D(s) \phi\right)(x)&:=\int_{\Gamma}{\partial_{n(y)} \Phi(x,y;s) \phi(y) \; d\Gamma(y)}, &\forall &x \in \R^d \setminus \Gamma.
    \end{align}
  \end{subequations}

  We will also need the following operators on the boundary, formally given by:
  \begin{subequations}
    \label{def_bem_operators}
  \begin{align}
    \label{def_bem_operators-a}
    V(s) \lambda&:= \int_{\Gamma}{\Phi( \cdot,y,s ) \lambda(y) \, d\Gamma(y)}, \\
    \label{def_bem_operators-b}
    K^T(s) \lambda&:=\int_{\Gamma}{\partial_{n( \cdot)} \Phi( \cdot,y,s ) \lambda(y) \, d\Gamma(y)}, \\
    \label{def_bem_operators-c}
    K(s) \phi&:=\int_{\Gamma}{\partial_{n(y)} \Phi( \cdot,y,s ) \phi(y) \, d\Gamma(y)}, \\
    \label{def_bem_operators-d}
    W(s) \phi&:= - \partial_{n} \int_{\Gamma}{\partial_{n(y)} \Phi( \cdot,y,s ) \phi(y) \, d\Gamma(y)}. 
  \end{align}
  \end{subequations}

  We have the following connections between the potentials and the operators:
  \begin{align}
    \label{bem_operator_properties}
    \gamma^{\pm} S = V, \quad \partial^{\pm}_n S  = \mp \frac{1}{2} I + K^T, \quad \gamma^{\pm} D = \pm \frac{1}{2} I + K, \quad \partial_n^{\pm} D= -W.
  \end{align}  
\end{definition}

We will often replace the wave number $s$ with a matrix. This is understood in the following sense:
\begin{definition}
  \label{def:riesz_dunford}
  Let $F: G \to \mathcal{B}(X,Y)$ be a holomorphic function that is defined on a domain $G \subseteq \C$ and maps into the space of bounded linear operators 
  between the Banach spaces $X$ and $Y$.
  Let $B$ be a matrix with $\sigma(B) \subseteq G$.
  We then define  $F(B)$ via the Riesz-Dunford functional calculus for holomorphic functions:
  \begin{align*}
    F(B):=\frac{1}{2 \pi \ii}\int_{\mathcal{C}}{ \left(B - \lambda\right)^{-1} \otimes F(\lambda)  d\lambda },
  \end{align*} 
  where $\mathcal{C}\subset G$ is a closed path with winding number $1$ encircling $\sigma(B)$. The operator 
  $\otimes$ denotes the Kronecker product, i.e., for a matrix $A$
  \begin{align*}
    A\otimes F:=
    \begin{pmatrix}
      a_{11} F& \dots & a_{1m} F \\
      \vdots& \dots & \vdots \\
      a_{m1} F& \dots & a_{mm} F 
    \end{pmatrix},
  \end{align*}
  defines an operator mapping from the product space $\prodSpace{X}$ to the product space $\prodSpace{Y}$.
\end{definition}


\begin{proposition}[Calderón system]
\label{proposition:calderon}
For $B \in \C^{m \times m}$, let $X \in \prodSpace{H^1}(\R^d \setminus \Gamma)$ solve the equation $-\prodOp{\laplace} X + B^2 X = 0$ in $\R^d \setminus \Gamma$.
Then, the following identities hold on the boundary:
\begin{align*}
  \colvec{ \gamma^- X \\ \partial^+_n X }&= 
  \begin{pmatrix}
    \frac{1}{2} - K(B)  & V(B) \\
    W(B) & -\frac{1}{2} + K^T(B)
  \end{pmatrix}
  \colvec{ \traceJump{X} \\ \normalJump{X} }.
\end{align*}
Here  $K(B)$ is defined using the scalar operator $K$ from 
    (\ref{def_bem_operators-c}), and the concatenation $K(B)$ is taken in the sense of 
    Def.~\ref{def:riesz_dunford}. The operator $\pm 1/2$ is a shorthand for $\pm1/2$ times the identity operator in the appropriate product space.
 \end{proposition}
\begin{proof}
  The result is well-known for the scalar case, and easily generalizes to the case of systems 
  (see Appendix~\ref{appendix_bem_systems} for the details).
\end{proof}

\begin{corollary}
  \label{cor:calderon_hatu}
  The Z-transform $\hat{U}$ satisfies the following boundary integral equations (cf.~(\ref{eq:referee-wanted-this})):
  \begin{align}    
\label{eq:cor:calderon_hatu-10}
    \begin{pmatrix}
      \frac{1}{2} - K  & V \\
      W & -\frac{1}{2} + K^T
    \end{pmatrix}
    \colvec{ \gamma^-\hat{U} \\ \partial_n^-\hat{U} }
    &=\colvec{ 0 \\  -\partial_n^- \hat{U}},
  \end{align}
  where all operators are understood with respect to the matrix
  \begin{align}
    \label{eq:def_Bz}
    B(z)&:=\sqrt{-\left( \frac{\ii \delta(z)}{k} - \mathcal{V}_0\right)},
  \end{align}
  using the principal branch of the square root (i.e. satisfying $\Re(z) \geq 0$) and the Riesz-Dunford calculus.
  \begin{proof} 
    The function $X$ defined by $X=\hat{U}$ in $\Omega^+$ and $0$ in $\Omega^-$ satisfies the Helmholtz equation. 
    Applying Proposition~\ref{proposition:calderon} to $X$ and using afterwards the fact that 
    $\gamma^- \hat{U} = \gamma^+\hat{U}$, $\partial_n^+ \hat{U} = \partial_n^- \hat{U}$ gives the stated result.
  \end{proof}
  
\end{corollary}

\begin{notation}
  For simplicity we will often drop the matrix dependence in the arguments of BEM operators and just write, 
for example, $V(z)$ instead of $V(B(z))$.  If it is not explicitly stated otherwise, the BEM operators will 
always be understood 
``with respect to the matrix  $B(z):=\sqrt{-\left( \frac{\ii \delta(z)}{k} - \mathcal{V}_0\right)}$''.
\eremk
\end{notation}

\begin{remark}
The fundamental solution $\Phi$ is an analytic function on $\C_+$. 
This implies that also the boundary integral operators depend analytically on the wave number $s$. 
Thus, also $z \to V\left(B(z)\right)$, etc. are analytic.    
\eremk
\end{remark}

Using the stability estimates of Lemma~\ref{sd_apriori} and \ref{sd_stages_apriori} it is easy to see  that we have the estimate
\begin{align*}
  \ltwonorm{\mathbf{H} \sdU{n}} \leq C(k) \left(\ltwonorm{\sdU{n}} + \ltwonorm{\sdu{n}} \right) \leq C(k) \ltwonorm{\sdu{0}}.
\end{align*}
Hence, the power series $\sum_{n=0}^{\infty}{ z^{n} \prodOp{\laplace} \sdU{n} }$ also converges, and we can apply $\gamma^-$ and $\partial^-_n$
to calculate
\begin{align*}
  \partial_n^- \hat{U} &= \partial_n^{-} \sum_{n=0}^{\infty}{\sdU{n} z^n} = \sum_{n=0}^{\infty}{\partial_n^- \sdU{n} z^{n}} = \widehat{\left( \partial_n^{-} \sdU{n} \right)_n}, \\
\gamma^- \hat{U} &= \widehat{\left(\gamma^- \sdU{n}\right)_n}. 
\end{align*}

We will use the following notation, which is standard in the literature on convolution quadrature:
\begin{definition}
\label{def:discrete_operational_calculus}
Let $X$, $Y$ be two Banach spaces and ${\mathcal B}(X,Y)$ be the space of bounded linear operators mapping
from $X$ to $Y$. Let $F: \C_+ \to \mathcal{B}(X,Y)$ be holomorphic.  
  Let $g = (g_n)_{n \in \N_0}$ be a sequence of elements in $\prodSpace{X}$.
  We define a sequence $F(\partial_t^k) g$ as
  \begin{align*}
    \left(F(\partial_t^k) g\right)_n :=\sum_{j=0}^{n}{ W^{n-j}(F) g_j},
  \end{align*}
  where the operators $W^{n-j}$ are defined as the coefficients of the power series
  \begin{align}
    \label{eq:def_cq_op_calc_weights}
    F\left(\frac{\delta(z)}{k}\right)&=:
    \sum_{j=0}^{\infty}{W^j(F) z^j}.
  \end{align}
  Here, $\delta(z)$ is defined in \eqref{def:delta}.
  Since we will always be dealing with operators of the form $F\left(\sqrt{-\ii z+ V_0}\right)$ (see  $B(z)$ as defined in Proposition~\ref{proposition:calderon}),
  we will just shorten the notation to $F(\partial_t^k) g:=(F \circ \sqrt{- \ii \cdot+V_0})(\partial_t^k) g$.
\end{definition}
\begin{notation}
  We will commit a slight abuse of notation in order to simplify the sequence notation. We write 
  \begin{align*}
    F(\partial_t^k) g_n:= \left( F(\partial_t^k) g \right)_n,
  \end{align*}
  i.e., we pretend $F(\partial_t^k)$ acts like an operator on $g_n$ instead of on the whole sequence.
  This should not lead to confusion, we only have to remember that all the CQ-operators will always be 
  non-local in time.
\end{notation}
An important property of the definition above, which makes it useful for deriving transparent boundary conditions, is that it commutes with the Z-transform. We formalize this in the following lemma:
\begin{lemma}
\label{lemma:z_transform_op_calc}
  Let $F$ and $g$ be as in Definition~\ref{def:discrete_operational_calculus}.
  Assume that $\widehat{g}(z)$ exists for $\abs{z}$ sufficiently small.
  Then $\left(\widehat{F(\partial_t^k) g}\right)(z)$ also exists, and the following identity holds:
  \begin{align*}
    \widehat{F(\partial_t^k) g} &= F\left(\frac{\delta(z)}{k} \right) \widehat{g}.
  \end{align*}
  \begin{proof}
    We start with the right-hand side. Abbreviate $\tilde{z}:=\frac{\delta(z)}{k}$.
    Inserting the power series from the definition of the coefficients $W^n$ and using the Cauchy product formula gives
    \begin{align*}
      F(\tilde{z}) \hat{g}&=\left(\sum_{n=0}^{\infty}{W^n z^n}\right)\left(\sum_{j=0}^{\infty}{g_j z^j}\right) 
      =\sum_{n=0}^{\infty}{z^n\Big(\sum\nolimits_{j=0}^{n}{W^{n-j} g_j} \Big)}
      = \widehat{F(\partial_n^k) g}.
\qedhere
    \end{align*}
  \end{proof}
\end{lemma}
Since we are interested in a Galerkin approximation, we will switch to a weak formulation. The following sesquilinear form, representing
the weak form of a Runge-Kutta step, will be used throughout the rest of the paper:
\begin{definition}
  For an open set $\mathcal{O}$ and a function  $g\in L^{\infty}(\mathcal{O})$ 
  we define the sesquilinear form 
${\mathcal A}_{\mathcal{O},g}{}{}$ by:
  \begin{align*}
    \blfA{\mathcal{O}}{g}{U}{V}:=\ltwoprodgenp{\mathcal{O}}{-\ii A^{-1} U}{V} + k\,\ltwoprodgenp{\mathcal{O}}{\nabla U}{\nabla V}
    + k\,\ltwoprodgenp{\mathcal{O}}{ g U}{ V}. 
  \end{align*}
\end{definition}

With this notation we can rewrite \eqref{rk_approx_whole_space_2} as an equivalent system 
with transparent boundary conditions that are realized in terms of boundary integral operators.

\begin{theorem}
\label{thm:semidiscrete_system_fem_bem_version}
Setting $\sdl{n}:=\partial^-_n \sdU{n}$, the semi-discrete problem of \eqref{rk_approx_whole_space_2} is equivalent
to the following problem for the sequence $(\sdU{n},\sdl{n})$:

For all $n \in \N$, find $\sdU{n} \in \prodHpInt{1}$, $\sdu{n} \in \hpint{1}$, $\sdl{n} \in \prodHpBdry{-1/2}$ 
such that 
\begin{subequations}
  \label{eq:semidiscrete_system_fem_bem_version_blf}
\begin{align}
  \blfAInt{\sdU{n}}{V} + k\,\bdryprodp{ W(\partial_t^k) \gamma^- \sdU{n} - \left(1/2 - K^T(\partial_t^k) \right)  \sdl{n}}{ \gamma^- V}  
  &=  \ltwoprodintp{\sdu{n}  \vecrhs}{V},  
\qquad  \forall V \in \prodHpInt{1}
\\
  \bdryprodp{(1/2 - K(\partial_t^k) ) \gamma^- \sdU{n}}{ \mu } + \bdryprodp{V(\partial_t^k) \lambda^{n}}{ \mu} &= 0.
\qquad \qquad \qquad \forall \mu \in \prodHpBdry{-1/2}. 
\end{align}
\end{subequations}
 The solution outside of $\Omega$ can be recovered by applying convolution quadrature to the
representation formula:
\begin{align*}
  U^n|_{\Omega^c}=- S(\partial_t^k) \sdl{n} + D(\partial_t^{k}) \gamma^- U^n.
\end{align*}

Introducing the operator $\mathcal{A}_{int}: \prodSpace{H^{1}}(\Omega) \to (\prodSpace{H^1}(\Omega))'$
corresponding to $\blfAInt{\cdot}{\cdot}$, the problem (\ref{eq:semidiscrete_system_fem_bem_version_blf}) 
can be written more compactly in the matrix operator form

\begin{align}  
  \label{eq:semidiscrete_system_fem_bem_version_matop}
  \renewcommand\arraystretch{2}
  \begin{pmatrix}
    \mathcal{A}_{int} + k\; (\gamma^{-})^\prime W(\partial_t^k)\gamma^- & k\;(\gamma^{-})^\prime \left(-1/2 + K^T(\partial_t^k) \right) \\
    \left( 1/2 - K(\partial_t^k) \right) \gamma^- & V(\partial_t^k) 
  \end{pmatrix}
  \begin{pmatrix}
    \sdU{n} \\ \sdl{n}
  \end{pmatrix}
  &=
  \renewcommand\arraystretch{2}
  \begin{pmatrix}
    \sdu{n} \, \vecrhs \\ 0
  \end{pmatrix},
\end{align}
where $\gamma^-$ denotes the trace operator and $(\gamma^-)^\prime$ its adjoint, and the equality is
understood in the sense of $(\prodOp{H^1}(\Omega))' \times \prodOp{H^{1/2}}(\Gamma)$.
\end{theorem}
\begin{proof}
  Here, we will only show that the sequences $\sdU{n},\sdl{n}$ solve 
  the problem (\ref{eq:semidiscrete_system_fem_bem_version_blf}).
  The equivalence will follow later from the uniqueness of the solution, as shown in 
  Corollary~\ref{corollary:uniqueness}. Recall (\ref{eq:cor:calderon_hatu-10}) of 
  Corollary~\ref{cor:calderon_hatu}. Using Lemma~\ref{lemma:z_transform_op_calc} and 
  exploiting that the coefficients of a power series are unique, we get:
  \begin{align}    
\label{eq:thm:semidiscrete_system_fem_bem_version-10}
    \begin{pmatrix}
      \frac{1}{2} - K(\partial_t^k)  & V(\partial_t^k) \\
      W(\partial_t^k) & -\frac{1}{2} + K^T(\partial_t^k)
    \end{pmatrix}
    \colvec{ \gamma^-{U^n} \\ \partial_n^-{U^n} }
    &=\colvec{ 0 \\  -\partial_n^- {U^n}}. 
  \end{align}
  We multiply \eqref{rk_approx_whole_space_2} with a test function $V \in \prodSpace{H^1}(\Omega)$, integrate
  over $\Omega$ and integrate by parts. The resulting boundary term $-k\bdryprodp{\partial^-_n \sdU{n}}{\gamma^- V}$ 
  can be replaced using the second equation of (\ref{eq:thm:semidiscrete_system_fem_bem_version-10}),
  and we arrive at (\ref{eq:semidiscrete_system_fem_bem_version_blf}). 
\end{proof}

\begin{remark}
  Looking at \eqref{eq:semidiscrete_system_fem_bem_version_matop} we clearly see the relation to the 
symmetric coupling of finite elements and boundary elements, as developed by 
Costabel \cite{costabel_symmetric_coupling} and Han \cite{han_a_new_fembem_coupling}.
We only had to replace the appearing boundary operators with the convolution quadrature version, 
e.g., $W \to W(\partial_t^k)$ etc.
\eremk
\end{remark}

\section{Spatial discretization} 
\label{sect:discretization}
In order to get a fully discrete scheme, we choose closed spaces
$X_h \subseteq \hpint{1}$ and $Y_h \subseteq \hpbdry{-1/2}$. Then the fully discrete problem is given by:

\begin{problem}
\label{problem:fully_discrete}
For all $n \in \N$, find $\fdU{n} \in \prodSpace{X_h}$, $\fdu{n} \in X_h$, $\fdl{n} \in \prodSpace{Y_h}$
such that for all $V_h \in \prodSpace{X_h}$, $\mu_h \in \prodSpace{Y_h}$,
\begin{subequations}
  \label{eq:fully_discrete_system_fem_bem_version_blf}
\begin{align}
  \label{eq:fully_discrete_system_fem_bem_version_blf-a}
  \blfAInt{\fdU{n}}{V_h} + k\,\bdryprodp{W(\partial_t^k) \gamma^- \fdU{n}  - 
  \left(1/2 - K^T(\partial_t^k) \right)  \fdl{n} } {\gamma^-V_h}
  &=  \ltwoprodintp{\fdu{n} \vecrhs}{V_h},  \\
  \label{eq:fully_discrete_system_fem_bem_version_blf-b}
  \bdryprodp{\big(1/2 - K(\partial_t^k) \big)\gamma^- \fdU{n}}{ \mu_h } + \bdryprodp{V(\partial_t^k) \fdl{n}}{ \mu_h} &= 0.
\end{align}
The approximation at $t=(n+1)k$ is then defined as:
\begin{align}
  \fdu{n+1}&=R(\infty) \fdu{n} + b^T A^{-1} \fdU{n}.
\end{align}
\end{subequations}
Define
\begin{subequations}
 \label{definition_ustar_explicit}
\begin{align}
  \fdUstar{n} (x)&
  :=\left(- S(\partial_t^k) \fdl{n} \right) (x) + \left(D(\partial_t^{k}) \gamma^- \fdU{n} \right) (x), \quad x \in \R^{d} \setminus \Gamma, \\
  \fdustar{n+1}&:=R(\infty) \fdustar{n} + b^T A^{-1} \fdUstar{n}.
\end{align}
The restrictions $\fdUstar{n}|_{\Omega^+}$ and $\fdustar{n}|_{\Omega^+}$ can be understood as 
approximations to $\sdU{n}|_{\Omega^+}$ and $\sdu{n}|_{\Omega^+}$.  
\end{subequations}
\end{problem}

\begin{remark}
  The fact that we allowed $x \in \Omega$ in the definition of $\fdUstar{n}$ will be important 
  for the later characterization of the FEM-BEM coupling problem as a PDE problem in $\R^n$.  
\eremk
\end{remark}

%

In the following, we will derive a problem that is equivalent to Problem~\ref{problem:fully_discrete} and that is better suited for theoretical analysis
since it avoids the non-locality in time of the convolution terms. However, it will no longer consist of computable
terms due to its being posed on the whole space. The construction is such that under the Z-transform it will result in the
non-standard transmission problem from \cite{laliena_sayas_cq_scatt} for the symmetric FEM-BEM coupling.

We introduce the following spaces:
\begin{align}
\pairltwo&:=\ltwogen{\Omega} \times \ltwogen{\R^d \setminus \Gamma}, & \pairhone:=\hpgen{1}{\Omega} \times \hpgen{1}{\R^d \setminus \Gamma},
\end{align}
equipped with the sum inner products, and a new sesquilinear form on $\prodSpace{\pairhone}$:
\begin{align}
  \label{definition_b}
  \blfB{U}{U^*}{V}{V^*}:= \blfAInt{U}{V} + \blfAFull{U^*}{V^*}.
\end{align}

For the analysis it will be useful to introduce a stabilized energy sesquilinear form.
Let 
\begin{equation}
\label{eq:alpha} 
\alpha > 1+ \norm{\mathcal{V}}_{L^{\infty}(\R^d)}
\end{equation}
and set:
\begin{align}
  \label{def:hstab}
  \blfHStab{u}{v}&:=\innerprod{\pairltwo}{\nabla u}{\nabla v}+ \innerprod{\pairltwo}{\mathcal{V} u}{v}+\alpha \innerprod{\pairltwo}{u}{v}
\end{align}
for all $u,v \in \pairhone$. Here $\mathcal{V}u$ denotes multiplication with $\mathcal{V}( \cdot )$ in the first component and $\mathcal{V}_0$ in
the second.
It is easy to see that $\blfHStab{u}{u}$ is equivalent to the $\pairhone$-norm with a constant that depends only on $\mathcal{V}$ and $\alpha$. 
We flag at this point that $\widetilde{\mathbf H} $ will also used to denote the 
operator induced by the sequilinear form (\ref{def:hstab}). Furthermore, we will require later 
$\underline{\widetilde {\mathbf H}}(\cdot,\cdot)$ and 
$\underline{\widetilde {\mathbf H}}$ to denote sesquilinear forms and induced operators on products of spaces. 

We recall the definition of the annihilator of a subspace:
\begin{definition}
\label{def_anihilator}
  Let $X \subseteq Y$ be Banach spaces. The annihilator of $X$ in $Y$, 
  denoted $X^\circ \subseteq Y^{\prime}$, is defined by
  \begin{align*}
    X^{\circ}:=\left\{ f \in Y^\prime\; : \; \dualprod{Y}{f}{x} = 0 \;\, \forall x \in X \right\}.
  \end{align*}
\end{definition}

We are now able to formulate the equivalent problem in the following lemma:
\begin{lemma}
\label{equivalent_formulation_lemma}
 For given Hilbert spaces $X_h \subseteq H^1(\Omega)$ and $Y_h \subseteq H^{-1/2}(\Gamma)$, define the space
 \begin{align*}
   \hat{H}(X_h,Y_h):=\{(v,v^*) \in X_h \times \hpfull{1} : \traceJump{v^*} = -\gamma^- v \land \gamma^- v^* \in Y_h^{\circ} \}.
 \end{align*}
  
  Then the sequence of problems: Find $(\fdU{n},\fdUstar{n}) \in \prodSpace{\hat{H}(X_h,Y_h)}$ such that
  \begin{align}
    \label{equivalent_formulation_lemma_eqn}
    \blfB{\fdU{n}}{\fdUstar{n}}{V_h}{V^*}&= \blfFp{\fdu{n} \vecrhs}{\fdustar{n} \vecrhs}{V_h}{V^*} & \forall (V_h,V^*) \in \prodSpace{\hat{H}(X_h,Y_h)},
    \end{align}    
    where the $\fdu{n+1}$ and $\fdustar{n+1}$ are again defined in the usual way, i.e., 
  \begin{align*}
    \fdu{n+1}&:=R(\infty) \fdu{n} + \vect{b}^T A^{-1} \fdU{n},
&
    \fdustar{n+1}&:=R(\infty) \fdustar{n} + \vect{b}^T A^{-1} \fdUstar{n},
  \end{align*}
  is equivalent to the fully discrete problem (Problem~\ref{problem:fully_discrete}) with the 
understanding that $\fdustar{n}, \fdUstar{n}$ are defined by the post-processing of \eqref{definition_ustar_explicit}.

In particular, for $X_h=H^1(\Omega)$ and $Y_h=\hpbdry{-1/2}$, 
the approximations 
\begin{align*}
\begin{cases} 
\fdU{n} & \text{ in $\Omega$} \\
\fdUstar{n}|_{\R^d\setminus\overline{\Omega}} & \text{ in $\R^d \setminus \overline{\Omega}$}
\end{cases}, 
\qquad \qquad 
\begin{cases} 
\fdu{n} & \text{ in $\Omega$} \\
\fdustar{n}|_{\R^d\setminus\overline{\Omega}} & \text{ in $\R^d \setminus \overline{\Omega}$}
\end{cases}, 
\end{align*}
coincide with those of \eqref{rk_approx_whole_space_1}. 
Furthermore, $\fdustar{n}|_{\Omega} = 0$ and $\fdUstar{n}|_{\Omega} = 0$. Finally, 
$\traceJump{\fdUstar{n}}= - \gamma^- \fdU{n}$ and $\normalJump{\fdUstar{n}} = -\partial^-_n \fdU{n} = -\lambda^n_h$. 
\end{lemma}

Before we prove this lemma, we  first take a separate look at a family of problems 
that will allow us to describe the ``exterior'' terms $(\fdustar{n}),(\fdUstar{n})$ as solutions
to elliptic problems with important trace relations.
\begin{lemma}
\label{extension_lemma}
  Let $X_h \subseteq H^1(\Omega)$, $Y_h \subseteq H^{-1/2}(\Gamma)$  be Hilbert spaces.
  Consider sequences of functions $\left(X^n_* \right)_{n\in \N}\subset \prodHpFull{1}$, $\left(x^n_* \right)_{n \in \N} \subset \hpfull{1} $ 
  that satisfy $\gamma^- X^n_* \in \prodSpace{Y_h^\circ}$ 
  and solve, for all $n \in \N$,  
  \begin{align}
    \label{extension_lemma_weak_form}
    \blfAFull{X^n_{*}}{V^*} &= \ltwoprodfullp{x^n_* \vecrhs}{V^*}  \quad \forall V^* \in \prodSpace{H^{*}_0}(Y_h), \\
    x^{n+1}_*&:=R(\infty) x_*^n + \vect{b}^T A^{-1} X^{n}_*
  \end{align}
  with $H^*_0(Y_h) := \{ v^* \in \hpfull{1} : \gamma^- v^* \in Y_h^\circ \,\land\, \traceJump{v^*} = 0 \}$.

  Then the sequences have the following properties:
  \begin{enumerate}[(i)]
  \item
    \label{item:extension_lemma-i}
    With 
    $H^*(X_h):=\left\{u \in \hpfull{1}: \traceJump{u} \in \gamma^-X_h \right\}$, there holds, 
    for all $V^* \in \prodSpace{H^*}(X_h)$,
    \begin{align}
      \blfAFull{X^n_*}{V^*} - k\,\bdryprodp{\partial^+_n{X^n_*}}{ \traceJump{V^*} } - k\,\bdryprodp{\normalJump{X^n_*}}{\gamma^- V^*}
      &= \ltwoprodfullp{x^n_* \vecrhs}{V^*}.
      \label{full_var_eq_ext_lemma}
    \end{align}
  \item 
    \label{item:extension_lemma-ii}
    On the boundary we have $\normalJump{X^n_*} \in \prodSpace{Y_h}$.
  \item 
    \label{item:extension_lemma-iii}
The traces solve 
\begin{subequations}
      \label{extension_lemma-5}
    \begin{align}
      \label{extension_lemma-10}
      \partial_n^+ X^n_*&= \left(-1/2+K^T(\partial_t^k) \right) \normalJump{X^{n}_*} + W(\partial_t^k)  \traceJump{X^{n}_*}, \\
      \label{extension_lemma-20}
      0&=\bdryprodp{V(\partial_t^k) \normalJump{X^{n}_*}}{\mu_h} + \bdryprodp{(1/2 - K(\partial_t^k)) \traceJump{X^{n}_*}}{\mu_h} 
            &\forall \mu_h \in \prodSpace{Y_h}. 
    \end{align}
\end{subequations}
  \end{enumerate}
  \begin{proof}       
    First we choose test functions $V^*= v^* \, e_j$ with $v^* \in C_0^{\infty}(\R^d \setminus \Gamma) \subseteq H_0^*(Y_h)$ 
    in (\ref{extension_lemma_weak_form}) and get by integration by parts:
    \begin{align}
\label{eq:extension_lemma:100}
      \left( - \ii A^{-1} - k\, \laplace + k\,\mathcal{V}_0  \right) X^n_* &= x^n_* \vecrhs \quad \text{ in } \R^d \setminus \Gamma.
    \end{align}
    This implies, by doing integration by parts in (\ref{extension_lemma_weak_form}),
    that if we insert arbitrary $V^* \in \prodSpace{H^*_0}(Y_h)$ (i.e., allowing  non-vanishing boundary terms), the following holds:
    \begin{align*}
      \bdryprodp{\partial_n^- X^n_*}{\gamma^- V^*} - \bdryprodp{\partial_n^+ X^n_*}{\gamma^+ V^{*}} &= 0.
    \end{align*}
{\em Proof of (\ref{item:extension_lemma-ii}):}
    Let $\xi \in \prodSpace{Y_h}^{\circ} \subseteq \prodSpace{H}^{1/2}(\Gamma)$ and choose $V^* \in \prodHpFull{1}$
    as a lifting such that $\gamma^+ V^*=\gamma^- V^* = \xi$. This gives
    \begin{align*}
      \bdryprod{\normalJump{ X^n_*}}{ \xi } &=0 \quad \forall \xi \in \prodSpace{Y_h}^{\circ},
    \end{align*}
    or $\normalJump{X^n_*} \in \left(\prodSpace{Y_h}^{\circ} \right)^{\circ}=\prodSpace{Y_h}$.
    
{\em Proof of (\ref{item:extension_lemma-iii}):}
We proceed completely analogously to the derivation of the transparent boundary conditions.    
    We take the Z-transform and see by Lemma~\ref{z_transform_lemma} that $\widehat{X_*}$ solves
    \begin{align*}
      -\left(\frac{\ii \delta(z)}{k} - V_0\right) \widehat{X_*} - \laplace \widehat{X_*} &= 0, \quad \text{ on } \R^d \setminus \Gamma.
    \end{align*}
    Applying the Calder\'on identities (Proposition~\ref{proposition:calderon}) and taking the inverse Z-transform then gives 
    (\ref{extension_lemma-5}) if we use that $\bdryprodp{\gamma^- X^n_*}{\mu_h}=0$, since $\gamma^- X^n_* \in \prodSpace{Y_h}^\circ$.

{\em Proof of (\ref{item:extension_lemma-i}):}
    Equation (\ref{full_var_eq_ext_lemma}) is a simple consequence of the differential equation (\ref{eq:extension_lemma:100}) 
    and integration by parts.
  \end{proof}
\end{lemma}

\begin{proof}[Proof of Lemma~\ref{equivalent_formulation_lemma}]
  We start with solutions $u^n_h$, $\lambda^n_h$, $U^n_h$ of \eqref{eq:fully_discrete_system_fem_bem_version_blf}.
  We construct a sequence $(\fdUstar{n},\fdustar{n})_{n\in \N}$ that satisfies the
  conditions of the previous lemma. To that end, we set $\fdustar{0}:=0$ and define the functions 
  $\fdUstar{n}$ and $\fdustar{n}$ inductively so that they satisfy 

\begin{subequations}
    \label{equiv_form_lemma_proof_1}
  \begin{align}
      \blfAFull{\fdUstar{n}}{V^*} 
    &= \ltwoprodfullp{\fdustar{n} \vecrhs}{V^*}  \quad \forall V^* \in \prodSpace{H^{*}_0}(Y_h), \\
    \traceJump{\fdUstar{n}} &= - \gamma^- \fdU{n}, \\
\fdustar{n+1}&:=R(\infty) \fdustar{n} + b^T A^{-1} \fdUstar{n}. 
    \end{align}
\end{subequations}
%
    To construct this, take $\xi_n$ a lifting of $\gamma^- \fdU{n}$ on the exterior and $0$ on the interior.
    Then we set $\fdUstar{n}:=X^n_* + \xi_n$, where $X^n_* \in \prodSpace{H^{*}_0(Y_h)}$ solves 
    \begin{align*}
      \blfAFull{X^n_*}{V^*} &= \ltwoprodfullp{\fdustar{n}  \vecrhs}{V^*} - \blfAFull{\xi_n}{V^*} & \forall V^* \in \prodSpace{H^{*}_0}(Y_h).
    \end{align*}
    The existence of the solutions is guaranteed by Lemma~\ref{lemma_solve_matrix_eqn} and the fact that the scalar problems are elliptic due
    to a non-vanishing imaginary part of the ``wave number.''
    We must show that $(\fdU{n},\fdUstar{n})$ solves \eqref{equivalent_formulation_lemma_eqn}.
    In view of (\ref{equiv_form_lemma_proof_1}), we may apply Lemma~\ref{extension_lemma} to $\fdUstar{n}$. 
    From (\ref{extension_lemma-20}) and $\traceJump{\fdUstar{n}}= - \gamma^- \fdU{n}$ we get 
    \begin{align*}
      \bdryprodp{V(\partial_t^k) \normalJump{\fdUstar{n}}}{\mu} -\bdryprodp{\left(1/2- K(\partial_t^k)\right) \gamma^- \fdU{n}}{\mu} &= 0
    \qquad\forall \mu \in \prodSpace{Y_h}. 
  \end{align*}
  This is the same equation as (\ref{eq:fully_discrete_system_fem_bem_version_blf-b}) for $-\lambda^n_h$. 
In order 
  to show $\normalJump{\fdUstar{n}} = -\lambda^n_h$ we use the  definition of $V(\partial_t^k)$ as the sum over 
  the history to arrive at 
  \begin{align*}
    \sum_{j=0}^{n}{ \bdryprodp{V^j \left(\normalJump{\fdUstar{n-j}} + \fdl{n-j}\right)}{\mu_h}}&=0
    \quad\forall \mu_h \in \prodSpace{Y_h}. 
  \end{align*}  
    Since both, $\lambda^n_h$ and $\normalJump{\fdUstar{n-j}}$ are in the discrete space $Y_h$,
    it is easy to see that an induction will yield 
  $\normalJump{\fdUstar{n}} = -\lambda^n_h$ as soon as we have asserted that 
$V^0$ is injective when viewed as  
    an operator $\prodSpace{Y_h} \to \prodSpace{Y_h}'$.
    We note that $V^0=V(B(0))$ with $B(0)$ defined in~\eqref{eq:def_Bz} since $V^0$ is the leading term in the
      Taylor series of $V(B(z))$ at $0$.
      By \cite[Proposition 16]{laliena_sayas_cq_scatt}, \cite{bamberger_duong_1,bamberger_duong_2},
      $V(s)$ satisfies the ellipticity estimate: 
      $
        \Re{\left(e^{\ii \operatorname{Arg}s}\bdryprod{\lambda}{V(s)\lambda} \right)} \geq \frac{\Re(s) \min(1,\Re{s})}{\abs{s}^2} \norm{\lambda}^2 _{-1/2},
      $      
    therefore the inverse operator $V^{-1}(s)$ exists as an operator between discrete spaces $Y_h' \to Y_h$. 
    From the composition property of the Riesz-Dunford calculus or by using the Jordan form, similar to the proof 
    of Lemma~\ref{lemma_solve_matrix_eqn}, this implies that $V(B(0)): \prodSpace{Y_h} \to \prodSpace{Y_h}'$ 
    is also invertible and in particular injective. 
    We conclude $\lambda^n_h=-\normalJump{\fdUstar{n}}$ for all $n \in \N$. 
  If we insert $\normalJump{\fdUstar{n}} = -\lambda^n_h$ into (\ref{extension_lemma-10}) 
  we get:
  \begin{align}
    \label{eq:equiv_form_lemma_proof_1-100}
    {-\partial_n^+\fdUstar{n}}{}&=  (-1/2+K^T(\partial_t^k)) \fdl{n} + W(\partial_t^k) \gamma^- \fdU{n}.
  \end{align}
%
We now claim that $(\fdU{n},\fdUstar{n})$ solves \eqref{equivalent_formulation_lemma_eqn}. 
  When evaluating $B\left((\fdU{n},\fdUstar{n}),(V_h,V^*)\right)$, we employ \eqref{full_var_eq_ext_lemma}
to write 
 \begin{equation} 
    \label{eq:equiv_form_lemma_proof_1-200}
\blfAFull{\fdUstar{n}}{V^*} = 
       k\,\bdryprodp{\partial^+_n{\fdUstar{n}}}{ \traceJump{V^*}}{}  + k\,\bdryprodp{\normalJump{\fdUstar{n}}}{\gamma^- V^*}{}
      + \ltwoprodfullp{u^n_* \vecrhs}{V^*}.
 \end{equation} 
  The second term of the  right-hand side of (\ref{eq:equiv_form_lemma_proof_1-200}) vanishes since 
$\normalJump{\fdUstar{n}}$ is in $\prodSpace{Y_h}$ and $\gamma^- V^*$ is in $\prodSpace{Y_h}^\circ$ by assumption.
We insert (\ref{eq:equiv_form_lemma_proof_1-100}) into the first term of the right-hand side of 
(\ref{eq:equiv_form_lemma_proof_1-200}) 
 \begin{equation} 
    \label{eq:equiv_form_lemma_proof_1-300}
\blfAFull{\fdUstar{n}}{V^*} = 
        k\,\bdryprodp{
     (-1/2+K^T(\partial_t^k)) \fdl{n} + W(\partial_t^k) \gamma^- \fdU{n}
}{ \traceJump{V^*}}{}  
      + \ltwoprodfullp{u^n_* \vecrhs}{V^*}
 \end{equation} 
and observe that this leads to \eqref{equivalent_formulation_lemma_eqn} in view of 
(\ref{eq:fully_discrete_system_fem_bem_version_blf}).
  It remains to be shown that the functions $\fdUstar{n}$, $\fdustar{n}$ that are obtained by the 
  convolution quadrature post-processing of $(\fdU{j})$ and $(\lambda^j_h)$ defined in 
  (\ref{definition_ustar_explicit}) coincide with solution components $\fdUstar{n}$, $\fdustar{n}$ 
  as defined above. 
We consider the Z-transform of the function $\fdUstar{n}$ 
  defined by (\ref{definition_ustar_explicit}). It satisfies the differential equation 
  \eqref{z_transform_lemma_result},
  and also has the same jumps across $\Gamma$. Uniqueness of the Helmholtz problem then gives the result.

In order to see that solutions of (\ref{equivalent_formulation_lemma_eqn}) solve problem 
(\ref{eq:fully_discrete_system_fem_bem_version_blf}),  we select test functions
  $V_h:=0$ and $V^* \in \prodSpace{H^*_0}(Y_h)$ (as defined in Lemma~\ref{extension_lemma}) and observe 
that (\ref{equivalent_formulation_lemma_eqn}) simplifies to
  \begin{align*}
    \blfAFull{\fdUstar{n}}{V^*}&=\ltwoprodfullp{\fdustar{n}  \vecrhs}{V^*}.
  \end{align*}
  Hence, we are in the setting of Lemma~\ref{extension_lemma}. 
  We set $\lambda^n_h:=- \normalJump{\fdUstar{n}}$.
  If we take any pair $(V_h,V^*) \in \prodSpace{\hat{H}(X_h,Y_h)}$, and again argue as above, we arrive at 
    (\ref{eq:equiv_form_lemma_proof_1-300}). Using (\ref{equivalent_formulation_lemma_eqn}) one then sees
  that $\fdU{n}$ and $\lambda_h^n$ solve 
  \eqref{eq:fully_discrete_system_fem_bem_version_blf-a}. The equation
  \eqref{eq:fully_discrete_system_fem_bem_version_blf-b} follows from (\ref{extension_lemma-20}). 

  In the case $X_h=H^1(\Omega)$ and $Y_h=H^{-1/2}(\Gamma)$, the condition $\gamma^- \fdUstar{n} \in \prodSpace{Y_h}^{\circ}$ implies
    $\gamma^- \fdUstar{n} = 0$. Since $\fdustar{0}|_{\Omega}=0$ by definition, we get by induction that $\fdUstar{n}|_{\Omega}=0$ for all $n \in \N$,
    since $\fdUstar{n}|_{\Omega}$ solves the homogeneous problem with zero boundary conditions.
    With this knowledge, it is easy to see that~\eqref{equivalent_formulation_lemma_eqn}
    is just the weak formulation of~\eqref{rk_approx_whole_space_2}. 
\end{proof}

\begin{corollary}
\label{corollary:uniqueness}
  The sequence of fully discrete problems is uniquely solvable for any choice of closed 
subspaces $X_h\subset H^1(\Omega)$, $Y_h\subset H^{1/2}(\Gamma)$  
  and any step size $k>0$.
  Choosing $X_h=\hpint{1}$, $Y_h=\hpbdry{-1/2}$ this also shows uniqueness for the semi-discrete problem in Theorem~\ref{thm:semidiscrete_system_fem_bem_version}.
  \begin{proof}
    Since the fully discrete problem is equivalent to (\ref{equivalent_formulation_lemma_eqn}), it suffices to show existence and uniqueness there.
    This is covered by the statement in Lemma~\ref{lemma_solve_matrix_eqn}.
  \end{proof}
\end{corollary}

\section{Abstract analysis}
\label{sect:abstract_analysis}
In this section, we analyze the time stepping of Lemma~\ref{equivalent_formulation_lemma} in an abstract setting.

\begin{assumption}
\label{assumption:abstract_setting}
Let $H_0$, $H_1$ be Hilbert spaces with $H_0 \supseteq H_1$ continuously and densely embedded, and let
$H_h \subseteq H_1$ be a closed subspace. We assume $H_h$ is equipped with the $H_1$ inner product,
and we will explicitly state when we equip it instead with the $H_0$ inner product.

Assume we are given  a sesquilinear form $\mathbf{H}: H_1 \times H_1 \to \C$
that is bounded and Hermitian, i.e., $\mathbf{H}(u,v)=\overline{\mathbf{H}(v,u)}$.
Also assume that there exists a constant $\alpha >0$ 
such that the stabilized sesquilinear form
\begin{align}
  \label{eq:def:abstract_blfHStab}
  \blfHStab{u}{v}:= \mathbf{H}(u,v) + \alpha \innerprod{H_0}{u}{v}
\end{align}
satisfies an inf-sup condition
\begin{align}
  \label{eq:abstract_setting_infsup}
  \inf_{u \in H_h \setminus \{0\}} \sup_{v \in H_h \setminus \{0\}}{ \frac{|\blfHStab{u}{v}|}{\norm{u}_{H_1} \norm{v}_{H_1}}} & \geq \beta_{\widetilde{\mathbf{H}}}.
\end{align}
We will write $\matH(\cdot,\cdot)$ and $\blfHStabP{\cdot}{\cdot}$ for the corresponding sum sesquilinear forms on
$\prodSpace{H_1} \times \prodSpace{H_1}$.

Define the sesquilinear form
\begin{align}
\label{eq:def:abstact_blf_B}
B(U,V)&:=  - \innerprod{\prodOp{H_0}}{\ii A^{-1}U}{V} + k\matH(U,V) \quad \forall U,V \in \prodSpace{H_1}.
\end{align}

We consider solutions $X^n_h \in \prodSpace{H_h}$, $x_h^n \in H_h$ of:
\begin{subequations}
  \label{eq:def:abstract_X_n}
  \begin{align}
    B(X^n_h,V_h) &= \innerprod{\prodOp{H_0}}{x_h^n \, \vecrhs}{V_h} + \innerprod{\prodOp{H_0}}{F_n}{V_h} \qquad \forall V_h \in \prodSpace{H_h},\\
    x^{n+1}_h &= R(\infty) x_h^n + \vect{b}^T A^{-1} X^n_h,
  \end{align}
  \label{subeq:lemma:discrete_stability_problem}    
\end{subequations}
for some given right-hand sides $F_n \in \prodSpace{H_0}$ and initial condition $x^0_h \in H_h$.
\end{assumption}

We will need the well-known spectral representation theorem for bounded, self-adjoint operators. We will use it in the following
``multiplication operator'' form:
\begin{proposition}[{\cite[Satz VII.1.21, page 335]{werner_fana}}, {\cite[Theorem VII.3, page 227]{reed_simon_1}}]
\label{prop:spectral_theorem}
  Let $T$ be a bounded, self-adjoint operator on a separable Hilbert space $H$. Then there exists a \emph{finite} measure space
  $\langle\mathcal{O},\mu \rangle$, a bounded measurable function $F: \mathcal{O} \to \R$, and a unitary map $\mathcal{U}: H \to L^2(\mathcal{O},d\mu)$, such that
  \begin{align*}
    ( \mathcal{U} T \mathcal{U}^{-1} f) (z) &= F(z) f(z) \quad \forall z \in \mathcal{O}.
  \end{align*}
\end{proposition}

We would like to keep the analysis as general as possible in order to set the stage for problems other 
than the Schr\"odinger equation. For this reason, we required in 
Assumption~\ref{assumption:abstract_setting} merely inf-sup stability and not ellipticity, although the 
Schr\"odinger Hamiltonian considered here is in fact elliptic. In order to track where the stronger
condition of ellipticity of $\blfHStab{\cdot}{\cdot}$ is needed instead of merely inf-sup stability, we mark the corresponding estimates 
with $(*)$. 
%

In Lemma~\ref{lemma_discrete_stability}, we will need that $H_h$ is able to represent its dual
space using the $\pairltwo$ inner product. That this is indeed the case is the subject of the following lemma:
\begin{lemma}
\label{lemma:density_of_H_h}
  The set $M:=\{ \innerprod{H_0}{\cdot}{u_h} : u_h \in H_h \}$ is dense in $\left(H_h,\norm{\cdot}_{H_1}\right)'$
\end{lemma}
\begin{proof}
  We show that the annihilator $M^\circ = \{0\}$. Let $x \in M^{\circ} \subset (H_h')'$. Since $H_h$ is reflexive, we
  can assume $x \in H_h$.
  This means that $f(x) = 0 $ for all $f \in M$, or $0=\innerprod{H_0}{x}{u_h} \; \forall u_h \in H_h$. 
  Setting $u_h=x$ shows $x=0$.  
\end{proof}

The following lemma is the main ingredient of our stability and convergence proofs. It can be seen as a version 
of a theorem by von Neumann (see \cite[Corollary 11.3]{hairer_wanner_2}) about Runge-Kutta stability, 
adapted to our setting.

\begin{lemma}[Discrete stability]
\label{lemma_discrete_stability}  
  Let Assumption~\ref{assumption:abstract_setting} hold. 
  Then, without any conditions on $k$ or the space $H_h$, we have that the sequence of solutions to \eqref{subeq:lemma:discrete_stability_problem} is non-expansive, i.e.,
  for all $n \in \N$:
  \begin{align}
    \label{eq:lemma_discrete_stability_l2}
    \norm{x_h^n}_{H_0} &\leq \norm{x^0_h}_{H_0} + C\sum_{j=0}^{n-1}{\norm{F_j}_{\prodOp{H_0}}}.
  \end{align}
  If we also assume that $\blfHStab{\cdot}{\cdot}$ is $H_1$-elliptic, with $\beta_{\widetilde{\mathbf{H}}}$ as the coercivity constant,
  then there exists a constant $C>0$ depending only on $\beta_{\widetilde{\mathbf{H}}}$ and the Runge-Kutta method 
  such that
  \begin{align}
    \label{eq:lemma_discrete_stability_h1_all_f}
    \norm{x_h^n}_{H_1} &\stackrel{(*)}{\leq} C \left(\norm{x^0_h}_{H_1} +
                         \sum_{j=0}^{n-1}{\inf_{W_h \in \prodOp{H_h}}{\left( \norm{W_h}_{\prodOp{H_1}}
                               + k^{-1/2} \norm{F_j - W_h}_{\prodOp{H_0}}\right)}}\right).
  \end{align}      
  For discrete right-hand sides $F_h \in \prodSpace{H_h}$ the following, stronger estimate is valid: 
  \begin{align}
    \label{eq:lemma_discrete_stability_h1}
    \norm{x_h^n}_{H_1} &\stackrel{(*)}{\leq} C \left(\norm{x^0_h}_{H_1} + \sum_{j=0}^{n-1}{\norm{F_j}_{\prodOp{H_1}}}\right).
  \end{align}
  In the case that the RK-method satisfies $\abs{R(\ii t)}=1$ for all $t\in \R$ and $F_n=0\; \forall n \in \N$ we 
  get conservation of the $H_0$-norm, i.e., 
  \begin{align*}
    \norm{x_h^n}_{H_0} &= \norm{x^0_h}_{H_0} \qquad \forall n \in \N.
  \end{align*}
  Under the stricter ellipticity assumption on $\widetilde{\mathbf{H}}(\cdot,\cdot)$,  we also get ``conservation of energy'':
  \begin{align*}
    \mathbf{H}(x_h^n,x_h^n)&\stackrel{(*)}{=}\mathbf{H}(x^0_h,x^0_h).
  \end{align*}
\end{lemma}
Before we can prove this statement, we need the following reformulation of a Runge-Kutta step:
\begin{lemma}
\label{lemma:3.5}
  Let $H_0$, $H_1$, $B$, $\widetilde{\mathbf{H}}$, $X_h^n$, $x_h^n$, $F_n$ be as in Assumption~\ref{assumption:abstract_setting}. 
The sequence $x_h^n$ solves the following equation:
  \begin{align}
    \label{eq:proof_rk_step_as_operators}
    \innerprod{H_0}{x^{n+1}_h}{\varphi}&=\innerprod{H_0}{ R_T x_{n}}{\varphi} + \innerprod{H_0}{ S_T F_n}{ \varphi}
\qquad \forall \varphi \in H_h, 
  \end{align}
   where $R_T: H_0 \to H_0$, $S_T: \prodSpace{H_0}  \to H_0$ are bounded linear operators 
     with $\operatorname{range}{R_T} \subseteq H_h$ and $\operatorname{range}{S_T} \subseteq H_h$. The operators satisfy the
    following bounds:
  
  \begin{align}
    \label{eq:operator_bounds_ST}
    \norm{S_T}_{\prodSpace{H_0} \to {H_0} } &\leq C,&   \norm{S_T}_{\prodSpace{H_h} \to H_h }  & \stackrel{(*)}{\leq} C,
    & \norm{S_T}_{\prodSpace{H_0} \to {H_1}} & \stackrel{(*)}{\leq} C k^{-1/2},  \\
      \norm{R_T}_{H_0 \to H_0 } &\leq C,&   \norm{R_T}_{H_h \to H_h }  & \stackrel{(*)}{\leq} C,
      & \norm{R_T}_{H_0 \to H_1} & \stackrel{(*)}{\leq} C k^{-1/2},
  \end{align}
  with constants that  depend only on the Runge-Kutta method and $\widetilde{\mathbf{H}}$, but not on $k$ or $H_h$. 

  The operator $R_T$ can be written as 
  \begin{align*}
    R_T u:= R(\infty) u - b^T A^{-1} \left( \left( \ii A^{-1} +  k \alpha I \right)\prodOp{T} - k I \right)^{-1} \prodOp{T} \vecrhs \;u,
  \end{align*}
  where $T$ is a self-adjoint, bounded operator on $H_0$ and bounded on $H_h$.
  If we assume $\blfHStab{\cdot}{\cdot}$ to be elliptic, 
  then $T$ is also self-adjoint on $H_h$ when equipped with the equivalent $\blfHStab{\cdot}{\cdot}$ inner product.
\end{lemma}
\begin{proof}
  We construct $T$ with the goal of  $T \approx \mathbf{H}^{-1}$, which we will then use to represent the Runge-Kutta 
  step in 
  terms of the stability function $R$.
  
  We define the operator $T: H_0 \to H_0$ by setting
  $T(w):=u$ where $u \in H_h$ is the unique solution to
  \begin{align*}
    \blfHStab{u}{y}
    &= \innerprod{H_0}{w}{y} \quad  \forall y \in H_h.
  \end{align*}
  Since the Hermitian sesquilinear form on the left-hand side satisfies an inf-sup condition,
  we get that $T$ is well-defined for all $w \in H_0$ and bounded (see for example \cite[Theorem 2.1.44]{book_sauter_schwab}), 
  with a constant that depends only on $\beta_{\widetilde{\mathbf{H}}}$.
  By construction the operator has $\operatorname{range}(T) \subseteq H_h$, and thus we may treat it
  also as a linear operator $H_h \to H_h$ and $H_{0} \to H_1$.
  
  For $w$, $x \in H_0$ we calculate:
  \begin{align*}
    \innerprod{H_0}{w}{Tx} &= \blfHStab{Tw}{Tx} = \overline{\blfHStab{Tx}{Tw}} = \overline{\innerprod{H_0}{x}{Tw}}= \innerprod{H_0}{Tw}{x},
  \end{align*}
  where we used the fact that $\blfHStab{\cdot}{\cdot}$ was assumed to be Hermitian and $Tx,Tw \in H_h$.
  
  The operator $T$ is in general not self-adjoint with respect to the $H_1$ inner product.
  In the case where $\blfHStab{\cdot}{\cdot}$ is elliptic, i.e., if it induces an equivalent inner product on $H_1$, 
  we calculate for $w$, $x \in H_h$:
  \begin{align*}
    \blfHStab{Tw}{x} &= \innerprod{H_0}{w}{x} = \overline{\innerprod{H_0}{x}{w}} = \overline{\blfHStab{Tx}{w}} = \blfHStab{w}{Tx}.
  \end{align*}
  Thus we have that $T$ is also self-adjoint in the $\widetilde{\mathbf{H}}$-scalar product.

  We define the operator $S_T:\prodSpace{H_0} \rightarrow H_0$ by 
  \begin{align*}
    S_T:= - b^T A^{-1} \left( \left(\ii A^{-1} +  k \alpha I \right)\prodOp{T} -  k I \right)^{-1} \prodOp{T}.
  \end{align*}

  We need to show that this operator is well-defined, i.e.,
  $\left( \left( \ii A^{-1} +  k \alpha I\right)\prodOp{T} -  k I\right)^{-1}$ exists, where
  the inverses are taken from $\prodSpace{H_0}\to \prodSpace{H_0}$.
  $T$ is a self-adjoint operator and therefore only has real spectrum. We rewrite the above inverse as
  \begin{align}
    \label{eq:proof_S_is_well_defined}
    \left( \left( \ii A^{-1} +  k \alpha I\right)\prodOp{T} -  k I\right)^{-1}
    &= \left( \prodOp{T} -  k \left( \ii A^{-1} +  k \alpha I\right)^{-1} \right)^{-1}
      \left( \ii A^{-1} +  k \alpha I\right)^{-1}.
  \end{align}  
  The inverse of $\left( \ii A^{-1} +  k \alpha I\right)$ exists, since $\Re(\sigma(A) ) > 0$ (Lemma~\ref{rk_spectrum_a}).
  For the other inverse of the right-hand side of~\eqref{eq:proof_S_is_well_defined}, it is easy to see that the
  matrix has a spectrum with non-vanishing imaginary part. Therefore we can apply Lemma~\ref{lemma_solve_matrix_eqn},
  setting $V=H= H_0$ and $a(x,y):=\innerprod{H_0}{Tx}{y}$ for the existence of the inverse in $H_0$.

  Next we show that $S_T$ satisfies the operator bounds~\eqref{eq:operator_bounds_ST}. 
  Let $\Phi \in \prodSpace{H_0}$ be arbitrary and set
  $Y:=\left( \left(-\ii A^{-1} -  k \alpha I \right)\prodOp{T} +  k I \right)^{-1} \prodOp{T} \Phi$.
  Since $\prodOp{T}$ has the structure $\prodOp{T}=\operatorname{diag}(T,\dots,T)$,
    it commutes with matrices, so that
    \begin{align*}
      \left( \left(\ii A^{-1} +  k \alpha I \right)\prodOp{T} -  k I \right)^{-1} \prodOp{T} &=
      \prodOp{T}\left( \left(\ii A^{-1} +  k \alpha I \right)\prodOp{T} -  k I \right)^{-1}.
    \end{align*}
    Thus we can write $Y=\prodOp{T}\left( \left(-\ii A^{-1} -  k \alpha I \right)\prodOp{T} +  k I \right)^{-1} \Phi$.
  This implies $Y \in \operatorname{range}{\prodOp{T}} \subseteq \prodSpace{H_h}$,
  and also $\operatorname{range}{S_T} \subseteq H_h$.
  
  We fix a test function $W_h \in \prodSpace{H_h}$ and calculate:
  \begin{align*}
    B(Y,W_h)&=\innerprod{\prodSpace{H_0}}{-\ii A^{-1} Y}{W_h} + k \prodOp{\mathbf{H}}(Y,W_h) \\
    &=\blfHStabP{\prodOp{T}(-\ii A^{-1} Y)}{W_h} + k \blfHStabP{Y}{W_h} - k \alpha \innerprod{\prodSpace{H_0}}{Y}{W_h}\\
    &=\blfHStabP{ \left( -\ii A^{-1} \prodOp{T} - k\alpha \prodOp{T} + kI\right)Y}{W_h} 
    +k\alpha \blfHStabP{\prodOp{T} Y}{W_h} - k\alpha \innerprod{\prodSpace{H_0}}{Y}{W_h} \\
    &=\blfHStabP{ \prodOp{T} \Phi}{W_h} + 0 
    =\innerprod{\prodSpace{H_0}}{\Phi}{W_h}.
  \end{align*}
  This variational problem fits the requirements of Lemma~\ref{lemma_solve_matrix_eqn}. This implies
  the estimates:
  \begin{align}
    \norm{Y}_{\prodSpace{H_0} } &\lesssim \norm{ \Phi}_{\prodSpace{H_0}}
    &  \norm{Y}_{\prodSpace{H_1}} & \stackrel{(*)}{\lesssim} \norm{\Phi}_{\prodSpace{H_h}},
    & \norm{Y}_{\prodSpace{H_1}} & \stackrel{(*)}{\lesssim}  k^{-1/2} \norm{\Phi}_{\prodSpace{H_0}}.
  \end{align}
  (In the second equation, we assumed $\Phi\in \prodSpace{H_h}$).
  From the definition of $S_T$ we get $S_T\Phi=b^{T} A^{-1} Y$, which implies~\eqref{eq:operator_bounds_ST}.

  To show the equality~\eqref{eq:proof_rk_step_as_operators}, we perform a similar calculation,
  but use the formal adjoint operator in the second argument of $B$.
  Let $\varphi \in H_0$ be arbitrary, and set $Y:=-\left( \left(-\ii A^{-T} +  k \alpha I\right)\prodOp{T} -  k I \right)^{-1} A^{-T} b \varphi$. By definition of $T$ we have $\prodOp{T}Y \in \prodSpace{H_h}$.
  Using the definition of $T$ we have
  for any function $W_h \in \prodSpace{H_h}$: 
  \begin{align*}
    B(W_h,\prodOp{T}Y)&=
    \innerprod{\prodSpace{H_0}}{-\ii A^{-1} W_h}{\prodOp{T}Y} + k\, \matH(W_h,\prodOp{T} Y)
    = \innerprod{\prodSpace{H_0}}{ -\ii A^{-1} W_h}{\prodOp{T}Y}  - \alpha k\innerprod{\prodSpace{H_0}}{ W_h}{\prodOp{T}Y}
                        + k\blfHStabP{W_h}{\prodOp{T}Y} \\
    &=\innerprod{\prodSpace{H_0}}{ -\ii A^{-1} W_h}{\prodOp{T}Y}  - \alpha k\innerprod{\prodSpace{H_0}}{ W_h}{\prodOp{T}Y}
      + k\innerprod{\prodSpace{H_0}}{W_h}{Y}
    = \innerprod{\prodSpace{H_0}}{ W_h}{\left(\left( \ii A^{-T} -  k \alpha I\right) \prodOp{T} +  k  I\right) Y} \\
    &=\innerprod{\prodSpace{H_0}}{W_h}{ A^{-T} b \varphi}.
  \end{align*}

  Using equation \eqref{subeq:lemma:discrete_stability_problem} with $\prodOp{T}Y$ as a test function and $W_h=X^n_h$ in the previous calculation, this gives:
  \begin{align*}
    \innerprod{H_0}{b^T A^{-1 } X^n_h}{\varphi} &= \innerprod{\prodSpace{H_0}}{X^n_h}{A^{-T} b \varphi}
    = B(X^n_h,\prodOp{T}Y) = \innerprod{\prodSpace{H_0}}{x_h^{n} \vecrhs}{ \prodOp{T}Y} + \innerprod{\prodSpace{H_0}}{F_n}{\prodOp{T} Y} \\
    &=\innerprod{H_0}{  S_T \vecrhs x_h^n }{\varphi} + \innerprod{H_0}{ S_T F_n }{\varphi},
  \end{align*}
  where in the last step, we used that $T$ is $H_0$-self-adjoint in order to move the operators to the left-hand side of the inner product.
  Adding a term $\innerprod{H_0}{R(\infty) x_h^n}{\varphi}$to both sides 
  and  using the definition $x^{n+1}_h=R(\infty) x_h^n + b^T A^{-1} X^n_h$ then completes the proof.
\end{proof}

\begin{proof}[Proof of Lemma~\ref{lemma_discrete_stability}]  
  Using the representation (\ref{eq:proof_rk_step_as_operators}), 
  we can show the stated stability estimates by using the A-stability of the method.
  Since $T$ is a self-adjoint operator, Proposition~\ref{prop:spectral_theorem} ensures  
  the existence of a measure space $\left(\mathcal{O},\mu\right)$, 
  a unitary transformation $\mathcal{U}: H_0 \to L^2(\mu)$, and
  a measurable function $f: \mathcal{O} \to \R$ such that for all $x \in L^2(\mu)$:
  \begin{align*}
    \mathcal{U}T\mathcal{U}^{-1} x &= f \cdot x =: M_{f}(x).
  \end{align*}

  Using this transformation we get:
  \begin{align*}
    (R_T -R(\infty)) u
    &= -\mathcal{U}^{-1} \vect{b}^T A^{-1} \left( \left( \ii A^{-1} +  k \alpha \right) \prodOp{M_f} - k \right)^{-1}  \prodOp{M_f} (-\ii A^{-1}\ones) \mathcal{U} x_h^n\\
    &= \mathcal{U}^{-1} M_g \, \mathcal{U} x_h^n
  \end{align*}
  with the new function $g(z):=\vect{b}^T A^{-1} \left( \left(I - \ii k \alpha A\right) f(z) + \ii k A \right)^{-1}  f(z) \ones$.

  For $f(z)=0$ it is easy to see that $g(z)=0$. For $f(z) \ne 0$ we get:
  \begin{align*}
    g(z)&= \vect{b}^T A^{-1} \left( \left(I - \ii k \alpha A\right) f(z) + \ii k A \right)^{-1}  f(z) \ones  \\
    &=\vect{b}^T A^{-1} \left( \left(I - \ii k \alpha A\right)  + \frac{\ii k}{f(z)} A \right)^{-1} \ones  \\
    &=\vect{b}^T A^{-1} \left( I - \ii k \left(\alpha- \frac{1}{f(z)} \right) A \right)^{-1} \ones \\
    &= R \left(\ii k \left(\alpha-\frac{1}{f(z)}\right)\right) - R(\infty). 
  \end{align*}

  Setting $h(z):=\left(\alpha-\frac{1}{f(z)}\right)$ with the convention $h(z):=\infty$ for $f(z)=0$, we arrive at 
  \begin{align}
    \norm{ R_T x_h^n}_{H_0} 
    &= \ltwonormgen{\mu}{\left(R(\infty) + g(z)\right) \mathcal{U} x_h^n} = 
    \ltwonormgen{\mu}{ R (\ii k h(z)) \mathcal{U} x_{n}} 
    \leq \norm{x_h^n}_{H_0} \label{eqn:stab_lemma_h_leq_1},
  \end{align}
  where in the last step we utilized that $z\mapsto f(z)$ and thus also $z\mapsto h(z)$
  is real valued, $\abs{R(z)} \leq 1$ on the imaginary axis and $\mathcal{U}$ is unitary. If we assume that
  $\blfHStab{\cdot}{\cdot}$ is elliptic and write $\widetilde{H}_1$ for $H_1$ 
  with the $\blfHStab{\cdot}{\cdot}$ inner product,
  we can apply the same argument, using the spectral representation theorem in 
  $\left(H_h,\norm{\cdot}_{\widetilde{H}_1}\right)$,  to show that
  $\norm{ R_T x_h^{n+1}}_{\widetilde{H}_1} \stackrel{(*)}{\leq} \norm{x_h^{n}}_{\widetilde{H}_1} $.

  Setting $\varphi:=x^{n+1}_h$ in \eqref{eq:proof_rk_step_as_operators} then directly gives 
  the $H_0$-stability estimate (\ref{eq:lemma_discrete_stability_l2}):
  \begin{align*}
    \norm{x^{n+1}_h}_{H_0} &\leq \norm{x^{n}_h}_{H_0} + C \norm{F_n}_{\prodSpace{H_0}},
  \end{align*}    
  which implies \eqref{eq:lemma_discrete_stability_l2} via the discrete Gronwall lemma.

  We now  show (\ref{eq:lemma_discrete_stability_h1_all_f}). 
  We repeat the previous construction. In order to reformulate \eqref{eq:proof_rk_step_as_operators} in terms of the $\widetilde H_1$ inner product instead
  of the $H_0$ inner product,
  we take a sequence $(L_{\varepsilon})_{\varepsilon \geq 0} \subseteq H_h$ such that
  $\innerprod{H_0}{W}{L_{\varepsilon}} \stackrel{\varepsilon\rightarrow 0}{\to} \blfHStab{W}{x^{n+1}_h}$
  for all $W \in H_h$ (which is possible due to Lemma~\ref{lemma:density_of_H_h}, and the fact that the right-hand side is a continuous functional in $H_h$)
. Then  use $\varphi:= L_{\varepsilon}$ as a test function. By Lemma~\ref{lemma:3.5}
  \begin{align*}
    \innerprod{H_0}{x^{n+1}_h}{\; L_{\varepsilon}}
    &= \innerprod{H_0}{ R_T x_h^n}{L_\varepsilon} + \innerprod{H_0}{S_T F_n}{L_{\varepsilon} }.
  \end{align*}
  Since the stability properties of $S_T$ depends on whether or not its argument is in $H_0$ or $H_h$ 
(cf.~(\ref{eq:operator_bounds_ST})) we choose an arbitrary $W_h \in \prodSpace{H_h}$ and write:
  
  \begin{align*}
    \innerprod{H_0}{x^{n+1}_h}{\; L_{\varepsilon}}
    &= \innerprod{H_0}{ R_T x_h^n}{L_\varepsilon} + \innerprod{H_0}{S_T W_h}{L_{\varepsilon} } + \innerprod{H_0}{S_T\left(F_n - W_h \right)}{L_{\varepsilon} }.
  \end{align*}
  Passing to the limit $\varepsilon \to 0$ and using the $H_1$-stability of $S_T$, we get 
  \begin{align*}
    \limsup_{\varepsilon \rightarrow 0} \abs{\innerprod{H_0}{S_T W_h}{ L_\varepsilon}}
    &\stackrel{(*)}{\leq} C \norm{W_h}_{\prodSpace{\widetilde H_1}} \norm{x^{n+1}_h}_{\widetilde H_1}, \\
    \limsup_{\varepsilon \rightarrow 0}  \abs{\innerprod{H_0}{S_T\left( F_n - W_h \right)}{ L_\varepsilon}}
    &\stackrel{(*)}{\leq} C k^{-1/2}\norm{F_n - W_h}_{\prodSpace{H_0}} \norm{x^{n+1}_h}_{\widetilde H_1}. \\
  \end{align*}
  Therefore we end up with:
  \begin{align}
    \label{int_stab_proof_h1est}
    \norm{x^{n+1}_h}_{\widetilde H_1}^2 & \stackrel{(*)}{\leq }
    \norm{ R_T x_h^n}_{\widetilde H_1} \norm{x^{n+1}_h}_{\widetilde H_1}
    + C \norm{W_h}_{\prodSpace{\widetilde H_1}} \norm{x^{n+1}_h}_{\widetilde H_1}
    + C k^{-1/2} \norm{ F_n - W_h}_{\prodSpace{H_0}} \norm{x^{n+1}_h}_{\widetilde H_1} .
  \end{align}
  Since we have already established the bound (\ref{eqn:stab_lemma_h_leq_1}) on $R_T x_h^n$ 
we get from (\ref{int_stab_proof_h1est}):
  \begin{align*}
    \norm{x^{n+1}_h}_{\widetilde H_1} &\stackrel{(*)}{\leq} \norm{x_h^n}_{\widetilde H_1} 
                                        + C \norm{W_h}_{\prodSpace{\widetilde H_1}} 
    + C k^{-1/2} \norm{ F_n - W_h}_{\prodSpace{H_0}}.
  \end{align*}
  By taking the infimum over all $W_h$ and applying the discrete Gronwall lemma, this gives:
  \begin{align*}
    \norm{x^{n}_h}_{\widetilde H_1} &\stackrel{(*)}{\leq} \norm{x_h^0}_{\widetilde H_1}
                                      + C \sum_{j=0}^{n-1}{\inf_{W_h \in \prodSpace{H_h}}
                                        \left( \norm{W_h}_{\prodSpace{\widetilde H_1}} + k^{-1/2}\norm{ F_j -W_h}_{\prodSpace{H_0}} \right)}.
  \end{align*}
  The equivalence of the $\widetilde H_1$ and $H_1$-norms then gives the estimate \eqref{eq:lemma_discrete_stability_h1}.
  
  To get the conservation of the $H_0$ norm we need to show the reverse inequality.
  This time, we use $\varphi:=R_T x_h^n$ as a test function in \eqref{eq:proof_rk_step_as_operators} to get
  \begin{align*}
    \innerprod{H_0}{x^{n+1}_h}{R_T x_h^n}
    &=\innerprod{H_0}{R_T \, x_h^n}{R_T x_h^n}.
  \end{align*}
  We again use the characterization of $R_T$ by the spectral theorem and see that we can replace the inequality
  \eqref{eqn:stab_lemma_h_leq_1} by an equality if we assume $\abs{R(\ii t)} = 1$. Combining
  this observation with the Cauchy-Schwarz inequality for the left-hand side we get
  \begin{align*}
    \norm{x^{n+1}_h}_{H_0}&\geq \norm{x_h^n}_{H_0}.
  \end{align*}    
  Completely analogously to the $H_0$-case we can also show 
  conservation of the $\widetilde H_1$-norm  when $\abs{R(\ii t)}=1$ (again, assuming ellipticity).
  Since the stated energy only differs by $\alpha \norm{x_h^n}_{H_0}^2$ from this norm, we can just subtract it (we already showed that the $H_0$-norm is
  conserved) to get energy conservation.
\end{proof}

We now investigate convergence properties of the spatial discretization, which in our abstract setting is determined by the
space $H_h \subseteq H_1$. The semi-discrete problem is formulated as follows:
\begin{assumption}
  \label{assumption:semidiscrete_problem}
  Let $H_2 \subseteq H_1$ be a subspace, and let $c: H_2 \times H_1 \to \R$ be a bounded sesquilinear form.
  Define $V_0 := \left\{v_h \in H_2: c(v_h,w_h) = 0 \;\; \forall  w_h \in H_h \right\}\subseteq H_2$.
  Let $X^n \in \prodSpace{H_2}$, $x_n \in H_2$ solve:
  \begin{subequations}
    \label{eq:def:abstract_sd_X_n}
    \begin{align}
      B(X^n,V) + k \underline{c}(X^n,V)&= \innerprod{\prodSpace{H_0}}{x^n \, \vecrhs}{V}
      + \innerprod{\prodSpace{H_0}}{F_n}{V} \qquad \forall V \in \prodSpace{H_h},\\
      x^{n+1} &= R(\infty) x^n + \vect{b}^T A^{-1} X^n,
    \end{align}
    \label{subeq:lemma:abstract_semidiscrete_problem}    
  \end{subequations}
  for some given right-hand sides $F_n \in \prodSpace{H_0}$ and $x^0 \in H_1$, where $\underline{c}(U,V):=\sum_{j=0}^{m}{c(u_j,v_j)}$.
\end{assumption}

\begin{remark}
  In our analysis below, the purpose of the sesquilinear form $c$ is 
  to account for a consistency error that arises from the fact that our error analysis is performed in a
  non-conforming setting. Specifically, the discrete and continuous test spaces satisfy, in general, 
  $\hat H(X_h,Y_h) \not\subseteq \hat H(H^1(\Omega),H^{-1/2}(\Gamma))$, where the constrained spaces 
  $\hat H(X_h,Y_h)$ and $\hat H(H^1(\Omega),H^{-1/2}(\Gamma))$ are defined in 
  Lemma~\ref{equivalent_formulation_lemma}.
\eremk
\end{remark}

In order to estimate the error $\sdx{n} - \fdx{n}$ we introduce a Ritz-style projector:
\begin{definition}
  Set
  $\Pi_h:  H_2 \to H_h, \;  w \mapsto u_h$,  where $u_h \in H_h$ solves
  \begin{align}
    \label{eq:ritz_projection_eq}
    \blfHStab{u_h}{v_h}
    &= \blfHStab{w}{v_h} + c(w,v_h)   \quad \forall v_h \in H_h.  
  \end{align}
\end{definition}
Note that $\Pi_h$ is well-defined by Assumption~\ref{assumption:abstract_setting}.
This projection allows us to bound the error of our fully discrete scheme in terms of the approximation properties of $\Pi_h$. We formalize this
in the following lemma:
\begin{lemma}
  \label{lemma:abstract_approx_with_projection}
  Let Assumptions~\ref{assumption:abstract_setting} and \ref{assumption:semidiscrete_problem} be satisfied.   
  Write $\matH X^{j}:=\ii k^{-1} A^{-1} \left(X^{j}-x^{j} \ones\right)$.
\begin{enumerate}[(i)]
\item 
  There exist constants $C_0,C_1 > 0$ that depend only on the Runge-Kutta method and $\blfHStab{\cdot}{\cdot}$ 
  but not on $k$ and $H_h$ such that for all $n \in \N$ we can estimate:
  \begin{align}
    \norm{\sdx{n+1} - \fdx{n+1}}_{H_0}
    &\leq C_0 \left(\norm{\sdx{0}-\fdx{0}}_{H_0} +  \norm{\sdx{0}-\Pi_h \sdx{0}}_{H_0}\right)
      + C_1 \; k \sum_{j=0}^{n}{\left(\norm{\left( I - \underline{\Pi_h} \right) \matH \sdX{j}}_{\prodSpace{H_0}} + \norm{\left( I - \prodOp{\Pi_h} \right) \sdX{j}}_{\prodSpace{H_0}} \right)}. \label{eq:error_est_0}
  \end{align}
\item
  Assume additionally that $\widetilde{H}(\cdot,\cdot)$ is elliptic and  the following approximation property holds
  for all $u \in H_1$:
  \begin{align}
  \label{eq:lemma:abstract_approx_with_projection-100}
    \inf_{w_h \in H_h}{\norm{ u- w_h}_{H_0} } &\leq C k^{+1/2} \norm{u}_{H_1}. 
  \end{align}
  Then we have 
  \begin{align}
    \norm{\sdx{n+1} - \fdx{n+1}}_{H_1}
    &\stackrel{(*)}{\leq} C_0 \left(\norm{\sdx{0}-\fdx{0}}_{H_1} +  \norm{\sdx{0}-\Pi_h \sdx{0}}_{H_1}\right) +
    C_1 \; k \sum_{j=0}^{n}{\left(\norm{\left( I - \underline{\Pi_h} \right) \matH \sdX{j}}_{\prodSpace{H_1}} 
        + \norm{\left( I - \prodOp{\Pi_h} \right) \sdX{j}}_{\prodSpace{H_1}} \right)}. \label{eq:error_est_1}
  \end{align}
\end{enumerate}
\end{lemma}
\begin{remark}
  The definition of $\matH$ may seem arbitrary in the abstract context, but it is chosen in a way that reflects the pointwise semi-discrete
  problem of \eqref{rk_approx_whole_space_1}.   
\eremk
\end{remark}

\begin{remark}
Assumption \eqref{eq:lemma:abstract_approx_with_projection-100} introduces a (in practice quite weak) coupling between mesh size $h$ and time-step size $k$. In Section~\ref{sect:better_h1_est} we will later also
see a way to remove this assumption for a restricted set of Runge-Kutta methods. 
\eremk
\end{remark}

\begin{proof}[Proof of Lemma~\ref{lemma:abstract_approx_with_projection}]
For simplicity, we consider for the moment the case $\alpha=0$ and  $F_j=0 \;\forall j \in \N_0$ and calculate
for $V_h \in H_h$
\begin{align*}
  B\left(\underline{\Pi_h} \sdX{n},V_h \right) &= \innerprod{\prodSpace{H_0}}{-\ii A^{-1} \underline{\Pi_h} \sdX{n}}{V_h}
  + k \, \blfHStabP{\underline{\Pi_h} \sdX{n}}{V_h} \\
  &= \innerprod{\prodSpace{H_0}}{-\ii A^{-1} \underline{\Pi_h} \sdX{n}}{V_h} + k\,\blfHStabP{ \sdX{n}}{V_h} + k \underline{c}(\sdX{n},V_h) \\
  &\stackrel{(\ref{eq:def:abstact_blf_B}), (\ref{eq:def:abstract_sd_X_n})}{=} \innerprod{\prodSpace{H_0}}{ -\ii A^{-1} \underline{\Pi_h} \sdX{n}}{V_h}
  + \innerprod{\prodSpace{H_0}}{\sdx{n} \vecrhs}{V_h} - \innerprod{\prodSpace{H_0}}{-\ii A^{-1} \sdX{n}}{V_h}\\
  &= \innerprod{\prodSpace{H_0}}{\Pi_h \sdx{n} \vecrhs}{V_h} + \innerprod{\prodSpace{H_0}}{\sdx{n} \vecrhs - \Pi_h \sdx{n} \vecrhs}{V_h} 
  + \innerprod{\prodSpace{H_0}}{\ii A^{-1} (I-\underline{\Pi_h}) \sdX{n}}{V_h} \\
  &=  \innerprod{\prodSpace{H_0}}{\Pi_h \sdx{n} \vecrhs}{V_h}
  + \innerprod{\prodSpace{H_0}}{\left(\sdx{n}\vecrhs + \ii A^{-1} \sdX{n}\right) - \underline{\Pi_h}\left(\sdx{n}\vecrhs + \ii A^{-1} \sdX{n}\right)}{V_h}   \\
  &= \innerprod{\prodSpace{H_0}}{\Pi_h \sdx{n} \vecrhs}{V_h} + k \innerprod{\prodSpace{H_0}}{\left(I-\underline{\Pi_h} \right) \matH \sdX{n}}{V_h}.
\end{align*}

In the general case, doing a completely analogous computation, we see that $\underline{\Pi_h} \sdX{n}$ solves:
\begin{align}
  B(\underline{\Pi_h} \sdX{n},V_h) &= \innerprod{\prodSpace{H_0}}{\Pi_h \sdx{n} \vecrhs}{V_h} + \innerprod{\prodSpace{H_0}}{F_n}{V_h}
  + \innerprod{\prodSpace{H_0}}{\Xi_n}{V_h}, \nonumber\\ 
  \text{with } \Xi^n&:= k \,\left( I - \underline{\Pi_h} \right)  \matH \sdX{n} +  k\, \alpha \, \left( I - \underline{\Pi_h} \right) \sdX{n}. \label{eq:def_consistency_errors}
\end{align}

We now consider the error propagation between the projection and the fully discrete solution, and set
$E^n:= \left(\underline{\Pi_h} \sdX{n} - \fdX{n} \right)$ and $e_n:=\Pi_h \sdx{n} - \fdx{n}$. 

For simplicity, we now assume that $\fdx{0}= \Pi_h \sdx{0}$. Then the error solves:
\begin{align}
  B(E^n,V_h)&= B(\underline{\Pi_h} \sdX{n},V_h) - B(\fdX{n},V_h)  
  =\innerprod{\prodSpace{H_0}}{\Pi_h \sdx{n} \vecrhs}{V_h} -  \innerprod{\prodSpace{H_0}}{\fdx{n}\vecrhs}{V_h} + \innerprod{\prodSpace{H_0}}{\Xi^n}{V_h}  \nonumber\\  
  &=\innerprod{\prodSpace{H_0}}{e_n \vecrhs}{V_h} + \innerprod{\prodSpace{H_0}}{\Xi^n}{V_h}.  \label{eq:error_eqn_with_consistency_errors}
\end{align}
By linearity of $\Pi_h$, we have $e_{n+1}= R(\infty) e_n + \vect{b}^T A^{-1} E^n$. 
So the error terms fit into the setting of our discrete
stability lemma (Lemma~\ref{lemma_discrete_stability}). We get:
\begin{align*}
  \norm{e_{n+1}}_{H_0} &\leq   C  \sum_{j=0}^{n}{ \norm{\Xi^j}_{\prodSpace{H_0}}},  &
  \norm{e_{n+1}}_{H_1} &\stackrel{(*)}{\leq}  C  \sum_{j=0}^{n}{ \norm{\Xi^j}_{\prodSpace{H_1}}},
\end{align*}
where the second estimate again depends on the ellipticity of $\blfHStab{\cdot}{\cdot}$, and we
absorbed the $k^{-1/2}$ term using the approximation assumption (\ref{eq:lemma:abstract_approx_with_projection-100}).
inserting the Ritz projector and using the triangle inequality gives:  
\begin{align*}
  \norm{\sdx{n+1} - \fdx{n+1}}_{H_0} &\leq \norm{\sdx{n+1} - \Pi_h \sdx{n+1}}_{H_0} + \norm{\Pi_h \sdx{n+1} - \fdx{n+1}}_{H_0}\\
  & \leq \norm{ \left( I - \Pi_h \right) \sdx{n+1}}_{H_0} +  C  \;k \sum_{j=0}^{n}{ 
    \left(\norm{\left( I - \underline{\Pi_h} \right) \matH \sdX{j}}_{\prodSpace{H_0}} + \norm{\left( I - \underline{\Pi_h} \right) \sdX{j}}_{\prodSpace{H_0}} \right)}. 
\end{align*}
In order to slightly simplify the above expression we would like to absorb the first term into the sum. Since
we assumed  $\Pi_h \sdx{0}=\sdx{0}$ we get
\begin{align*}
  \norm{ \left( I - \Pi_h \right) \sdx{n+1}}_{H_0} &= 
  \norm{ \sum_{j=0}^{n}{\left( I - \Pi_h \right) \left(\sdx{j+1} - \sdx{j} \right)}}_{H_0} 
  \leq \sum_{j=0}^{n}{ \norm{\left( I - \Pi_h \right) \left(\sdx{j+1} - \sdx{j} \right)}_{H_0}} \\
  &= k\sum_{j=0}^{n}{ \norm{\left( I - \Pi_h \right) \left( \vect{b}^T \matH \sdX{j} \right)}_{H_0}} 
  \leq C\, k \sum_{j=0}^{n}{ \norm{\left( I - \underline{\Pi_h} \right) \matH \sdX{j}}_{\prodSpace{H_0}}},
\end{align*}
which then gives \eqref{eq:error_est_0}.
In order remove the requirement $\fdx{0}=\Pi_h\sdx{0}$, we just note that due to the discrete stability, proved in Lemma~\ref{lemma_discrete_stability}, perturbing the initial condition only adds a term
 $\norm{\sdx{0}-\fdx{0}}_{H_0} + \norm{\sdx{0}-\Pi_h \sdx{0}}_{H_0}$ to our final estimate.
A completely analogous argument, replacing $H_0$ with $H_1$ gives \eqref{eq:error_est_1}, as long as we make the stated additional assumptions.
\end{proof}

\begin{remark}
  Careful inspection of the proof shows that we did not in fact need the approximation property\eqref{eq:lemma:abstract_approx_with_projection-100} for arbitrary $u \in H_1$, but only
  for the semi-discrete solutions $\sdX{n}$ and $\matH \sdX{n}$. This insight may be useful when using non-uniform triangulations.
\eremk
\end{remark}

In the previous lemma, we reduced the approximation in each time step to the approximation properties of the Ritz projection $\Pi_h$.
The following lemma, which is a modified variation of C\'ea's lemma, tells us that this approximation is quasi-optimal in $H_1$.
\begin{lemma}
\label{lemma:quasi_optimality_of_ritz}
There exists a constant $C>0$ that depends only on the continuity of $\blfHStab{\cdot}{\cdot}$ and $c(\cdot,\cdot)$
and the inf-sup constant $\beta_{\widetilde{\mathbf{H}}}$ from \eqref{eq:abstract_setting_infsup} such that
for all $x \in H_2$ the following estimate holds:
\begin{align}
  \norm{ (I-\Pi_h) x}_{H_1}&\leq C \left(\inf_{x_h \in H_h} {\norm{x-x_h}_{H_1} } + \inf_{y_h \in V_0}{\norm{c(x-y_h, \cdot)}_{H_h'}} \right). 
\end{align}
\end{lemma}
\begin{proof}
  For any $x_h \in H_h$ and $y_h \in V_0$, condition \eqref{eq:abstract_setting_infsup} gives:
    \begin{align*}
      \norm{(I - \Pi_h ) x}_{H_1}&\leq  \norm{x - x_h}_{H_1} + \norm{ x_h - \Pi_h x}_{H_1} \\
      &\lesssim \norm{x - x_h}_{H_1} + \sup_{v_h \in H_h \setminus \{0\}} \frac{|\blfHStab{x_h - \Pi_h x}{v_h}|}{\norm{v_h}_{H_1}}\\
      &= \norm{x - x_h}_{H_1} + \sup_{v_h \in H_h \setminus \{0\}} \frac{|\blfHStab{x_h - x}{v_h} - c(x,v_h)|}{\norm{v_h}_{H_1}}  \\
      &= \norm{x - x_h}_{H_1} + \sup_{v_h \in H_h \setminus \{0\}} \frac{|\blfHStab{x_h - x}{v_h} - c(x - y_h,v_h)|}{\norm{v_h}_{H_1}} \\
      &\lesssim \norm{x-x_h}_{H_1} + \sup_{v_h \in H_h \setminus \{0\}} {\frac{|c(x - y_h,v_h)|}{\norm{v_h}_{H_1}}}.
    \end{align*}
    where we used that $c(y_h,v_h)=0$ for $y_h \in V_0$ and $v_h \in H_h$.
\end{proof}

\section{Convergence and stability of the fully discrete scheme}
\label{sect:convergence}
In this section, we will apply the abstract theory that we developed in Section~\ref{sect:abstract_analysis} to the Schrödinger equation.
It is easy to verify that the fully discrete problem, as described in Lemma~\ref{equivalent_formulation_lemma}, satisfies
Assumption~\ref{assumption:abstract_setting} with $H_0=\pairltwo$, $H_1=\pairhone$ and 
$H_h=\hat{H}(X_h,Y_h)=\{(v_h,v^*) \in X_h \times \hpfull{1} : \traceJump{v^*} = -\gamma^- v_h \land \gamma^- v^* \in Y_h^{\circ} \}$.
We have already seen that the stabilized Hamiltonian is elliptic if we assume $\alpha > 1 + \norm{\mathcal{V}}_{L^{\infty}(\R)}$. This implies the
inf-sup condition \eqref{eq:abstract_setting_infsup}.

In order to prove that the space $H_h$ inherits some important properties from $X_h$ and $Y_h$ we need the following well-known result:
\begin{proposition}[Extension operator, see \protect{\cite[Chap.~{VI.3}]{stein70}}]
\label{prop:extension_op}
  Let $\Omega \subseteq \R^n$ be a  Lipschitz domain.
  Then there exists a linear operator $\mathcal{E}$ with the following properties:
  \begin{itemize}
    \item for every $k \in \N_0$, ${\mathcal E}: H^k(\Omega) \rightarrow H^k(\R^d)$ 
      is a bounded linear operator:  
        $\displaystyle \norm{\mathcal{E} u}_{H^{k}(\R^d)}\leq C(k,\Omega)\norm{u}_{H^k(\Omega)},$
      \item $\mathcal{E}u$ is an extension of $u$, i.e., 
        $\displaystyle  \mathcal{E} u|_{\Omega} =u.$
  \end{itemize}
\end{proposition}

It is well-known that the time evolution of solutions to the Schrödinger equation corresponds to a unitary semigroup,
i.e., the $L^2$-norm of the initial condition is conserved. Since we are only considering a bounded subset of $\R^d$ we cannot hope to
retain that property, but we still have a slightly weaker result for the fully discrete scheme.
Similarly, it is known that the energy $\innerprod{L^2(\R^d)}{\mathbf{H}u (t)}{u(t)}$ is conserved over time. 
Our discrete system also almost retains this property. 

\begin{corollary}
  Let $\mathcal{V}$ be constant in time and bounded. Then
  the sequence of fully discrete solutions of Problem~\ref{problem:fully_discrete} is non-expansive:
  \begin{align*}
    \ltwonormint{\fdu{n}}\leq \ltwonormint{\fdu{0}}, \\
    \honenormint{\fdu{n}}\leq C \honenormint{\fdu{0}}.
  \end{align*}

  In the case of RK-methods that satisfy $\abs{R(\ii t)}=1$ for all $t \in \R$, the damping that appears in the previous inequalities can
  be controlled by adding additional (computable)-terms to get a ``mass'' and ``energy''-conserving scheme, i.e.,
  \begin{align*}
    \ltwonormint{\fdu{n}}^2 + \ltwonormfull{\fdustar{n}}^2&= \ltwonormint{\fdu{0}}^2, \\
    \mathbf{H}(\fdu{n},\fdu{n}) + \mathbf{H}(\fdustar{n},\fdustar{n}) &= \mathbf{H}(\fdu{0},\fdu{0}),
  \end{align*}
  with energy $\mathbf{H}(u,u):=\ltwonormfull{\nabla u}^2 + \ltwoprodfull{\mathcal{V} u}{u}$.
  \begin{proof}    
    We apply the discrete stability lemma (Lemma~\ref{lemma_discrete_stability}) to the equivalent formulation (\ref{equivalent_formulation_lemma_eqn}).
    Since $\fdustar{0}=0$ we directly get the stated results.
  \end{proof}
\end{corollary}

We are now interested in an estimate for the convergence rate of the fully discrete scheme.
We will again use the equivalent form from Lemma~\ref{equivalent_formulation_lemma} and apply the abstract theory of Section~\ref{sect:abstract_analysis}.
In order to do so, we need to verify Assumption~\ref{assumption:semidiscrete_problem}. 
The pairs $(\sdU{n}, \sdUstar{n})$ and $(\fdU{n},\fdUstar{n})$ satisfy similar equations that 
differ, however, in the test functions, namely, $\hat H(H^1(\Omega),H^{-1/2}(\Gamma))$ and 
$H_h = \hat H(X_h,Y_h)$. 
For $(V_h,V^*) \in \prodSpace{H_h}$ one has $\gamma^-V_h = -\traceJump{V^*}$; furthermore, using 
$\lambda^n = -\normalJump{\sdUstar{n}}$ and $\sdUstar{n}|_{\Omega} = 0$ (cf.~Lemma~\ref{equivalent_formulation_lemma})
we assert by integration by parts 
%
\begin{align}
  \blfB{\sdU{n}}{\sdUstar{n}}{V_h}{V^*}  + \bdryprodp{k \sdl{n}}{ \gamma^- V^*} &= \blfFp{\sdu{n} \,\vecrhs}{\sdustar{n} \,\vecrhs}{V_h}{V^*}
  \quad \forall (V_h,V^*) \in \prodSpace{H_h}.
\end{align}
Thus we are in the setting of Assumption~\ref{assumption:semidiscrete_problem}, if we define 
\begin{align*}
  c\left((u,u_*),(v,v_*) \right):=\bdryprod{ \normalJump{u_*}}{ \gamma^- v^*}.
\end{align*}

In Lemmas~\ref{lemma:abstract_approx_with_projection} and \ref{lemma:quasi_optimality_of_ritz}, the approximation problem is reduced to
the question of best approximation in the space $H_h$. Relating this to the properties of the spaces $X_h$ and $Y_h$ is the subject to the next lemma.
\begin{lemma}
\label{lemma:approximation_ritz_schroed}
  There exists a constant $C > 0$  that depends only on $\Omega$ such that
  for every $\underline{v}=(v,v_*) \in \hpint{1} \times \hpfull{2}$ with $\traceJump{v_*} = -\gamma^{-} v$ 
and $\gamma^- v_* \in Y_h^\circ$,
  the following approximation property holds for  $s=0,1$:
  \begin{align*}
    \inf_{\underline{v}_h \in H_h}{\norm{\underline{v} - \underline{v}_h}_{\mathcal{X}^s}} 
    &\leq C \inf_{x_h \in X_h}{\hpnormint{s}{v - x_h}}, \\
    \inf_{y_h \in V_0}{\norm{c(\underline{v}-y_h, \cdot)}_{H_1'}} &\leq C \inf_{y_h \in Y_h}{ \hpbdrynorm{-1/2}{\normalJump{v_*} -y_h}}.
  \end{align*}
  \begin{proof}
    Let $x_h \in X_h$ be arbitrary, and set $x_*:=v_* + \delta_*$ where $\delta_*=\mathcal{E}(x_h - v)$,
    with the extension operator of Proposition~\ref{prop:extension_op} in $\Omega^+$ and $\delta_* = 0$ in $\Omega$.
    Since $\traceJump{v_*} = -\gamma^{-} v$ and $\gamma^- v_* \in Y_h^\circ$, we get that
    $\underline{x}:=(x_h,x_*) \in H_h$.
    From the continuity of the extension operator $\mathcal{E}$ we get
    $\honenormfull{\delta_*} \leq C \honenormint{v-x_h }$ and $\ltwonormfull{\delta_*} \leq \ltwonormint{v-x_h}$.

    For the difference $\underline{v}-\underline{x}$ we get for $s=0$, $1$:
    \begin{align*}
      \norm{\underline{v}-\underline{x}}^2_{\mathcal{X}^s}&= 
      \norm{v-x_h}^2_{H^s(\Omega)} +  \norm{v_*-v_* -\delta}^2_{\hpfull{s}} \\
      &\leq \norm{v-x_h}^2_{H^s(\Omega)} + C \norm{v - x_h}^2_{H^s(\Omega)}.
    \end{align*}

    We are left with estimating the contribution due to $c(\cdot,\cdot)$.
    Let $(w,w_*) \in H_h$ be arbitrary, and let $\xi_h \in Y_h$ be arbitrary. 
    Since $\gamma^- w_* \in (Y_h)^\circ \subset H^{1/2}(\Gamma)$, we may 
    choose a lifting $y_*$ to the full space such that $\normalJump{y_*}= \xi_h \in Y_h$.
    We get $c((0,y_*),(w,w_*)) = \bdryprod{\normalJump{y_*}}{\gamma^- w_*}=0$, and therefore $(0,y_*) \in V_0$, as defined in Assumption~\ref{assumption:semidiscrete_problem}. Since taking traces is continuous in
   $\hpfull{1}$, we get
   \begin{align*}
     \inf_{y_h \in V_0}\norm{c(\underline{v}-y_h,\cdot)}_{(\pairhone)'}&\stackrel{y_h = (0,y_*)}{\lesssim }
\inf_{\xi_h \in Y_h} \norm{\normalJump{v_*} - \xi_h}_{\hpbdry{-1/2}}.
\qedhere
   \end{align*}
%
  \end{proof}
\end{lemma}

This allows us to give an estimate for the error due to spatial discretization:
\begin{theorem}
  \label{thm_fd_err_est}
  Let $\mathcal{V} \in L^{\infty}(\R^d)$.
  Then, there exists a constant $C > 0$ that depends only on $\mathcal{V}$, $\Omega$, and the Runge-Kutta method 
  (namely, $A$ and $b$),
  such that for all closed subspaces $X_h \subseteq \hpint{1}$, $Y_h \subseteq \hpbdry{-1/2}$,
  for all $n \in \N$, and for all $k >0 $,
  the following estimate holds:
  
  \begin{align*}
    \ltwonormint{ \sdu{n} - \fdu{n}} &\leq C\; k\, \sum_{j=0}^{n-1}{
    \left(
      \inf_{x_h \in \prodSpace{X_h}} \honenormintp{\matH \sdU{j} - x_h} +  \inf_{x_h \in \prodSpace{X_h}} \honenormintp{\sdU{j} - x_h} \right)} \\
  &+  C \,k \; \sum_{j=0}^{n-1}{ \left(
      \inf_{y_h \in \prodSpace{Y_h}} \hpbdrynormp{-1/2}{\partial^+_n \matH \sdU{j} - y_h} +  \inf_{y_h \in \prodSpace{Y_h}} \hpbdrynormp{-1/2}{\partial^+_n \sdU{j} - y_h} 
    \right)}.
  \end{align*}
  If we assume that $k$ and $X_h$ satisfy:
  \begin{align}
    \label{eq:approximation_assumption}
    \inf_{w_h \in X_h}{ \norm{u -w_h}_{L^2(\Omega)}} &\leq C_{approx} k^{1/2} \norm{u}_{H^1(\Omega)},
  \end{align}
  then the estimate holds in the $H^1$-norm (the constant now additionally depends on $C_{approx}$):
  \begin{align*}
    \hpnormint{1}{ \sdu{n} - \fdu{n}} &\leq C\; k\, \sum_{j=0}^{n-1}{
    \left(
      \inf_{x_h \in \prodSpace{X_h}} \honenormintp{\matH \sdU{j} - x_h} +  \inf_{x_h \in \prodSpace{X_h}} \honenormintp{\sdU{j} - x_h} \right)} \\
  &+  C \,k \; \sum_{j=0}^{n-1}{ \left(
      \inf_{y_h \in \prodSpace{Y_h}} \hpbdrynormp{-1/2}{\partial^+_n \matH \sdU{j} - y_h} +  \inf_{y_h \in \prodSpace{Y_h}} \hpbdrynormp{-1/2}{\partial^+_n \sdU{j} - y_h} 
    \right)}.
\end{align*}
  
  \begin{proof}
    We want to apply Lemma~\ref{lemma:abstract_approx_with_projection}. We have already seen that
    we can reduce the approximation requirements of the constrained space $H_h$ to $X_h$ and $Y_h$ via Lemma~\ref{lemma:approximation_ritz_schroed}.
    By \eqref{rk_approx_whole_space_1}, the semi-discrete full-space solutions satisfy
    ${(-\ii A^{-1}+ k \matH)U^n = u_n \vecrhs}$. This means that the definition of $\matH \sdU{n}$ in 
    Lemma~\ref{lemma:abstract_approx_with_projection} coincides with the pointwise application of the 
    Hamilton operator to the semi-discrete functions $\sdU{j}$
    (up to identifying the global function with the pair $(\sdU{j}|_{\Omega}, \sdU{j}|_{\Omega^+})$).
    Using \eqref{eq:error_est_1}, Lemma~\ref{lemma:quasi_optimality_of_ritz} and applying Lemma~\ref{lemma:approximation_ritz_schroed} then gives the stated result.
  \end{proof}
\end{theorem}

\section{The semi-discrete problem}
\label{sect:semi_discrete}
In the last theorem we showed that our fully discrete scheme gives quasi-optimal convergence to 
the semi-discrete solution. In
order to estimate the error for the exact solution we will need some properties of the semi-discrete problem.
We only consider the simplest case of potentials that are constant in time, since they allow us to use the theory of $C_0$-semigroups.

First we show some approximation properties:
\begin{theorem}
  \label{sd_approx}
  Assume that a Runge-Kutta method of order $q$ is used. 
  Let $\sdu{0}$ be sufficiently smooth. Then the following estimates hold
  for all $n k \leq T$:
  \begin{align*}
    \ltwonorm{\sdu{n}- u(n k)} &\leq C T k^q \ltwonorm{\mathbf{H}^{q+1} \sdu{0}} , \\
    \honenorm{\sdu{n}- u(n k)} &\leq C T k^q \left(\ltwonorm{\mathbf{H}^{q+2} \sdu{0}} + \ltwonorm{\mathbf{H}^{q+1} \sdu{0}} \right).
  \end{align*}
  \begin{proof}
    We use some results from the theory of rational approximations of semigroups. \cite[Theorem 4]{brenner_thomee_rat_approx_semigroup} states that:
    \begin{align*}
      \ltwonorm{\sdu{n}-u(n k)} &\leq C T k^q \ltwonorm{\mathbf{H}^{q+1} \sdu{0}}.
    \end{align*}
    Since $\mathbf{H}$ commutes with both the time evolution and the application of the Runge-Kutta method, this also gives
    \begin{align*}
      \ltwonorm{\mathbf{H}\sdu{n}- \mathbf{H}u(n k)} &\leq C T k^q \ltwonorm{\mathbf{H}^{q+2} \sdu{0} }.
    \end{align*}
    Thus it is easy to see that
    \begin{align*}
      \honenorm{\sdu{n}- u(n k)} &\leq C T k^q \left(\ltwonorm{\mathbf{H}^{q+2} \sdu{0}} + \ltwonorm{\mathbf{H}^{q+1} \sdu{0}} \right).
\qedhere
    \end{align*}
  \end{proof}
\end{theorem}

Since the convergence rates depend on the approximation quality for the semi-discrete stages we need some 
{\sl a priori} estimates.
\begin{lemma}
  \label{sd_apriori}
    Let $x \mapsto \mathcal{V}(x)$ be sufficiently smooth. Let
    $\sdu{0} \in H^{s}(\R^d)$ for some $s \in \R$, $s\ge 0$. Then there exists a constant $C_s$ that only depends on $\mathcal{V}$ and $s$ such that    
    \begin{align*}
      \hpnorm{s}{\sdu{n}} &\leq C_s \hpnorm{s}{\sdu{0}} 
    \end{align*}
    \begin{proof}
      Denote by $R(\ii k \mathbf{H})^n$ the solution operator
      $\sdu{0} \mapsto \sdu{n}$. We use that the time stepping commutes with $\mathbf{H}$. Therefore we get for $\ell \in \N$, $\ell \geq s/2$:
      \begin{align*}
        \mathbf{H}^{\ell} \sdu{n} &= \mathbf{H}^{\ell} R(\ii k \mathbf{H})^n \sdu{0} = R(\ii k \mathbf{H})^n \mathbf{H}^{\ell} \sdu{0}.
      \end{align*}
      Lemma~\ref{eq:lemma_discrete_stability_l2} gives that $\norm{R(\ii k \mathbf{H})^n}_{L^2(\R^d)  \to L^2(\R^d)} 
\leq 1$, and therefore,
      as long as $\sdu{0}$ is smooth enough that $\sdu{n}$ is smooth as well,
      and the norms are uniformly bounded by $\ltwonorm{\mathbf{H}^l \sdu{n}} \leq \ltwonorm{\mathbf{H}^l \sdu{0}}$. Since the potential $\mathcal{V}$ is
      assumed to be smooth we can estimate the norm of $-\laplace^l \sdu{n}$ by $\norm{\mathbf{H}^l \sdu{n}} + \text{lower order terms}$.
      Since we are working on the full space $\R^d$,  we can use Fourier techniques to bound the full $H^{2l}$ norm by $\ltwonorm{-\laplace^l \sdu{n}}$.
      This gives that the operator $R(\ii k \mathbf{H})^n$ is bounded in $L^2(\R^d) \to L^2(\R^d)$ and also in $H^{2l}(\R^d) \to H^{2l}(\R^d)$, uniformly
      with respect to $n$.
      By interpolation, we also get the uniform bound in $H^s(\R^d)$.
    \end{proof}
\end{lemma}

We need the smoothness of the internal stages. Since we already have smoothness of the semi-discrete solutions, and thus the
right-hand side of the defining equation of the stage vectors, this is a simple consequence of elliptic regularity:
\begin{corollary}
  \label{sd_stages_apriori}
  Let $\mathcal{V}$ be sufficiently smooth. Let
  $\sdu{0} \in H^{s}(\R^d)$ for some $s \in \R$, $s \ge 0$. Then there exists a constant $C>0$ that 
  depends only on $\mathcal{V}$ and $s$, such that
  \begin{align*}
    \hpnorm{s}{\sdU{n}} \leq C \hpnorm{s}{\sdu{0}}.
  \end{align*}
  \begin{proof}
    For $\ell \in \N$, $\ell \geq s/2$, $\matH^\ell \sdU{n}$ solves the equation
    \begin{align*}
      \left( -\ii A^{-1}+ k \matH \right) \matH^{\ell} \sdU{n} &= \mathbf{H}^{\ell} \sdu{n} \vecrhs.
    \end{align*}
    By Lemma~\ref{lemma_solve_matrix_eqn} and Lemma~\ref{sd_apriori} we can bound the $L^2$ norms as
    \begin{align*}
      \ltwonorm{\matH^{\ell} \sdU{n}} &\leq C \ltwonorm{\mathbf{H}^{\ell} \sdu{n}} \leq C \hpnorm{2\ell}{\sdu{0}}.
    \end{align*}
    This allows us to estimate, again assuming smoothness of the potential, the full $H^{2\ell}$ norm and via interpolation, the $H^{s}$ norm.
  \end{proof}
\end{corollary}

\section{Full error estimate}
\label{sec:full_error_est}
All that remains is to estimate the error between the fully discrete approximation and the exact solution.
We assume $\Gamma$ to be piecewise smooth and write 
$\hppwbdry{s}$ for the space of functions that are  in $H^{s}(\Gamma_i)$ for each boundary piece $\Gamma_i$
 (see \cite[Definition 4.1.48]{book_sauter_schwab}). The convergence of the fully discrete scheme is
summarized in the following theorem.
\begin{theorem}
\label{thm:full_error_est}
  Let $\Gamma$ be piecewise smooth,
  and denote by $q$ the order of the Runge-Kutta method used.  
  Assume the following approximation properties:
  \begin{subequations}
    \label{eq:approx_spaces_assumption}
    \begin{align}
      \inf_{x_h \in X_h} \hpnormint{s}{ u - x_h} &\leq C h_1^{p_1 + 1-s} \hpnormint{p_1+1}{u}  \quad \quad \forall u \in \hpint{p_1+1}, \\
      \inf_{y_h \in Y_h} \hpbdrynorm{-1/2}{ \lambda - y_h} &\leq C h_0^{p_0 + 3/2} \hppwbdrynorm{p_0+1}{\lambda} \quad \quad \forall \lambda \in \hppwbdry{p_0 +1},
    \end{align}
  \end{subequations}
  for $s \in \{0,1\}$ and parameters $h_0>0$, $h_1>0$ and $p_0$, $p_1 \in \N_0$, with constants that depend only on $\Omega$ and $p_0,p_1$.

  Let $\sdu{0} \in H^{\max(p_1+1,p_0+5/2)}(\R^d)$ and let $\mathcal{V}$ be sufficiently smooth (i.e., such that the 
  semi-discrete sequences satisfy 
  $\sdU{n}, \matH{\sdU{n}} \in \prodSpace{H^{p_1+1}(\R^d)}$ and $\partial_n^- \sdU{n}, \partial_n^- \matH \sdU{n} \in \prodSpace{\hppwbdry{p_0+1}}$,
  see Corollary~\ref{sd_stages_apriori}).
  Then there exists a constant depending on $\Omega$, the Runge-Kutta method (i.e., $A$ and $b$), $\mathcal{V}$,
  $p_0$, $p_1$ and $\sdu{0}$, but not on $k$, $n$, $h$, or $T$ such that:
  \begin{align*}
    \ltwonormint{\fdu{n}-u(n k)}\leq C T \left(   h_{1}^{p_1} +   h_{0}^{p_0+3/2} +  k^q \right). 
  \end{align*}
  If we assume that the approximation assumption $\inf_{w_h \in X_h}{ \norm{u -w_h}_{L^2(\Omega)}} \leq C_{approx} k^{1/2} \norm{u}_{H^1(\Omega)},$
  i.e., $h \lesssim k^{1/2}$ (see \eqref{eq:approximation_assumption}) holds, then
  \begin{align*}
    \honenorm{\fdu{n}-u(n k)}\leq C T \left(   h_{1}^{p_1} +   h_{0}^{p_0+3/2} +  k^q \right).
  \end{align*}
  \begin{proof}
    We only show the $H^1$ bound, the $L^2$ one follows along the same lines. We use the triangle inequality to get:
    \begin{align*}
      \honenormint{\fdu{n}-u(n k)} &\leq \honenormint{\fdu{n}-\sdu{n}} + \honenormint{\sdu{n}-u(n k)} .
    \end{align*}
    The first term can be estimated by Theorem~\ref{thm_fd_err_est}. 
    Using the regularity results and the approximation properties from the finite element spaces we get
    \begin{align*}
      \honenormint{\fdu{n}-\sdu{n}}\leq C T  \left( h^{p_1}  +   h^{p_0+3/2} \right),
    \end{align*}
    where the constants depend on $\sdu{0}$ but not on $n$ or $k$.
    The second term can be controlled via the approximation property of the semi-discrete solution from Theorem~\ref{sd_approx}:
    \begin{align*}
      \honenormint{\sdu{n}-u(n k)}&\leq C T k^q.
\qedhere
    \end{align*}
  \end{proof}
\end{theorem}

\begin{remark}
  The assumptions on the FEM/BEM spaces of \eqref{eq:approx_spaces_assumption} are
  satisfied, for example, for standard continuous piecewise polynomial discretizations of degree $p_1$ to discretize $X_h$ on a quasiuniform mesh and
  discontinuous polynomial boundary elements of degree $p_0$ to discretize $Y_h$  (see \cite[Theorem 4.3.20,Theorem 4.3.22]{book_sauter_schwab}).  
\eremk
\end{remark}

\subsection{A better $H^1$ and $H^{-1/2}$-estimate}
\label{sect:better_h1_est}
The  requirement on the mesh size for the $H^1$-estimate 
in Theorem~\ref{thm:full_error_est} is somewhat artificial. In order to get rid of it,
we first bound a sequence of finite difference quotients of the spatial discretization error in the $L^2$-norm,
and then use the definition of the stage vectors to leverage this ``time-regularity'' for stronger spatial norms.
\begin{lemma}
\label{lemma:time_stepping_difference}
Let $\fdX{n}$, $\fdx{n}$, and $F^n$ be defined as in Assumption~\ref{assumption:abstract_setting} and assume $\fdx{0}=0$.
  Consider the sequences $y^0:=0$, $Y^n:=k^{-1}A^{-1}\left( X^n - x^{n}\ones\right)$  and $y^{n+1}:=R(\infty) y_n + b^{T} A^{-1} Y^n$.
  Let $\left(\Theta^n\right)_{n \in \N_0} \subset \prodSpace{H_0}$ be the sequence of $\prodSpace{H_0}$-functions defined as the inverse Z-transform of
  \begin{align*}
    \widehat{\Theta}(z):&=\frac{\delta(z)}{k} \widehat{F}(z).
  \end{align*}
  Then the sequence $y^n$, $Y^n$ solves the following equations for all $n \in \N$:
\begin{subequations}
    \label{eq:time_stepping_difference}
  \begin{align}
   B(Y^n,V_h )&= \innerprod{\prodSpace{H_0}}{y^n \vecrhs}{V_h} + \innerprod{\prodSpace{H_0}}{\Theta^n}{V_h} \quad \quad \forall V_h \in H_h, \\
    y^{n+1}&=R(\infty) y^n + b^T A^{-1} Y^n.
  \end{align}
\end{subequations}
  This implies the following {\sl a priori} estimates:
  \begin{align}
    \label{eq:apriori_est_differentiated}
    \norm{y^n}_{H_0}&\leq C \sum_{j=0}^{n-1}{\norm{\Theta^j}_{\prodSpace{H_0}}}.
  \end{align}
  We write $[\partial_t^k x ]^n:=y^n$ (this notation can be justified by taking the Z-transform to establish the equivalence to the definition via operator calculus notation for $K(s)=s$).
\end{lemma}
\begin{proof}
  We show that the sequences $Z^n$, $z^n$, defined as the solutions to \eqref{eq:time_stepping_difference}, 
and the sequence of the functions $Y^n$, $y^n$, defined in the statement of the lemma, have the same Z-transforms. 
Proceeding as in the proof of Lemma~\ref{z_transform_lemma}, it is easy to see that
  $\widehat{Z}$ solves:
  \begin{align*}
    \innerprod{\prodSpace{H_0}}{\frac{-\ii \delta(z)}{k} \widehat{Z}}{V_h} + \matH\left(\widehat{Z},V_h\right) &= \innerprod{\prodSpace{H_0}}{\widehat{\Theta}}{V_h} \quad\quad \forall V_h \in \prodSpace{H_h}.
  \end{align*}
  
  Analogously, we get that the $Z$- transform of $\fdX{n}$ solves:
  \begin{align*}
    \innerprod{\prodSpace{H_0}}{\frac{-\ii \delta(z)}{k} \widehat{X}}{V_h} + \matH\left(\widehat{X},V_h\right) &= \innerprod{\prodSpace{H_0}}{\widehat{F}}{V_h} \quad\quad \forall V_h \in \prodSpace{H_h}. 
  \end{align*}
  By \eqref{eq:_zy_in_terms_of_stages}, we also have $\widehat{x}(z)= \left(z^{-1} - R(\infty)\right)^{-1} \vect{b}^T A^{-1} \widehat{X}(z)$. By the definition of $\delta(z)$, this can be rewritten as
  $\widehat{x}\,\ones = \widehat{X} - A \delta(z) \widehat{X}$.

  Inserting the definition of $\widehat{Y}$, this implies for $V_h \in \prodSpace{H_h}$:
  \begin{align*}
    \innerprod{\prodSpace{H_0}}{\frac{-\ii \delta(z)}{k} \widehat{Y}}{V_h} + \matH\left(\widehat{Y},V_h\right)
    &=\innerprod{\prodSpace{H_0}}{\frac{-\ii\delta(z)}{k} k^{-1}A^{-1} \left(\widehat{X} - \widehat{x}\,\ones\right)}{V_h} + \matH\left(k^{-1}A^{-1} \left(\widehat{X} - \widehat{x}\,\ones\right),V_h\right) \\
    &=\innerprod{\prodSpace{H_0}}{\frac{ -\ii \delta(z)}{k} k^{-1}A^{-1} \left( A\delta(z)\right)\widehat{X}}{V_h} + \matH\left(k^{-1}A^{-1} \left( A\delta(z)\right)\widehat{X},V_h\right) \\
    &=\innerprod{\prodSpace{H_0}}{\frac{-\ii\delta(z)}{k} \widehat{X}}{\frac{\delta(z)^T}{k}V_h} + \matH\left(\widehat{X},\frac{\delta(z)^T}{k}\,V_h\right) \\
    &=\innerprod{\prodSpace{H_0}}{\widehat{F}}{\frac{\delta(z)^T}{k} V_h} 
    =\innerprod{\prodSpace{H_0}}{\widehat{\Theta}}{ V_h}.
  \end{align*}
  The stability estimate~\eqref{eq:apriori_est_differentiated} is then a direct corollary of Lemma~\ref{lemma_discrete_stability}.
\end{proof}

We can now improve the results of Theorem~\ref{thm:full_error_est}, assuming some additional regularity of the initial condition, and  an additional stability condition for the method.
\begin{theorem}
\label{thm:full_error_est_refined}
  Let $\Gamma$ be piecewise smooth.
  Assume $\abs{R(\infty)}<1$ and denote by $q$ the order of the Runge-Kutta method used.
  Let $X_h$, $Y_h$ satisfy the approximation properties \eqref{eq:approx_spaces_assumption}.

  Let $\sdu{0} \in H^{\max(p_1+3,p_0+7/2)}(\R^d)$ and let $\mathcal{V}$ be sufficiently smooth (i.e., such that the 
semi-discrete sequences satisfy 
  $\sdU{n}, \matH{\sdU{n}},\matH^2 \sdU{n} \in  \prodSpace{H^{p_1+1}(\R^d)}$ and $\partial_n^- \sdU{n}, \partial_n^- \matH \sdU{n},\partial_n^- \matH^2 \sdU{n} \in \prodSpace{\hppwbdry{p_0+1}}$,
  see Corollary~\ref{sd_stages_apriori}). 

  Then, there exists a  constant $C>0$ depending on $\Omega$, the Runge-Kutta method (i.e., $A$ and $b$), $\mathcal{V}$,
  $p_0$, $p_1$ and $\sdu{0}$, but not on $k$, $n$, $h$ or $T$ such that:
  \begin{align*}
    \honenorm{\fdu{n}-u(n k)}&\leq C T \left(   h_{1}^{p_1} +   h_{0}^{p_0+3/2} +  k^q \right). \\
    \norm{\fdl{n}-\partial_n u(n k)}_{\hpbdry{-1/2}}&\leq C T \left(   h_{1}^{p_1} +   h_{0}^{p_0+3/2} +  k^q \right).
  \end{align*}
  
  Compared to Theorem~\ref{thm:full_error_est} this means we do not have any mesh size restriction and obtain
  an error estimate for $\lambda$.
\end{theorem}
\begin{proof}
  We proceed analogously to the proof of Theorem~\ref{thm:full_error_est} and use the triangle inequality to estimate:
  \begin{align*}
    \honenormint{\fdu{n}-u(n k)} &\leq \honenormint{\fdu{n}-\sdu{n}} + \honenormint{\sdu{n}-u(n k)} .
  \end{align*}
  The second term can be estimated via the approximation property of the semi-discrete solution from Theorem~\ref{sd_approx}:
  \begin{align*}
    \honenormint{\sdu{n}-u(n k)}&\leq C T k^q.
  \end{align*}
    
  For the estimates of the first term, we go back to the proof of Lemma~\ref{lemma:abstract_approx_with_projection}, and again consider the difference
  $e^n:=\fdu{n} - \Pi_h \sdu{n}$, $E^n:= \fdU{n} - \Pi_h \sdU{n}$. Assume for the moment that $u^0_h=\Pi_h u^0$.
  From Lemma~\ref{lemma:time_stepping_difference} and the stability of solving \eqref{eq:time_stepping_difference} as shown in Lemma~\ref{lemma_solve_matrix_eqn}, we obtain 
  \begin{align*}
    \norm{k^{-1}A^{-1} \left(E^n - e^n\ones\right)}_{\mathcal{X}_0}
    &\leq C \sum_{j=0}^{n}{\norm{\Theta^j}_{\mathcal{X}_0}},
  \end{align*}
  where $\Theta^j$ are defined so that $\widehat{\Theta}=\frac{\delta(\cdot)}{k} \widehat{\Xi}$ and $\Xi^j:=k(I-\Pi_h)\left( \matH U^j + \alpha U^j\right)$
  are the consistency  errors from~\eqref{eq:def_consistency_errors} . We also write $\xi^j = k (I - \Pi_h) \left(\mathbf{H} u^j + \alpha u^j\right)$. 
  Since the sequence $U^j$ originates from a Runge-Kutta time stepping, it is easy to compute $\Theta^j$.
  We claim:
  \begin{align}
    \label{eq:def_theta_j}
    \Theta^j&=  k^{-1}A^{-1}(\Xi^j - \xi^j \ones) + R(\infty)^j  k^{-1}A^{-1} \xi^0 \ones.            
  \end{align}
  This can be seen  by taking the Z-transform of the right-hand side, analogously to the proof of Lemma~\ref{z_transform_lemma}, and noting that $u_0 \neq 0$ so that an additional term appears. 
  This means, writing $\mathcal{Z}$ for the $Z$-transform, 
  \begin{align*}
    \mathcal{Z}\left[k^{-1}A^{-1}(\Xi^j - \xi^j \ones) + R(\infty)^j k^{-1} A^{-1} \xi^0 \ones\right]
    &=k^{-1}A^{-1} \widehat{\Xi} - k^{-1}A^{-1} \widehat{\xi}\ones + \frac{1}{1-R(\infty)z} k^{-1}A^{-1} \xi^0 \ones \\
    &=\frac{\delta(z)}{k} \widehat{\Xi}, 
%
  \end{align*}
  where, in the last step we used the equality $\widehat{\xi}=\left( z^{-1} - R(\infty)\right)^{-1} b^T A^{-1} \widehat{\Xi} + \left(1- R(\infty)z\right)^{-1} \xi^0 $,
  which follows  
  analogously to \eqref{eq:_zy_in_terms_of_stages} ($\xi^j$ and $\Xi^j$ satisfy the same relation $\xi^{j+1}=R(\infty) \xi^j + b^T A^{-1} \Xi^j$ as the usual Runge-Kutta approximations
  due to the linearity of $\Pi_h$ and $\matH$).
  
  Inserting the definition of $\Xi^j$ in~\eqref{eq:def_theta_j} and then the equation for the semi-discretization for the difference $U^j - u^j\ones$ gives:  
  \begin{align*}
    \Theta^j&=A^{-1} (I-\Pi_h)\left( \left(\matH  + \alpha\right)\left( U^j - u^j\ones\right)\right)  
    + A^{-1} R(\infty)^j(I-\Pi_h)\left( \mathbf{H} u_0 +\alpha u_0\right)\ones  \\
    &= -\ii k  (I-\Pi_h)\left( \left(\matH  + \alpha\right) \matH U^j\right) + A^{-1}\ones (I-\Pi_h) R(\infty)^{j}\left( \mathbf{H} u_0 +\alpha u_0\right).
  \end{align*}

  The first term is of the right order already, as we can bound the sum with the factor of $k$. 
  We use the formula for the $n$-th term of the geometric series to to estimate
  \begin{align*}
    \sum_{j=0}^{n}{\norm{ A^{-1} \ones (I-\Pi_h) R(\infty)^{j}\left( \mathbf{H} u_0 +\alpha u_0\right)}_{\mathcal{X}_0}}
    &=\frac{1- \abs{R(\infty)}^{n}}{1-\abs{R(\infty)}} \norm{ A^{-1} \ones (I-\Pi_h) \left( \mathbf{H} u_0 +\alpha u_0\right)}_{\mathcal{X}_0}  \\
    &\lesssim  \norm{(I-\Pi_h) \mathbf{H} u_0}_{L^2(\Omega)} +\alpha\norm{ (I-\Pi_h)u_0}_{L^2(\Omega)}
  \end{align*}
  since we assumed $\abs{R(\infty)}<1$.
  
   Via the approximation properties of the spaces and the Ritz projector we arrive at:
  \begin{align}
    \label{eq:apriori_stages_differentiated}
    \norm{k^{-1}A^{-1} \left(E^n - e^n\ones\right)}_{\mathcal{X}_0}
    &\leq C k \sum_{j=0}^{n}{\left(h^{p_1} + h^{p_0+3/2}\right)}.
  \end{align}

  Analogously we can use \eqref{eq:error_eqn_with_consistency_errors} and the discrete stability of Lemma~\ref{lemma_discrete_stability} to 
    bound
    \begin{align}
      \label{eq:aprioir_stages_l2}
      \norm{E^n}_{\mathcal{X}_0}
      &\leq C \sum_{j=0}^{n}{ \norm{\Xi^j}_{\mathcal{X}_0}} \leq C  k \sum_{j=0}^{n}{\left(h^{p_1} + h^{p_0+3/2}\right)}.
    \end{align} 
  The weak form of the stage vector equation is:
  \begin{align}
    \label{eq:full_approximation:weak_form_stage_vector}
    \matH(E^n,V_h) &= k^{-1}\innerprod{\mathcal{X}_0}{-\ii A^{-1} \left(E^n - e^n\ones\right)}{V_h} + \innerprod{\mathcal{X}_0}{\Xi^j}{V_h}.
  \end{align}
  Using $V_h:=E^n$ as a test function  and applying the Cauchy-Schwarz inequality we get via~\eqref{eq:apriori_stages_differentiated} and \eqref{eq:aprioir_stages_l2}:
   $|\matH(E^n,E^n)| \leq C  \left[k\sum_{j=0}^{n}{\left(h^{p_1} + h^{p_0+3/2}\right)}\right]^2$.
  Adding another $L^2$ term to compensate for $\mathcal{V}(\cdot)$ gives:
  \begin{align*}
    \norm{E^n}_{\mathcal{X}^1} \leq Ck \sum_{j=0}^{n}{\left(h^{p_1} + h^{p_0+3/2}\right)}.
  \end{align*}

  The triangle inequality $\norm{ \fdu{n} - \sdu{n}}_{\mathcal{X}^1} \leq \norm{ \fdu{n} - \Pi_h \sdu{n}}_{\mathcal{X}^1} + \norm{ \Pi_h \sdu{n} - \sdu{n}}_{\mathcal{X}^1}$ and the approximation properties of $\Pi_h$ then give the stated result. 
  
  For the case of $u^0_h\neq \Pi_h u^0$, we just note that the discrete time-stepping is stable with regard to perturbations of the initial conditions via Lemma~\ref{lemma_discrete_stability},
  thus this only implies another error term of order $\norm{u^0 - u^0_h}_{\mathcal{X}_1}$.

  To get the $H^{-1/2}$ estimate, we use  $V_h = (0, V_*)$ with $V_* \in C_{0}^{\infty}(\R^d \setminus \Gamma)$ as a test function in~\eqref{eq:full_approximation:weak_form_stage_vector}, and get the pointwise equality:
  $\matH E^n_* = -\ii k^{-1} A^{-1} \left(E^n_* - e^n_*\ones\right) + \Xi^{n,*}$. (Here $E^n_*$ denotes the second component of the error $E^n=(E^n_h,E^h_*)$, and analogously for $e^n_*$ and $\Xi^n_*$.) 
  Using test functions in $C_0^{\infty}(\R^d\setminus \Gamma)$ in the definition of the Ritz projector \eqref{eq:ritz_projection_eq} gives
  $\widetilde{\matH} \left[\Pi_h \sdUstar{n}\right]=\widetilde{\matH} \sdUstar{n}$ pointwise in $\R^d \setminus \Gamma$. 
  Therefore we can write:
  \begin{align*}
    \matH \left(\sdUstar{n} - \fdUstar{n} \right) &=\widetilde{\matH}\left(\sdUstar{n} - \fdUstar{n} \right) - \alpha \left(\sdUstar{n} - \fdUstar{n} \right) \\
                                          &= \widetilde{\matH} (\Pi_h \sdUstar{n} - \fdUstar{n}) - \alpha \left(\sdUstar{n} - \fdUstar{n} \right),
  \end{align*}
  where $\fdUstar{n}$ denotes the second component of the fully discrete solution \eqref{equivalent_formulation_lemma_eqn}.
  This in turn implies the estimate
  \begin{align*}
    \norm{ \matH \left(\sdUstar{n} - \fdUstar{n} \right)}_{L^2(\R^d \setminus \Gamma)}
    &\lesssim \norm{ -\ii k^{-1} A^{-1} \left(E^n_* - e^n_*\ones\right) + \Xi^{n}_{*} - \alpha \left(\sdUstar{n} - \fdUstar{n} \right)}_{L^2(\R^d \setminus \Gamma)}.
  \end{align*}
  Together with estimate \eqref{eq:apriori_stages_differentiated} and the $H^1$-estimate for the error, this allows us to bound the normal trace.
\end{proof}

\begin{remark}
  The assumption $\abs{R(\infty)}\leq 1$ is satisfied by all $L$-stable methods, including the family of Radau-IIA methods, since they satisfy $R(\infty)=0$.  
\eremk
\end{remark}

\subsection{A refined $L^2$ estimate}
In Theorem~\ref{thm:full_error_est}, the convergence rate in space with respect to the $L^2$ norm is the same as the one for the $H^1$ norm. Under
some additional conditions on $\Omega$, this can be improved using the usual ``Aubin-Nitsche trick''.
\begin{lemma}
\label{lemma:ritz_projection_l2}
  Assume that $\Omega$ is convex or has a smooth boundary (so that a shift theorem holds for the homogeneous Dirichlet problem) and that
  $\mathcal{V}$ is sufficiently smooth.
  Let $\underline{u}=:(u,u^*) \in \pairhone$ with  $\norm{\laplace u^*}_{L^2(\R^d \setminus \Gamma)}<\infty$ and $\gamma^- u=-\traceJump{u^*}$,
  as well as $\gamma^- u^* = 0$.
  Then the following error estimate holds for the Ritz projector $\Pi_h$:
  \begin{align*}
    \norm{\underline{u}- \Pi_h \underline{u}}_{\pairltwo}
    &\leq C h \left(\norm{\underline{u}-\Pi_h \underline{u}}_{H^1} + \inf_{y_h \in Y_h} \norm{\normalJump{u^*} - y_h}_{H^{-1/2}(\Gamma)}\right)
  \end{align*}
\end{lemma}
  \begin{proof}
    We write $\Pi_h \underline{u}=:(u_h,u_h^*)$ for the two components. Consider the solutions $\psi_1, \psi_2$ to the following two problems:
    \begin{align*}
      -\laplace \psi_1 + \left(V+ \alpha\right) \psi_1&= \begin{cases}
        u - u_h & \text{ in $\Omega$,} \\
        u^* - u_h^* & \text{ in $\Omega^+$,}
        \end{cases}\\
        \traceJump{\psi_1}&=\normalJump{\psi_1}=0, \\
        -\laplace \psi_2 + \left(V_0+ \alpha\right) \psi_2&= u^* - u_h^* \quad\quad \text { in $\Omega$,}  \\
 \gamma^- \psi_2&=0.
    \end{align*}
    Since $\psi_1$ is the solution to a full space elliptic problem, we can estimate $\norm{\psi_1}_{H^2(\R^d)}\leq C \norm{\underline{u} - \Pi_h \underline{u}}_{\pairltwo}$.
    The same estimate holds for $\psi_2$, as we assumed that a shift theorem holds for  $\Omega$,
    i.e, $\norm{\psi_2}_{H^2(\Omega)}\leq C \norm{\underline{u} - \Pi_h \underline{u}}_{\pairltwo}$.
    We rearrange the terms into
    \begin{align*}
      \psi&:=\psi_1|_{\Omega}, \\
      \psi^*&:=\begin{cases}
        \psi_2 & \text{ in } \Omega, \\
        \psi_1 & \text{ in } \Omega^+ 
      \end{cases}
    \end{align*}
    and write $\underline{\psi}:=(\psi,\psi^*)$.
    Integration by parts then gives:
    \begin{align*}
      \norm{u-u_h}_{L^2(\Omega)}^2 + \norm{u^* - u_h^*}_{L^2(\R^d)}^2
      &= \ltwoprodint{-\laplace \psi + \left(V(x)+\alpha \right) \psi}{u-u_h} + \ltwoprodfull{-\laplace \psi^* + \left(V_0 + \alpha \right) \psi^*}{u^* -u_h^*} \\
      &= \blfHStab{\underline{\psi}}{\underline{u}-\Pi_h \underline{u}} - \bdryprod{ \partial_n^- \psi }{\gamma^- (u-u_h)}  \\
      & - \bdryprod{\partial_n^- \psi^*}{\gamma^- (u^*-u_h^*)} + \bdryprod{\partial_n^+ \psi^*}{\gamma^+ (u^* -u_h^*)} \\
      &= \blfHStab{\underline{\psi}}{\underline{u}- \Pi_h \underline{u}}
      - \bdryprod{ \partial_n^- \psi }{\gamma^- (u-u_h)} - \bdryprod{\partial_n^+ \psi^*}{\traceJump{(u^* -u_h^*)}} \\
      &-\bdryprod{\normalJump{\psi^*}}{\gamma^-(u^* -u_h^*)}.
    \end{align*}
    Since $\partial^-_n \psi=\partial^-_n \psi_1=\partial^+_n \psi_1=\partial^+_n \psi^*$ and $\gamma^-\left(u-u_h\right)=-\traceJump{(u^*-u_h^*)}$, this becomes:
    \begin{align*}
      \norm{u-u_h}_{L^2(\Omega)}^2 + \norm{u^* - u_h^*}_{L^2(\R^d)}^2
      &= \blfHStab{\underline{\psi}}{\underline{u}-\Pi_h \underline{u}}  - \bdryprod{\normalJump{\psi^*}}{\gamma^-(u^* -u_h^*)}.
    \end{align*}
    For $\underline{\psi_h}:=(\psi_h,\psi^*_h) \in H_h$ and $\lambda_h$, $\mu_h \in Y_h$ we can use the definition of the Ritz projection $\Pi_h \underline{u}$,
    the fact that $\gamma^- \psi_h$ and $\gamma^- (u-u_h) \in Y_h^{\circ}$ and $\gamma^- \psi^*=0$, to get:
    \begin{align}
      \norm{u-u_h}_{L^2(\Omega)}^2 &+ \norm{u^* - u_h^*}_{L^2(\R^d)}^2 \\
      &=\blfHStab{\underline{\psi}-\underline{\psi_h}}{\underline{u}-\Pi_h \underline{u}} \nonumber 
      + \bdryprod{\normalJump{u} - \lambda_h}{\gamma^- (\psi^*- \psi_h^*)}  -\bdryprod{\normalJump{\psi^*} - \mu_h}{\gamma^-(u^* -u_h^*)} \nonumber \\
      &\lesssim \left( \norm{\underline{\psi} - \underline{\psi_h}}_{\pairhone} + \norm{\normalJump{\psi^*} - \mu_h}_{H^{-1/2}(\Gamma)} \right)
      \left(\norm{\normalJump{u} - \lambda_h}_{H^{-1/2}(\Gamma)} + \norm{ \underline{u} - \Pi_h \underline{u}}_{\pairhone} \right).
      \label{eq:ltwo_estimate_final_proof}
    \end{align}
    The best approximation property of $H_h$, given in Lemma~\ref{lemma:approximation_ritz_schroed}, together with the approximation property of $X_h$ and $Y_h$ from
    \eqref{eq:approx_spaces_assumption} then give:
    \begin{align*}
      \inf_{\underline{\psi_h} \in H_h}\norm{\underline{\psi} - \underline{\psi_h}}_{\pairhone} + \inf_{\mu_h \in Y_h} \norm{\normalJump{\psi^*} - \mu_h}_{H^{-1/2}(\Gamma)}
      &\lesssim h \left(\norm{\psi}_{H^2(\Omega)} + \norm{\normalJump{\psi^*}}_{H^{1/2}(\Gamma)}\right)
        \lesssim h \norm{\underline{u} - \Pi_h \underline{u}}_{\pairltwo},
    \end{align*}
    where in the last step we used the regularity of $(\psi,\psi^*)$. Combining this estimate with \eqref{eq:ltwo_estimate_final_proof} then completes the proof.
  \end{proof}  

\begin{remark}
  It can be shown that the Ritz projector is equivalent to the Galerkin projection for the symmetric coupling of the
  problem $-\laplace u + (V+\alpha)u = f $, where $u^*$ is computed via the representation formula.
  Thus Lemma~\ref{lemma:ritz_projection_l2} also gives a result about the $L^2$ convergence of such a post-processing step for
  the FEM-BEM coupling of stationary elliptic problems. \eremk
\end{remark}

Analogous to Theorem~\ref{thm:full_error_est}, we get the following stronger convergence result in the $L^2$ norm:
\begin{theorem}
\label{thm:full_error_est_l2}
  Assume that the assumptions of Theorem~\ref{thm:full_error_est} are satisfied.
  Additionally, assume that $\Omega$ is convex or has a smooth boundary.
  Then there exists a constant $C>0$, depending on $\Omega$, the Runge-Kutta method (i.e., $A$ and $b$), $\mathcal{V}$,
  $p_0$, $p_1$ and $\sdu{0}$, but not on $k$, $n$, $h$ or $T$ such that:
  \begin{align*}
    \ltwonormint{\fdu{n}-u(n k)}\leq C T \left( h_{1}^{p_1+1} +   h_{0}^{p_0+5/2} +  k^q \right)
  \end{align*}
\end{theorem}
\begin{proof}
  The proof follows along the same lines as the one for Theorem~\ref{thm:full_error_est}, but using the stronger approximation result for $\Pi_h$ 
  given by Lemma~\ref{lemma:ritz_projection_l2}.
\end{proof}

\section{Numerical results}
\label{sec:numerics}
\subsection{Implementation}
We implemented the fully discrete scheme described in this paper, using the software package NGSolve (see \cite{ngsolve}) for the finite element discretization
and Bem++ (see \cite{bempp_preprint}) for the boundary integral operators. To compute the convolution quadrature contributions, we used the FFT-based method introduced
by Banjai in \cite{banjai_cq_algs}, which  avoids the explicit computation of the convolution weights, as defined by \eqref{eq:def_cq_op_calc_weights},
and instead is based on approximating them via numerical quadrature.

Let $\partial B_\lambda(0)$ denote the circle of radius $\lambda>0$ centered at $0$. 
By the Cauchy integral formula we can write for the different operators
\begin{align}
  \label{psi_cauchy_int_disc}
  A^n&:= \frac{1}{2\pi \ii} \int_{\partial B_{\lambda}(0)}{A(z) \; z^{-n-1} \;dz} , 
\end{align}
where $A$ may stand for $V$, $K$, $K^T$ or $W$.

In order to get an approximation that can actually be computed, we discretize the integrals above via a $Q$-point trapezoidal rule:
\begin{align}
  \label{trapezoidal_rule_gen}
  A^n \approx \widetilde{A}^n:=\frac{\lambda^{-n}}{Q+1} \sum_{l=0}^{Q}{A\left(\lambda \zeta_{Q+1}^{-l}\right) \zeta_{Q+1}^{ln}}, 
\end{align}
where $\zeta_{Q+1}:=e^{ \frac{2\pi \ii}{Q+1}}$. In the theory about convolution quadrature, it is well-known 
that choosing $\lambda \approx \operatorname{eps}^{\frac{1}{2 (Q+1)}}$, where $\operatorname{eps}$ denotes machine precision,
leads to good approximation results (this was already suggested in \cite{lubich_cq_and_op_calc2}).
In \cite[Remark 5.11]{banjai_sauter_rapid_solution_wave_eqn} it was observed that, when considering an additional perturbation of the operators $A\left(\lambda \zeta_{Q+1}^{-l}\right)$,
for example due to $\mathcal{H}$-matrix approximation, it is recommended to choose $\lambda \sim k^{\frac{3}{Q+1}}$.
In our experiments, we therefore used $\lambda:=\max\left(\operatorname{eps}^{\frac{1}{2 (Q+1)}},k^{\frac{3}{Q+1}}\right)$.
Our analysis did not account for quadrature errors, but we observed that choosing $Q \geq n$ gives good results.
In order to evaluate the matrix functions $V(B(z))$ etc. we diagonalize the matrix $\delta(z)$ instead of
computing the contour integral in Definition~\ref{def:riesz_dunford}. This is justified for
Radau IIA methods of 2 stages in \cite[Proposition 3.4]{banjai_cq_algs}, and we did not observe any problems for any of the other methods tested.
If we write $M$ and $S$ for the mass and stiffness matrix of the finite element approximation, the appearing block systems have the structure
\begin{align*}
  \begin{pmatrix}
    -\ii A^{-1} M + k S+ k W(0) & k (1/2 - K^T(0)) \\
    -1/2 +  K(0) & V(0)
  \end{pmatrix}
\end{align*}
and were solved the linear systems using a preconditioned GMRES method. The preconditioner used has diagonal block structure, i.e.,
\begin{align*}
  P^{-1}:=\begin{pmatrix}
    P_{FEM}^{-1} & 0 \\
    0 & P_{BEM}^{-1}
  \end{pmatrix},
\end{align*}
where the preconditioner $P_{BEM}$ makes use of the fact that $V(0)$ is already assembled in diagonalized form by using an $\mathcal{H}$-matrix LU-factorization
for each operator $V(\lambda_j)$, where the $\lambda_j$ are the eigenvalues of $B(0)$. The FEM preconditioner is again block-diagonal itself and defined as
\begin{align*}
  P_{FEM}^{-1}:=\begin{pmatrix}
    P_{MG}^{-1}(A_{11}) & 0 & \dots & 0 \\
    0 & P_{MG}^{-1}(A_{22}) & 0      & \vdots \\
    \vdots & 0& \ddots  & 0\\
    0 & 0& &P_{MG}^{-1}(A_{mm})
  \end{pmatrix},
\end{align*}
where $P_{MG}(\lambda)$ is a standard multigrid preconditioner, based on a block-Jacobi smoother
as is already implemented in NGSolve,
for the FEM-matrix $-\ii \lambda M + k\,S$. We selected this preconditioning strategy because it is 
easily implemented using the preconditioners already available in NGSolve and Bem++. While we do not have 
any theoretical analysis of the preconditioning strategy, it appears to work well for our model problem, 
taking for example only 56 steps to reduce the residual by a factor $10^{-11}$, 
in the case of a 2 stage Radau IIA method and degree $(3,2)$ FEM-BEM spaces, where the FEM space consisted of $912,673$ degrees of freedom.

\subsection{Gaussian beams and the free Schrödinger equation}
In this section we look at numerical results for the free Schrödinger equation, $\mathcal{V}=0$ in $3D$.
That is, we consider the model problem:
\begin{align}
\label{eq:schroedinger_model_problem}
  \begin{cases}
    \ii u_t(x,t)=-\laplace u, & x \in \R^3, \\
    u(x,0)=\sdu{0}. &
  \end{cases}
\end{align}

Given a point $x_c \in \R^3$ and a wave vector $p_0 \in \R^3$, we consider the Gaussian beam
\begin{align*}
  \sdu{0}(x):=\sqrt[4]{\frac{2}{\pi}} e^{-\abs{x-x_c}^2 + \ii p_0 \cdot (x-x_c)}.
\end{align*}
For this initial condition, the exact solution is given by
\begin{align*}
  u_{ex}(x,t)=\sqrt[4]{\frac{2}{\pi}} \sqrt{\frac{\ii}{-4 t + \ii}} \exp\left(\frac{-\ii \abs{x-x_c}^2 - p_0 \cdot (x-x_c) + \abs{p_0}^2t}{-4t + \ii}\right).
\end{align*}

\begin{figure}[h!]
  \includegraphics[width=8cm]{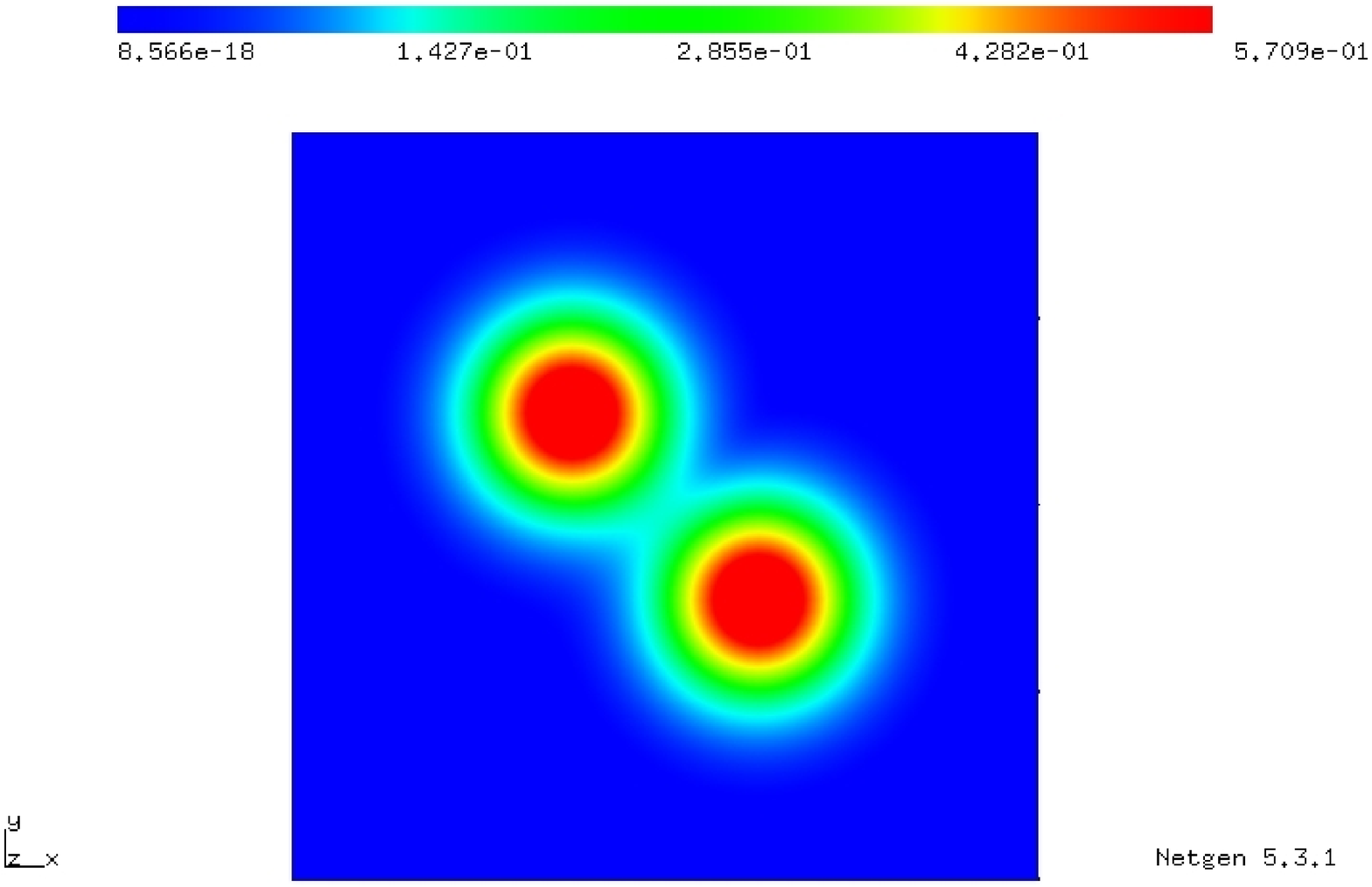}
  \includegraphics[width=8cm]{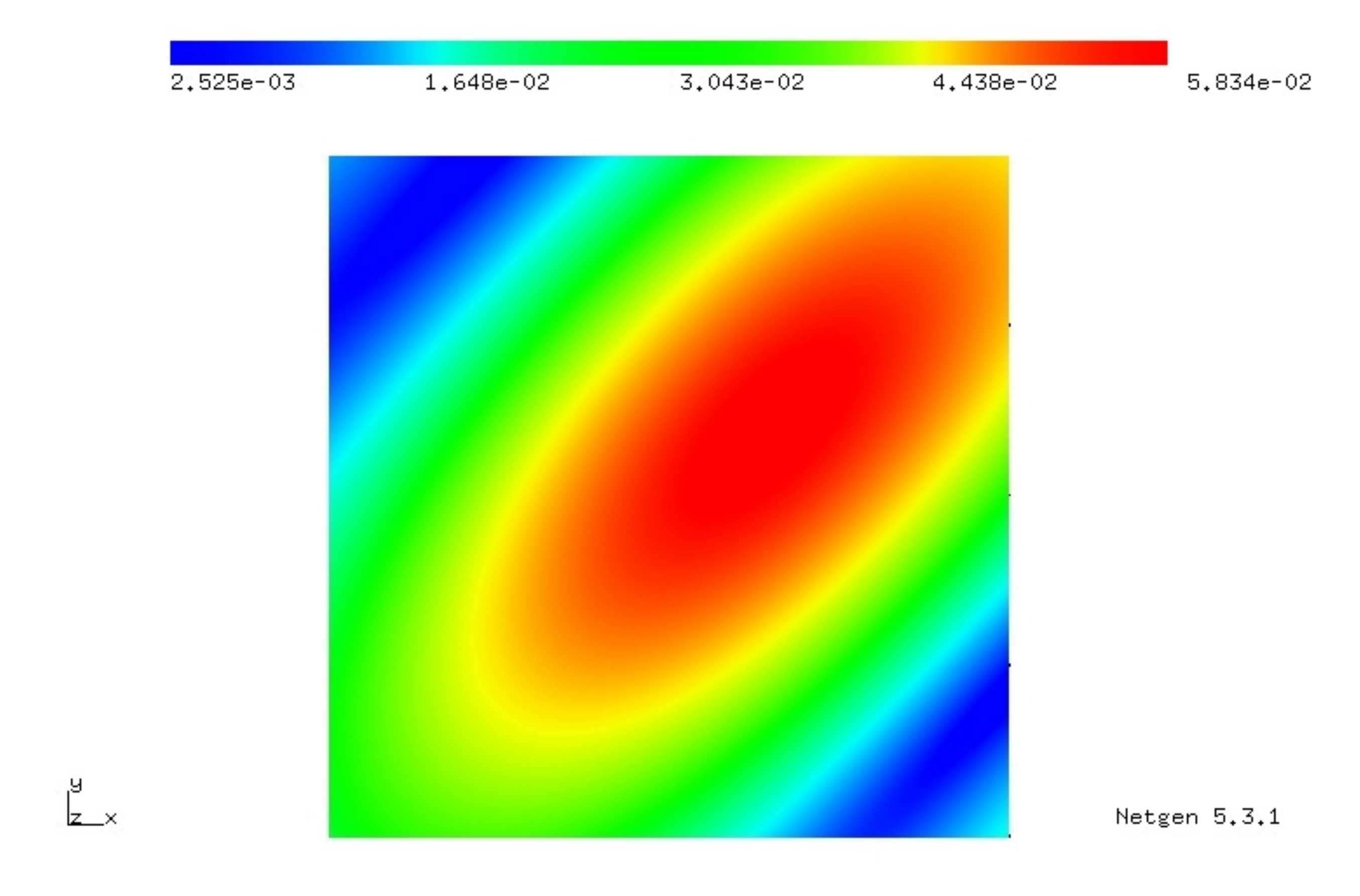}
\caption{Modulus of  exact solution of \eqref{eq:schroedinger_model_problem} at $t=0$(left) and $t=2$ (right) for $z=0$}
\label{fig:exact_solution}
\end{figure}

As a computational domain, we chose a cube with side length $8$ centered at the origin.
For our numerical experiments, we chose a combination of two Gaussian beams as initial condition, $u^0_1$ and $u^0_2$.
$u^0_1$ is centered at $(-1,1,0)$ and has a wave number $(1,0,0)$. This makes the exact solution a Gaussian wave packet, traveling out of the domain $\Omega$.
We center $u^0_2$ at $(1,-1,0)$ with wave number $(0,0,0)$, which means that we will mostly see a dispersive effect. This second term was added, to better distinguish
between convergence and artificial damping introduced by the method. This choice of initial condition does not satisfy the condition 
  $\operatorname{supp}{\sdu{0}}\subseteq \Omega$,  but due to the fast decay rate, the error due to truncating outside of $\Omega$ becomes negligible.
Figure~\ref{fig:exact_solution} shows the exact solution for $t=0$ and  $t=2$.

\begin{example}
  In this example, we look at the convergence rates for the one stage Gauss method and the 2 and 3 stage Radau IIA methods.
  We  chose the mesh  and time step size to be proportional, i.e., $k \sim h$ by performing a uniform refinement of the mesh, every time we halved the time step size.
  In light of Theorem~\ref{thm:full_error_est}, we
  expect convergence of order $2$, $3$, and $5$ respectively, as long as we couple with Finite Elements of the same order, and boundary elements of order $p_0=p-1$.
  We compare the maximum of the $L^2$ and $H^1$ error, taken between $t=0$ and $t=2$  in the FEM term,
    i.e., $\max_{n=0,\dots, N}{\norm{\fdu{n}-u(t_n)}_{L^2(\Omega)}}$ and  $\max_{n=0,\dots, N}{\norm{\fdu{n}-u(t_n)}_{H^1(\Omega)}}$.
  In order better to compare the two methods, we plot $m n$ in the $x$-axis, where $m$ is the number of stages. This reflects the fact that for the higher order
  method, we need to assemble $m$-times the number of boundary operators.
  We see that the 1 stage Gauss and the 2 stage Radau IIA  methods converge with the predicted full rates of 2 and 3 respectively. For the higher order
  Radau method, we do not see the predicted rate, most likely due to a preasymptotic behavior, but comparing the number of operators 
  to the achieved accuracy, we see that the higher order methods prove more efficient.
  \begin{figure}
    \centering
      \begin{subfigure}[b]{0.45\textwidth}
        \includeTikzOrEPS{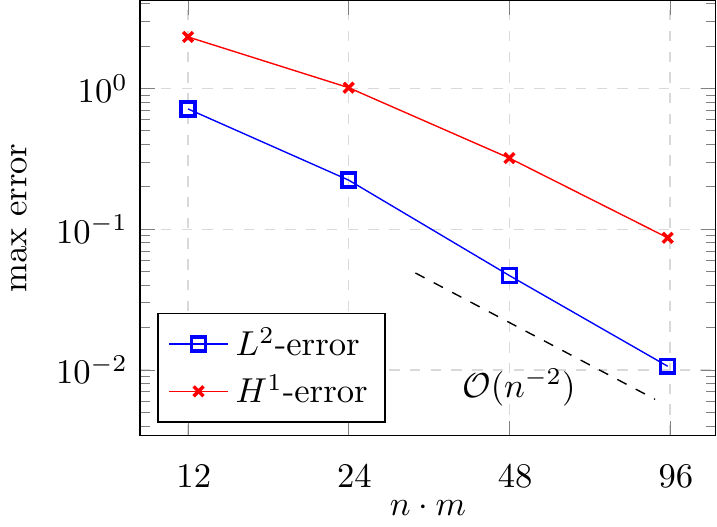}
        \caption{1 stage Gauss method  (order 2)}
      \end{subfigure}
      \begin{subfigure}[b]{0.45\textwidth}
        \includeTikzOrEPS{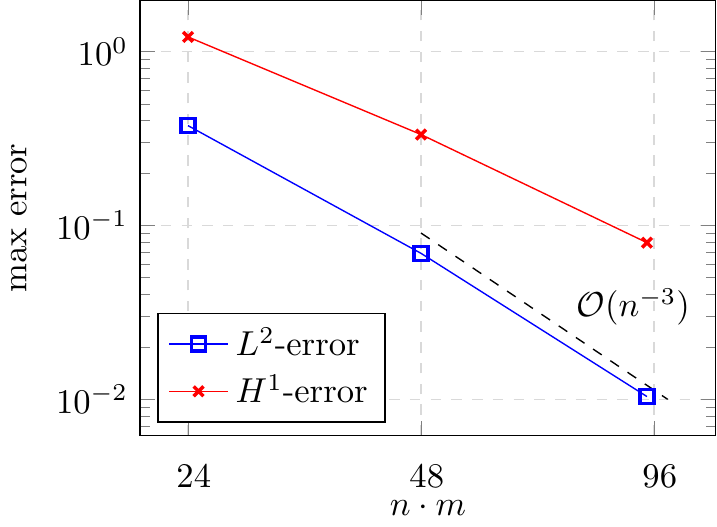}
        \caption{2 stage Radau IIA method (order 3)} 
      \end{subfigure}
      \begin{subfigure}[b]{0.45\textwidth}
        \vspace{5mm}
        \includeTikzOrEPS{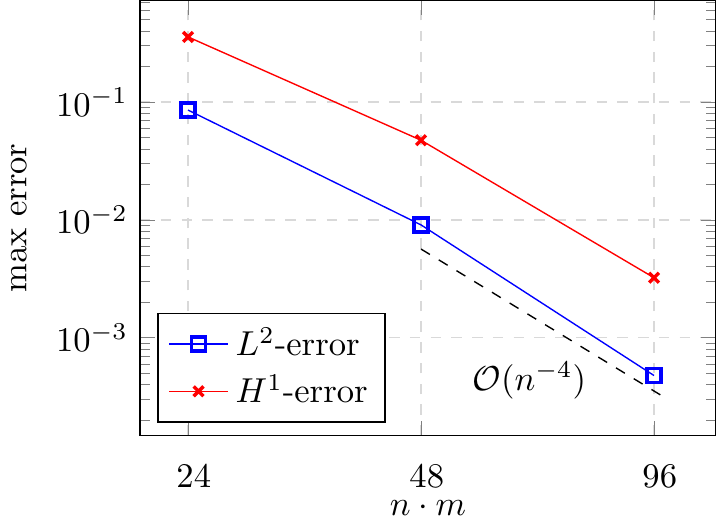}
        \caption{3 stage Radau method (order 5)}
      \end{subfigure}
    \caption{Comparison of a 1 stage Gauss method and 2 and 3 stage Radau IIA methods}
    \label{fig:conv_gauss_radau}
  \end{figure}
\eremk
\end{example}

\clearpage
\newpage
\appendix{\bf Boundary element methods for vector valued problems}
\label{appendix_bem_systems}

In this section we generalize some well-known results about boundary element methods for the Helmholtz equation to the
case of vector valued problems, where the ``wave number'' is replaced by a matrix. We start by recalling the
scalar case with the following proposition:
\begin{proposition}[Representation formula]
   Let  $u \in H^1( \R^d \setminus \Gamma)$ with $(\Delta - s^2)u \in L^2(\R^d \setminus \Gamma)$. 
  Then for $\Re(s)>0$ we can write $u$ as:
  \begin{align}
\label{eq:general-representation-formula}
    u=-N(s)\left((\laplace - s^2 ) u \right) + S(s) \normalJump{u} - D(s) \traceJump{u} \quad \quad 
\mbox{ on $\R^d \setminus \Gamma$}.
  \end{align}
  For solutions to the Helmholtz equation, i.e., $\left(\laplace - s^2 \right) u =0$  this becomes
  \begin{align*}
    u= S(s) \normalJump{u} - D(s) \traceJump{u} \quad \quad \mbox{ on $\R^d \setminus \Gamma$}.
  \end{align*}
\end{proposition}
\begin{proof} 
 Formula (\ref{eq:general-representation-formula}) is shown as follows: 
For large balls $B_R(0) \subset \R^d$, (\ref{eq:general-representation-formula}) is obtained by integration by parts with the additional term 
$\int_{\partial B_R(0)} \Phi(x,y; s) \partial_{n(y)} u(y)d\Gamma(y) -\int_{\partial B_R(0)} \partial_{n(y)} \Phi(x,y; s) u(y) d\Gamma(y)$. 
The assumption $\Re(s) >0$ implies that, for fixed $x$, the function  $\Phi(x,\cdot;s)$ (and its derivatives) exponentially decays
as $|y| \rightarrow \infty$. The assumption $u \in H^1(\R^d \setminus\Gamma)$ then allows one to show that the additional term vanishes 
in the limit $R \rightarrow \infty$. 
\end{proof}
\begin{lemma}[Representation Formula, matrix version]
  Let $B$ be a matrix with $\sigma(B) \subseteq \C_+:=\{ z \in \C: \Re{z} > 0 \}$, 
and let $Y\in \prodSpace{H^1}(\R^d\setminus \Gamma)$ be a solution to the differential equation
  \begin{align}
    \label{frml_helmholtz_matrix_version}
    -\prodOp{\laplace} \vect{Y} + B^2 \vect{Y} &= 0, \quad \quad \text{in } \R^d \setminus \Gamma.
  \end{align}
  
  Then we can write $\vect{Y}$ as
  \begin{align*}
    \vect{Y} &= S(B) \normalJump{\vect{Y}} - D(B) \traceJump{\vect{Y}}.
  \end{align*}
  \begin{proof}
    We start with the right-hand side. Inserting the definitions, we get for the $j$-th unit vector $e_j$
    and an integration path $\mathcal{C} \subset \C_+$ encircling $\sigma(B)$:
    \begin{align}
      e_j^T \left(S(B) \normalJump{\vect{Y}} - D(B) \traceJump{\vect{Y}} \right)  \nonumber
      &= \frac{1}{2\pi \ii} e_j^T \int_{\mathcal{C}}{  (B - \lambda)^{-1} \otimes S(\lambda) \normalJump{\vect{Y}} 
        -  (B - \lambda)^{-1} \otimes D(\lambda) \traceJump{\vect{Y}} \,d\lambda} \nonumber\\
      &= \frac{1}{2\pi \ii}  \int_{\mathcal{C}}{ S(\lambda) e_j^T(B - \lambda)^{-1} \normalJump{\vect{Y}} - D(\lambda) e_j^T(B - \lambda)^{-1} \traceJump{\vect{Y}} \,d\lambda}.
      \label{int_rep_formula_proof1}
    \end{align}
    If we apply the scalar representation formula for the function $e_j^T(B - \lambda)^{-1} \vect{Y}$ we get:
    \begin{align}
      (\ref{int_rep_formula_proof1})
       &= \frac{1}{2\pi \ii}  \int_{\mathcal{C}}{ e_j^T(B - \lambda)^{-1} \vect{Y} 
         + N(\lambda) \left(\laplace - \lambda^2 \right)  \left(e_j^T(B - \lambda)^{-1} \vect{Y} \right) \, d\lambda}  \nonumber\\
       &=  e_j^T \vect{Y} +  \frac{1}{2\pi \ii}\int_{\mathcal{C}}{  N(\lambda) \left(\laplace - \lambda^2 \right)  \left( e_j^T(B - \lambda)^{-1} \vect{Y}\right) \, d\lambda}. 
       \label{int_rep_formula_proof2}
     \end{align}
     Thus it remains to show that the last term vanishes.
     For $\lambda \in \C\setminus \sigma(B)$ we calculate
     \begin{align*}
       \left( \prodOp{\laplace} -\lambda^2 \right) \left(B-\lambda \right)^{-1} \vect{Y} 
       &=\left(B-\lambda \right)^{-1} \left( \prodOp{\laplace} \vect{Y} - B^2 \vect{Y} \right) 
       +\left(B-\lambda \right)^{-1} \left(B^2 - \lambda^2\right) \vect{Y} \\
       &= 0  
       +\left(B-\lambda \right)^{-1} \left(B - \lambda\right)\left(B + \lambda \right) \vect{Y} \\
       &= \left(B + \lambda\right) \vect{Y}.
     \end{align*}
     The integral in (\ref{int_rep_formula_proof2}) becomes
     \begin{align*}
       \int_{\mathcal{C}}{ N(\lambda) e_j^T \left(B + \lambda\right) \vect{Y} \, d\lambda }.
     \end{align*}
     Since the integrand is holomorphic on $\C^+$ and $\mathcal{C}$ is a closed path this integral vanishes.
  \end{proof}
\end{lemma}

In this paper we often need to solve systems of equations of special structure arising from the Runge-Kutta method. The following lemma
gives a condition for unique solvability and some stability estimates that are used throughout the paper.
\begin{lemma}
  \label{lemma_solve_matrix_eqn}
  Let $B \in \C^{m \times m}$. Let $V$, $H$ be Hilbert spaces with continuous embedding
  $V \subseteq H$. 
  Let $a(\cdot,\cdot) : V \times V \mapsto \C$ be a continuous sesquilinear form.
  Assume the variational problem of finding $u \in V$ such that
  \begin{align*}
    a(u,v) + \innerprod{H}{\lambda u}{v} &= \dualprod{V}{f}{v} & \forall v \in V
  \end{align*}
  has a unique solution for all $\lambda \in \sigma(B)$ and for all right-hand sides $f \in V'$. Then the following
is true: 
\begin{enumerate}[(i)]
\item 
  \label{item:lemma_solve_matrix_eqn-i}
  There exists a unique solution $\vect{u} \in \prodSpace{H}$ to the vector valued problem
  \begin{align}
\label{eq:lemma_solve_matrix_eqn-10}
    \underline{a}\left(\vect{u},\vect{v} \right) + \innerprod{\prodSpace{H}}{B \vect{u}}{\vect{v}} &= \dualprod{\prodSpace{V}}{f}{v}  & \forall v \in \prodSpace{V},
  \end{align}
  where $\underline{a}(\cdot,\cdot)$ denotes the sum sesquilinear form
  \begin{align*}
    \underline{a}\left(\vect{u},\vect{v} \right):=\sum_{j=1}^{m}{a(\vect{u}_j,\vect{v}_j)}.
  \end{align*}
\item 
  \label{item:lemma_solve_matrix_eqn-ii}
  Assume $0 \not\in \Im(\sigma(B))$.
Let $f \in H'$. Then the solution can be estimated in the $H$ norm by
  \begin{align}
    \label{eq:solve_matrix_eq_H}
    \norm{\vect{u}}_{\prodSpace{H}} &\leq C \norm{\vect{f}}_{\prodSpace{H'}}, 
  \end{align}
  where $C > 0$ depends on $B$ but is independent of $a(\cdot,\cdot)$.
\item  
  \label{item:lemma_solve_matrix_eqn-iii}
  Let $a(\cdot,\cdot)$ be Hermitian and positive semidefinite (i.e., $a(u,u)$ induces a seminorm on $V$). 
  Assume $0 \not \in \Im(\sigma(B))$. 
  Consider the family of sequilinear forms given by $a_{\varepsilon}(\cdot,\cdot):=\varepsilon\; a(\cdot,\cdot)$ 
  for a small parameter $\varepsilon >0$, and let $\vect{u_{\varepsilon}}$ be the solution when $a$ is replaced
  with $a_{\varepsilon}$ in (\ref{eq:lemma_solve_matrix_eqn-10}). 
  Then there exists a constant $C>0$ depending on $B$ but independent of $\varepsilon$ such that for all 
  right-hand sides $f \in H'$ the following estimate holds:
  \begin{align}
    \label{eq:solve_matrix_eq_eps_dependent}
    \varepsilon \underline{a}\left(\vect{u_{\varepsilon}},\vect{u_{\varepsilon}}\right) + \norm{\vect{u_{\varepsilon}}}^2_{\prodSpace{H}} 
    \leq C \norm{\vect{f}}^2_{\prodSpace{H'}}.
  \end{align}
\item 
  \label{item:lemma_solve_matrix_eqn-iv}
  If we identify the functional $f \in H'$ in (\ref{item:lemma_solve_matrix_eqn-iii}) with its Riesz representation, 
   i.e., $\innerprod{\prodSpace{H}}{f}{v}=f(v) \; \forall v \in \prodSpace{H}$
  and make the regularity assumption that $f \in V$, then we can further estimate:
  \begin{align}
    \label{eq:solve_matrix_eq_no_eps}
    a(\vect{u_\varepsilon},\vect{u_\varepsilon}) + \norm{\vect{u_\varepsilon}}^2_{\prodSpace{H}} \leq C \norm{\vect{f}}^2_{\prodSpace{V}}.
  \end{align}
  Again the constant $C$ depends on $B$ but is independent of $\varepsilon$.
\end{enumerate}
  \begin{proof}
    We transform the matrix $B$ to Jordan form: $B=XJX^{-1}$. Then the problem transforms to
    \begin{align*}
      \underline{a}( X^{-1} \vect{u},X^T\vect{v}) + \innerprod{\prodSpace{H}}{J X^{-1} \vect{u}}{X^T\vect{v}}
      &= \dualprod{\prodSpace{V}}{X^{-1}\vect{f}}{ X^{T}v}  & \forall v \in \prodSpace{V}.
    \end{align*}
    By setting $\tilde{\vect{u}} := X^{-1} \vect{u}$ and $\tilde{\vect{v}}:= X^T \vect{v}$ and $\tilde{\vect{f}}:=X^{-1} \vect{f}$, 
    the problem above has a unique solution if and only if 
    \begin{align}
\label{eq:appendix-100}
      \underline{a}\left( \tilde{\vect{u}},\tilde{\vect{v}}\right) + \innerprod{\prodSpace{H}}{J \tilde{\vect{u}}}{\tilde{\vect{v}}}
      &= \dualprod{\prodSpace{V}}{\tilde{\vect{f}}}{ \tilde{\vect{v}}}  & \forall \tilde{\vect{v}} \in \prodSpace{V},
    \end{align}   
    has a unique solution. 
    To simplify the notation we only consider the case that $J$ only consists of a single Jordan block. The proof of the general case works along the same lines.
    Selecting test functions $\tilde{v}=(0,\dots,v_j,\dots 0)$ for all $j=1,\dots,m$ with $v_j \in V$ shows that 
    equation (\ref{eq:appendix-100}) is equivalent to the system of scalar problems
    \begin{align}
\label{eq:appendix-200}
      a\left( \tilde{\vect{u}}_j,\vect{v}_j \right) + \innerprod{H}{\lambda \tilde{\vect{u}}_j + \tilde{\vect{u}}_{j+1}}{v_j}
      &= \dualprod{V}{\tilde{\vect{f}}_j}{ v_j}  & \forall v_j \in V,\;  j=1, \dots, m-1,
    \end{align}  
    where $\lambda$ is the eigenvalue of the Jordan block. For the case $j=m$ a similar equation holds: 
    \begin{align}
\label{eq:appendix-300}
      a\left(\tilde{\vect{u}}_m,v_m \right) + \innerprod{H}{\lambda \tilde{\vect{u}}_m}{v_m}
      &= \dualprod{V}{\tilde{\vect{f}}_m}{ v_m}  & \forall v_m \in V.
    \end{align}  

    By our assumption this last problem has a solution $\tilde{u}_m \in V$. This enables us to solve the 
    $(m-1)$-st equation and by induction we get the solution $\vect{\tilde{u}}$.
    Since each solution of the scalar problems is unique this also makes the vector valued solution unique.
    Hence, (\ref{item:lemma_solve_matrix_eqn-i}) is shown. 
    
    To get the estimates \eqref{eq:solve_matrix_eq_H} and \eqref{eq:solve_matrix_eq_eps_dependent} 
    of (\ref{item:lemma_solve_matrix_eqn-ii}) and (\ref{item:lemma_solve_matrix_eqn-iii}), 
    we set $v_j:=\tilde{\vect{u}}_j$ and recall the definition $a_\varepsilon=\varepsilon a$. 
    This gives for (\ref{eq:appendix-300}): 
    \begin{align*}
      \varepsilon a\left(\tilde{\vect{u}}_m,\tilde{\vect{u}}_m\right) + \innerprod{H}{ \lambda \tilde{\vect{u}}_m}{\tilde{\vect{u}}_m}
      &= \dualprod{V}{\tilde{\vect{f}}_m}{ \tilde{\vect{u}}_m}.
    \end{align*}
    Separating real and imaginary parts gives
    \begin{align*}
      \varepsilon a \left(\tilde{\vect{u}}_m,\tilde{\vect{u}}_m \right)  + \Re(\lambda) \innerprod{H}{ \tilde{\vect{u}}_m}{\tilde{\vect{u}}_m}
      &= \Re\dualprod{V}{\tilde{\vect{f}}_m}{ \tilde{\vect{u}}_m}, \\
      \Im(\lambda) \innerprod{H}{ \tilde{\vect{u}}_m}{\tilde{\vect{u}}_m}
      &= \Im\dualprod{V}{\tilde{\vect{f}}_m}{ \tilde{\vect{u}}_m} .
    \end{align*}
    Since by assumption $\Im(\lambda) \ne 0$ and $f \in H'$ we easily get from these two equations the estimates 
    \eqref{eq:solve_matrix_eq_H} and \eqref{eq:solve_matrix_eq_eps_dependent} for $\tilde{\vect{u}}_m$.
    By doing similar calculations for (\ref{eq:appendix-200}) for $j = m-1,\ldots,1$, we get the desired estimates 
    by induction.
   
    We turn to the proof of (\ref{item:lemma_solve_matrix_eqn-iv}). 
    In order to refine our estimates for the smooth case $f \in V$, i.e., to show \eqref{eq:solve_matrix_eq_no_eps},
    we proceed similarly. By the previous result we only need to show that we can bound $a(u,u)$.
    We choose $v_j:=\lambda \tilde{\vect{u}}_j-\tilde{\vect{f}}_j$ in (\ref{eq:appendix-300}) and get for 
    the $m$-th component:
    \begin{align*}
      \varepsilon a \left(\tilde{\vect{u}}_m, \lambda \tilde{\vect{u}}_m - \tilde{\vect{f}}_m \right)
      + \innerprod{H}{ \lambda \tilde{\vect{u}}_m}{\lambda \tilde{\vect{u}}_m - \tilde{\vect{f}}_m}
      &= \dualprod{V}{\tilde{\vect{f}}_m}{ \lambda \tilde{\vect{u}}_m - \tilde{\vect{f}}_m}. 
      \end{align*}
      Rearranging terms and taking the imaginary part gives in view 
      of $\dualprod{V}{\tilde{f}_m}{v} = \innerprod{H}{\tilde{f}_m}{v}$ for all $v \in V$ and 
      $\innerprod{H}{v}{v} \in \R$:
      \begin{align*}
        \varepsilon\;\Im \left(  a\left(\tilde{\vect{u}}_m,\lambda \tilde{\vect{u}}_m - \tilde{\vect{f}}_m \right) \right)
        &= \Im\innerprod{H}{\tilde{\vect{f}}_m - \lambda \tilde{\vect{u}}_m }{ \lambda \tilde{\vect{u}}_m - \tilde{\vect{f}}_m} = 0.
      \end{align*}
      Hence, ${\varepsilon \Im(\lambda) a\left( \tilde{\vect{u}}_m,\tilde{\vect{u}}_m \right)= 
        \varepsilon \Im \left( a \left(  \tilde{\vect{u}}_m,\tilde{\vect{f}}_m \right) \right)}$, 
      or, using the Cauchy-Schwarz inequality for $a$:
      \begin{align*}        
        a\left( \tilde{\vect{u}}_m,\tilde{\vect{u}}_m \right) 
        &\lesssim a\left( \tilde{\vect{u}}_m,\tilde{\vect{u}}_m \right)^{1/2} a\left( \tilde{\vect{f}}_m , \tilde{\vect{f}}_m \right)^{1/2} 
        \lesssim a\left( \tilde{\vect{u}}_m,\tilde{\vect{u}}_m\right)^{1/2}  \norm{\tilde{\vect{f}}_m}_{V}.
      \end{align*}
      Induction then again gives the analogous statement for the $\tilde{\vect{u}}_j$, $j=1,\ldots,m-1$. 
      
      In order to transform back, we use the fact that $a$ induces a seminorm on $V$. Since $\vect{u}=X \vect{\tilde{u}}$ we can estimate
      \begin{align*}
        \underline{a}(\vect{u},\vect{u})^{1/2} &= \underline{a}(X\vect{\tilde{u}},X\vect{\tilde{u}})^{1/2}
        \leq \norm{X} \underline{a}\left(\vect{\tilde{u}},\vect{\tilde{u}}\right)^{1/2}
      \end{align*}
      and similarly for the $H$-norm. All the estimates then transfer to the original $u$ by taking linear combinations.
  \end{proof}
\end{lemma}

\clearpage
\textbf{Acknowledgments:} Financial support by the Austrian Science Fund (FWF) through the 
doctoral school ``Dissipation and Dispersion in Nonlinear PDEs'' (project W1245, A.R.).

\bibliographystyle{plain}
\bibliography{schroedinger}
\end{document}